\documentclass{amsart}
\usepackage{amssymb, graphics}

\usepackage[cyr]{aeguill}
\usepackage[french]{babel}
\usepackage{lmodern}
\usepackage{xfrac}
\usepackage{caption} 
\usepackage{amsthm}
\usepackage{amssymb}
\usepackage[margin = 1in]{geometry}
\DeclareMathOperator{\tr}{tr}
\DeclareMathOperator{\vol}{vol}

\DeclareMathOperator{\phg}{phg}
\DeclareMathOperator{\Diff}{Diff}
\DeclareMathOperator{\Ric}{Ric}
\DeclareMathOperator{\End}{End}

\DeclareMathOperator{\Sym}{Sym}

\DeclareMathOperator{\Id}{Id}
\DeclareMathOperator{\Ker}{Ker}
\DeclareMathOperator{\pr}{pr}

\newtheorem{thm}{Theorem}[section]
\newtheorem{prop}[thm]{Proposition}
\newtheorem{cor}[thm]{Corollaire}
\newtheorem{remq}[thm]{Remarque}
\newtheorem{lemme}[thm]{Lemme}
\newtheorem{exemp}[thm]{Exemple}
\newtheorem{defi}[thm]{D\'efinition}
\newtheorem{defis}[thm]{D\'efinitions}

\newtheorem{them}[thm]{Th\'eor\`eme}
\newtheorem{hyp}[thm]{Hypoth\`{e}se}
\newcommand{\fact}[1]{#1\mathpunct{}!}
\renewcommand{\Re}{\operatorname{Re}}
\renewcommand{\Im}{\operatorname{Im}}
\usepackage[all]{xy} 
\begin{document}
\title[]
{Polyhomog\'en\'eit\'e des m\'etriques compatibles avec une structure de Lie \`a l'infini le long du flot de Ricci}
\author{Mahdi Ammar}
\begin{abstract}   
Le long du flot de Ricci, on \'etudie la polyhomog\'en\'eit\'e des m\'etriques pour des vari\'et\'es riemanniennes non\texttt{-}compactes ayant \texttt{<<} une structure de Lie fibr\'ee \`a l'infini  \texttt{>>}, c'est\texttt{-}\`a\texttt{-}dire une classe de structures de Lie \`a l'infini qui induit dans un sens pr\'ecis des structures de fibr\'es sur les bords d'une certaine compactification par une vari\'et\'e \`a coins. Lorsque cette compactification est une vari\'et\'e \`a bord, cette classe de m\'etriques contient notamment les b\texttt{-}m\'etriques de Melrose, les m\'etriques \`a bord fibr\'e de Mazzeo\texttt{-}Melrose et les m\'etriques edge de Mazzeo. On montre alors que la polyhomog\'en\'eit\'e \`a l'infini des m\'etriques compatibles avec une structure de Lie fibr\'ee \`a l'infini est pr\'eserv\'ee localement par le flot de Ricci\texttt{-}DeTurck. Si la m\'etrique initiale est asymptotiquement Einstein, on obtient la polyhomog\'en\'eit\'e des m\'etriques tant que le flot existe. De plus, si la m\'etrique initiale est \texttt{<<} lisse jusqu'au bord \texttt{>>}, alors il en sera de m\^eme pour les solutions du flot de Ricci normalis\'e et du flot de Ricci\texttt{-}DeTurck. 

\end{abstract}
\maketitle

\tableofcontents

\section*{Introduction} 
L'objet principal de cet article est l'\'etude du comportement \`a l'infini des solutions du flot de Ricci sur des vari\'et\'es non compactes. Le flot de Ricci est en quelque sorte une version non\texttt{-}lin\'eaire de l'\'equation de la chaleur, qui, au lieu d'uniformiser la temp\'erature, tend plut\^ot \`a uniformiser la courbure. Il a \'et\'e introduit en 1982 par Hamilton dans son article~\cite{hamilton1982three} et s'est r\'ev\'el\'e \^etre un outil tr\`{e}s utile dans la compr\'ehension de la topologie des vari\'et\'es. En particulier, il a jou\'e un r\^ole central dans la preuve de Grigory Perelman de la conjecture de g\'eom\'etrisation de Thurston, dont la conjecture de Poincar\'e~\cite{morgan2007ricci} est un cas particulier. Pour des vari\'et\'es compactes, Hamilton a prouv\'e l'existence locale et l'unicit\'e du flot de Ricci. Par suite, DeTurck~\cite{deturck1983deforming} a introduit une astuce \'el\'egante afin de donner une preuve simplifi\'ee de l'existence locale.

Cependant, le comportement du flot pour les vari\'et\'es non compactes est plus d\'elicat et d\'epend de la g\'eom\'etrie \`a l'infini. Plusieurs recherches ont port\'e sur des extensions naturelles du flot de Ricci sur des vari\'et\'es compl\`{e}tes non compactes. L'existence d'une solution au flot a \'et\'e \'etablie par Shi~\cite{shi1989deforming} pour des vari\'et\'es riemanniennes compl\`{e}tes ayant une courbure born\'ee. De plus, la solution existe tant que la courbure reste born\'ee. Par cons\'equent, pour estimer le temps maximal d'existence du flot de Ricci, il suffit de contr\^oler la courbure. Plus tard, Chen et Zhu~\cite{chen2006uniqueness} ont repris le travail de Shi et ont montr\'e l'unicit\'e de la solution. Sous des hypoth\`{e}ses techniques supl\'ementaires, Albert Chau~\cite{chau2004convergence} a prouv\'e l'existence et la convergence d'une solution globale pour certaines vari\'et\'es k\"ahl\'eriennes non compactes. R\'ecemment, il y a eu une intense activit\'e pour comprendre dans quelle mesure le flot de Ricci pr\'eserve d'autres conditions g\'eom\'etriques sur des vari\'et\'es compl\`{e}tes non compactes. Dans les r\'esultats de beaucoup de recherches, les m\'etriques consid\'er\'ees sont non seulement approxim\'ees par le mod\`{e}le g\'eom\'etrique \`a l'infini, mais elles admettent aussi une expansion asymptotique lisse, ou plus g\'en\'eralement une expansion polyhomog\`{e}ne, \`a savoir une expansion constitu\'ee de termes de la forme $\rho^{\alpha} (\log \rho)^{k}$, o\`u $\rho$ est la distance par rapport \`a un point fix\'e, $\alpha$ est un nombre r\'eel (pas n\'ecessairement entier) et $k\in \mathbb{N}_{0}.$ Une question naturelle est donc la suivante :

Quel type de g\'eom\'etrie \`a l'infini sur des vari\'et\'es compl\`{e}tes non compactes sera pr\'eserv\'e par le flot de Ricci ? De plus, lorsque la m\'etrique initiale a une expansion polyhomog\`{e}ne \`a l'infini, en est\texttt{-}il de m\^eme pour les m\'etriques ult\'erieures le long du flot ?

Lott et Zhang~\cite{lott2011ricci} ont \'etudi\'e le flot de K\"ahler\texttt{-}Ricci pour des m\'etriques k\"ahl\'eriennes \`a pointes fibr\'ees (ou de type Poincar\'e) sur des vari\'et\'es quasi\texttt{-}projectives et ont montr\'e que de telles m\'etriques, le long du flot de Ricci, pr\'eservent certains comportements asymptotiques. Ces r\'esultats peuvent \^etre compris comme \'etant le terme d'ordre 0 du d\'eveloppement asymptotique des m\'etriques \`a l'infini. Ult\'erieurement, ces d\'eveloppements asymptotiques ont \'et\'e d\'evelopp\'es dans l'article~\cite{rochon2012asymptotics} de Rochon et Zhang en montrant que la solution admet une expansion asymptotique compl\`{e}te tant que le flot existe. Un autre r\'esultat de Lott et Zhang~\cite{lott2016ricci} consid\`{e}re trois types diff\'erents de comportements asymptotiques spatiaux (cylindrique, bomb\'e et conique) pr\'eserv\'es par le flot de K\"ahler\texttt{-}Ricci. Albin, Aldana et Rochon~\cite{albin2013ricci} ont travaill\'e sur des surfaces munies de m\'etriques compl\`{e}tes approxim\'ees par des pointes (cusps en anglais) ou des entonnoirs \`a l'infini et ont montr\'e que le flot de Ricci converge vers une m\'etrique avec une courbure constante et que le d\'eterminant du laplacien augmente tant que le flot existe. Ils ont notamment eu besoin de montrer que l'expansion asymptotique \`a l'infini de ces m\'etriques est pr\'eserv\'ee le long du flot. Isenberg, Mazzeo et Sesum~\cite{isenberg2013ricci} ont \'etudi\'e le comportement du flot sur des surfaces \`a bouts asymptotiquement euclidiens de caract\'eristique d'Euler n\'egative. Pour les m\'etriques asymptotiquement hyperboliques, la persistance le long du flot de Ricci d'un d\'eveloppement asymptotique lisse a \'et\'e obtenue par Bahuaud~\cite{bahuaud2011ricci}. Rochon~\cite{rochon2015polyhomogeneite} a obtenu un r\'esultat similaire pour les m\'etriques asymptotiquement hyperboliques complexes. Dans la plupart de ces travaux, une compactification a \'et\'e introduite pour obtenir une vari\'et\'e \`a bord compacte, de sorte que le comportement du flot \`a l'infini puisse \^etre d\'ecrit en termes du bord. Dans ce contexte, nous allons \'etendre ce probl\`{e}me de r\'egularit\'e pour des types de g\'eom\'etrie plus g\'en\'eraux. Nous nous int\'eressons aux vari\'et\'es non compactes admettant une compactification par une vari\'et\'e compacte non seulement \`a bord mais aussi plus g\'en\'eralement \`a coins. 
Nous supposerons que cette compactification induit une structure de Lie \`a l'infini compatible avec les m\'etriques au sens de~\cite{ammann2004geometry}. Pour pouvoir \'etudier le comportement asymptotique des m\'etriques \`a l'infini, nous devrons imposer certaines conditions sur les structures de Lie \`a l'infini consid\'er\'ees. Plus pr\'ecis\'ement, nous introduirons les structures de Lie \texttt{<<} fibr\'ees \texttt{>>} \`a l'infini, une certaine classe de structures de Lie \`a l'infini induisant sur chaque face de la vari\'et\'e \`a coins un fibr\'e naturel. 
Notre th\'eor\`{e}me peut s'\'enoncer comme suit  (voir les Th\'eor\`{e}me~\ref{theore} et Th\'eor\`{e}me~\ref{theorii}) :

 \begin{them}
Soit $(M, g_{0})$ une vari\'et\'e riemannienne compl\`{e}te de dimension $n$ compatible avec une structure de Lie fibr\'ee \`a l'infini $(M,\mathcal{V}_{SF})$ dont le rayon d'injectivit\'e est strictement positif. Si la m\'etrique $g_{0}$ admet une expansion polyhomog\`{e}ne \`a l'infini compatible avec $(M,\mathcal{V}_{SF}),$ alors la solution du flot de Ricci\texttt{-}DeTurck pr\'eserve l'expansion polyhomog\`{e}ne \`a l'infini dans un court laps de temps. De plus, si la m\'etrique $g_{0}$ est asymptotiquement Einstein, alors la solution du flot de Ricci\texttt{-}DeTurck pr\'eserve l'expansion polyhomog\`{e}ne \`a l'infini tant que le flot existe. 
\end{them}
Dans la premi\`{e}re section, nous commen\c{c}ons par introduire la notion de vari\'et\'es \`a coins au sens de Melrose~\cite{melrose1996differential}). 
Ensuite, en se r\'ef\'erant \`a l'article de~\cite{ammann2004geometry}, 
nous pr\'esentons la notion d'alg\`{e}bre de Lie structurale et la notion \'equivalente d'un alg\'ebro\"ide de Lie bordant, ce qui permet de d\'efinir une structure de Lie \`a l'infini. Elle d\'etermine des m\'etriques compl\`{e}tes dont les d\'eriv\'ees contravariantes du tenseur de courbure sont born\'ees. Nous introduisons alors la notion de structures de Lie fibr\'ees \`a l'infini, la classe de structures de Lie \`a l'infini pour laquelle notre r\'esultat s'applique. Bien que cette classe ne couvre pas toutes les structures de Lie \`a l'infini, par exemple celles associ\'ees \`a des m\'etriques feuillet\'ees au bord au sens de~\cite{rochon2012pseudodifferential}, il n'en demeure pas moins qu'un tr\`{e}s large \'eventail de m\'etriques compl\`{e}tes sont compatibles avec une structure de Lie fibr\'ee \`a l'infini. Voici quelques exemples \`a titre indicatif : les b\texttt{-}m\'etriques~\cite{melrose1993atiyah}, les m\'etriques de diffusion~\cite{melrose1995geometric} (incluant les m\'etriques asymptotiquement localement euclidiennes~\cite{joyce2001asymptotically}), les m\'etriques edges~\cite{rafe1991elliptic} (incluant les 0\texttt{-}m\'etriques~\cite{mazzeo1987meromorphic}), les $\Phi$\texttt{-}m\'etriques de Mazzeo\texttt{-}Melrose~\cite{mazzeo1998pseudodifferential} (incluant les m\'etriques asymptotiquement localement plates~\cite{gibbons1979positive}) ainsi que leur g\'en\'eralisation \`a des m\'etriques quasi\texttt{-}asymptotiquement coniques (QAC) ou quasi\texttt{-}fibr\'ees au bord (QFB) de~\cite{conlon2019quasi}.

\`A la deuxi\`{e}me section, nous traitons de la polyhomog\'en\'eit\'e globale des solutions d'\'equations paraboliques lin\'eaires d\'etermin\'ees par une structure de Lie fibr\'ee \`a l'infini. Plus pr\'ecis\'ement, nous montrons que la solution admet un d\'eveloppement asymptotique polyhomog\`{e}ne complet \`a l'infini pourvu que ce soit le cas initialement.

La troisi\`{e}me section concerne les \'equations paraboliques quasi\texttt{-}lin\'eaires d\'etermin\'ees par une structure de Lie fibr\'ee \`a l'infini. Nous \'etablissons l'existence et l'unicit\'e des solutions de telles \'equations en s'inspirant de l'article de Bahuaud~\cite{bahuaud2011ricci} d'une part. D'autre part, nous d\'eterminons les crit\`{e}res qui nous permettent d'assurer que la polyhomog\'en\'eit\'e sera pr\'eserv\'ee pour un court laps de temps.

\`A la derni\`{e}re section, nous pr\'esentons notre r\'esultat principal en se basant sur les deux sections pr\'ec\'edentes. En premier lieu, nous prouvons la polyhomog\'en\'eit\'e locale des m\'etriques compatibles avec une structure de Lie fibr\'ee \`a l'infini le long du flot de Ricci\texttt{-}DeTurck. Remarquons ici que ce r\'esultat s'applique aux m\'etriques asymptotiquement coniques ou cylindriques pour lesquelles le terme d'ordre $0$ de l'expansion a \'et\'e d\'ej\`a \'etabli dans~\cite{lott2016ricci} lorsque la m\'etrique est K\"ahler.
Ensuite, nous \'etablissons la polyhomog\'en\'eit\'e globale pour des m\'etriques asymptotiquement Einstein le long du flot. 
Enfin, lorsqu'initialement la m\'etrique admet un d\'eveloppement lisse, nous montrons qu'il en sera de m\^eme pour les m\'etriques le long du flot. Dans ce cas particulier, nos r\'esultats sont valides non seulement pour le flot de Ricci\texttt{-}DeTurck, mais aussi pour le flot de Ricci.

\textbf{Remerciement.}
Cet article pr\'esente les r\'esultats de ma th\`{e}se de doctorat. Je tiens vivement \`a remercier le Professeur Fr\'ed\'eric Rochon mon directeur de th\`{e}se qui a supervis\'e ce travail avec beaucoup de vision et de rigueur. Il a dirig\'e ma th\`{e}se avec patience et il a d\'edi\'e beaucoup de temps \`a mon travail en \'etant toujours tr\`{e}s disponible. Je remercie ensuite tr\`{e}s sinc\`{e}rement le Professeur Eric Bahuaud pour avoir lu attentivement une version pr\'eliminaire de ce travail et avoir fait plusieurs remarques et suggestions importantes. Finalement, je tiens vivement \`a remercier le rapporteur de cet article pour le temps consacr\'e  \`a  la lecture, et pour les suggestions et les remarques judicieuses qu'il m'a indiqu\'ees. 
\section{Vari\'et\'e \`a coins et structures de Lie \`a l'infini} 

Les vari\'et\'es riemanniennes non compactes avec une structure de Lie \`a l'infini ont \'et\'e introduites par Ammann, Lauter et Nistor dans leur article~\cite{ammann2004geometry}. Toute structure de Lie \`a l'infini d\'etermine une classe de m\'etriques compl\`{e}tes qui sont de g\'eom\'etrie born\'ee lorsque le rayon d'injectivit\'e est strictement positif.
Cette structure englobe une large classe de vari\'et\'es compl\`{e}tes non compactes comme on le verra plus\texttt{-}bas.
En particulier, nous allons d\'efinir une nouvelle notion : celle de \texttt{<<} structure de Lie fibr\'ee \`a l'infini \texttt{>>}. Cette derni\`{e}re forme une classe de structures Lie \`a l'infini qui induisent dans un sens pr\'ecis des structures de fibr\'es au bord. 

L'analyse globale sur les vari\'et\'es \`a coins a \'et\'e d\'evelop\'ee par Melrose dans plusieurs ouvrages, notamment~\cite{melrose1993atiyah} et~\cite{melrose1996differential}. Nous donnerons un aper\c{c}u g\'en\'eral du concept de vari\'et\'es \`a coins afin de d\'evelopper les quelques notions de la g\'eom\'etrie des vari\'et\'es avec une structure de Lie \`a l'infini qui nous seront n\'ecessaires. Pour plus d\'etails, on r\'ef\`{e}re aussi le lecteur \`a~\cite{grieser2017scales}.\\
Pour $k \in \{0,...,n\}$, on note par $\mathbb{R}^{n}_{k}$ l'espace donn\'e par $$\mathbb{R}^{n}_{k} =([0,+\infty[)^{k} \times \mathbb{R}^{n-k}.$$ L'ensemble des ouverts de $\mathbb{R}^{n}_{k}$ est $\{\Omega' \,\cap \, \mathbb{R}^{n}_{k} \mid \Omega' \:\textrm{ouvert\:de}\: \mathbb{R}^{n} \}.$ 
Soit $\Omega$ un ouvert de $\mathbb{R}^{n}_{k}$. Une fonction $f : \Omega \rightarrow \mathbb{C}$ est dite lisse s'il existe un ouvert $\Omega'$ de $\mathbb{R}^{n}$ avec $\Omega=\Omega' \,\cap \, \mathbb{R}^{n}_{k}$ et une fonction lisse $F : \Omega' \rightarrow \mathbb{C}$ tel que $F_{\scriptscriptstyle{\vert\Omega }}=f$. Soient deux ouverts  $\Omega_{i}$ de $ \mathbb{R}^{n}_{k_{i}}$, o\`u $i \in \{1,2\}$. On dit que $f: \Omega_{1} \rightarrow\Omega_{2}$ est un diff\'eomorphisme si elle est un hom\'eomorphisme qui admet un inverse $g: \Omega_{2}\rightarrow\Omega_{1} $ tel que chaque composante de coordonn\'ees de $f$ ou $g$ soit une application lisse.

Maintenant, soit $M$ un espace topologique, s\'epar\'e et paracompact.
\begin{defi}
Une carte \`a coins de $M$ est un couple $(U,\varphi)$ constitu\'e d'un ouvert $U$ de $M$ et d'un hom\'eomorphisme sur un ouvert de $\mathbb{R}^{n}_{k}$, $\varphi : U \rightarrow  \mathbb{R}^{n}_{k}$. Elle est dite centr\'ee en $m\in U$ si $\varphi (m)=0.$\\
Une $C^\infty$ structure \`a coins est d\'etermin\'ee par la donn\'ee d'une famille $\mathfrak{A}=\{(U_{i}, \varphi_{i}), i\in I\}$ de cartes \`a coins de $M$ ayant les propri\'et\'es suivantes :
\begin{enumerate}
\item $(U_{i})_{i\in I}$ est un recouvrement ouvert de $M$. \label{lar}
\item Si $i\neq j$, alors $U_{i}$ et $U_{j}$ sont compatibles, c'est\texttt{-}\`a\texttt{-}dire que si $U_{i} \cap U_{j} \neq \emptyset $, alors $ \varphi_{j}  \circ  \varphi_{i}^{-1} $ est un diff\'eomorphisme lisse de $ \varphi_{i}(U_{i} \cap U_{j})$ sur $ \varphi_{j}(U_{i} \cap U_{j})$. \label{lari}
\item Si $\mathfrak{B} \supset \mathfrak{A} $ est une famille de cartes \`a coins ayant les propri\'et\'es~\ref{lar} et~\ref{lari} alors $\mathfrak{B} = \mathfrak{A} $. 
L'ensemble  $\mathfrak{A}$ est l'atlas maximal de $M$.  
\end{enumerate}
\end{defi}
Si $M$ a une $C^\infty$ structure \`a coins, alors on note par
$$C^\infty(M)= \{f : M \rightarrow \mathbb{R} \mid f \circ \varphi^{-1} \textrm{est\:lisse\:sur}\:\varphi(U) \:\textrm{ pour\:chaque\:carte}\: (U, \varphi) \in  \mathfrak{A} \}.$$
\begin{defi}
Une pr\'e\texttt{-}vari\'et\'e \`a coins est un couple $(M,\mathfrak{F}),$ o\`u $M$ est un espace topologique, s\'epar\'e et paracompact et $\mathfrak{F}=C^\infty(M)$ pour une certaine $C^\infty$ structure \`a coins. 
\end{defi}
\begin{defis}
Soient $M$ une pr\'e\texttt{-}vari\'et\'e \`a coins et $$\partial_{l} M = \{m\in M \: \textrm{tel\:qu'il\:existe\:une\:carte}\:  \varphi : U \rightarrow  \mathbb{R}^{n}_{l} \: \textrm{centr\'ee\:en} \:m \},$$ pour $l \in \{0,...,k\}$ appel\'ee la codimension, ou la profondeur de $m$. L'espace $\partial_{l} M$ consiste en l'ensemble des points de profondeur $l$ qui se d\'ecompose en une r\'eunion de composantes connexes appel\'ees les faces ouvertes. La fermeture d'une face ouverte est une face ferm\'ee. On appelle une hypersurface bordante de $M$ une face ferm\'ee de codimension $1$. L'ensemble des hypersurfaces bordantes est not\'e par $\mathcal{M}_{1}(M)$. La profondeur de $M$ est la profondeur maximale qu'un point $m$ de $M$ puisse avoir. Autrement dit, la profondeur de $M$ est la codimension maximale qu'une face ouverte de $M$ puisse avoir.
\end{defis}
\begin{remq}
L'ensemble $\partial_{l} M$ est ind\'ependant du choix des cartes \`a coins. Le bord de $M$, not\'e $\partial M$, peut s'\'ecrire comme l'union $\bigcup _{l=1}^k\partial_{l} M$. Ainsi l'int\'erieur de $M$ est d\'efini par $\overset{\circ}{M}:=M \setminus \partial M$.
\end{remq}
\begin{defi}
On dit qu'un sous\texttt{-}ensemble $S \subset M$ est une sous\texttt{-}vari\'et\'e, si $\forall s \in S$, $\exists G \in Gl_{n}(\mathbb{R})$ et $\exists \varphi : U \rightarrow \mathbb{R}^{n}_{k}$ une carte \`a coins centr\'ee en $s$ tels que $$ \varphi _{\scriptscriptstyle{\vert S\cap U }} : S\cap U \rightarrow (G \cdot \mathbb{R}^{n-n'}_{k'} \times \{0\}^{n'}) \cap \varphi(U) $$ est un hom\'eomorphisme, o\`u $k' \in \{0,...,k\}$ et $n' \in \{0,...,n\}$ avec  $k'\leq n-n'$. \\
Lorsque $k=k'=0$ et $G$ est l'identit\'e, cela co\"incide avec la d\'efinition d'une sous\texttt{-}vari\'et\'e d'une vari\'et\'e sans bord.
\end{defi}
\begin{exemp} 
La diagonale $\Delta = \{(x,x) \mid x \in [0,+\infty) \} $ est une sous\texttt{-}vari\'et\'e de $\mathbb{R}^{2}_{2}$ car \\$\Delta=$
$\begin{pmatrix} 1 & -1\\ 1 & 1 \end{pmatrix} $ $ \cdot \;\{(x,0) \mid x \in [0,+\infty)\}$. 
\end{exemp}
\begin{defi}
Une fonction de d\'efinition d'une hypersurface bordante est une fonction lisse positive $\rho \in C^\infty(M)$ qui s'annule sur $H$ et seulement sur $H$, et dont la diff\'erentielle soit non\texttt{-}nulle partout sur $H$.
\end{defi}
\begin{defi}
Une vari\'et\'e \`a coins $M$ est une pr\'e\texttt{-}vari\'et\'e \`a coins telle que toutes les hypersurfaces bordantes sont des sous\texttt{-}vari\'et\'es, ou de mani\`{e}re \'equivalente (Lemme 1.8.1~\cite{melrose1996differential}), telle que chaque hypersurface bordante $H$ admette une fonction de d\'efinition.
\end{defi}
\begin{remq} 
Si $F$ est une face ferm\'ee connexe de $M$ de codimension $l$, alors $F$ est une composante connexe de l'intersection des hypersurfaces bordantes contenant $F$. Autrement dit, si $\rho_{1},\rho_{2}, ...,\rho_{l}$ sont des fonctions de d\'efinition des hypersurfaces bordantes contenant $F$, alors $F$ est une composante connexe de $$\{ m\in M \mid \rho_{1}(m)=\rho_{2}(m)= ...=\rho_{l}(m)=0\}.$$
\end{remq}
\begin{exemp}
Le carr\'e $[0,1]\times [0,1]$ est une vari\'et\'e \`a coins. Par contre, la goutte $$T=\{(x,y) \in \mathbb{R}^{2} \mid x\geq 0, \: y^{2}\leq x^{2}-x^{4} \},$$ voir la figure ci\texttt{-}dessous, est une pr\'e\texttt{-}vari\'et\'e \`a coins qui n'est pas une vari\'et\'e \`a coins car son hypersurface bordante n'est pas une sous\texttt{-}vari\'et\'e.
 \begin{center}
\includegraphics[scale=0.3]{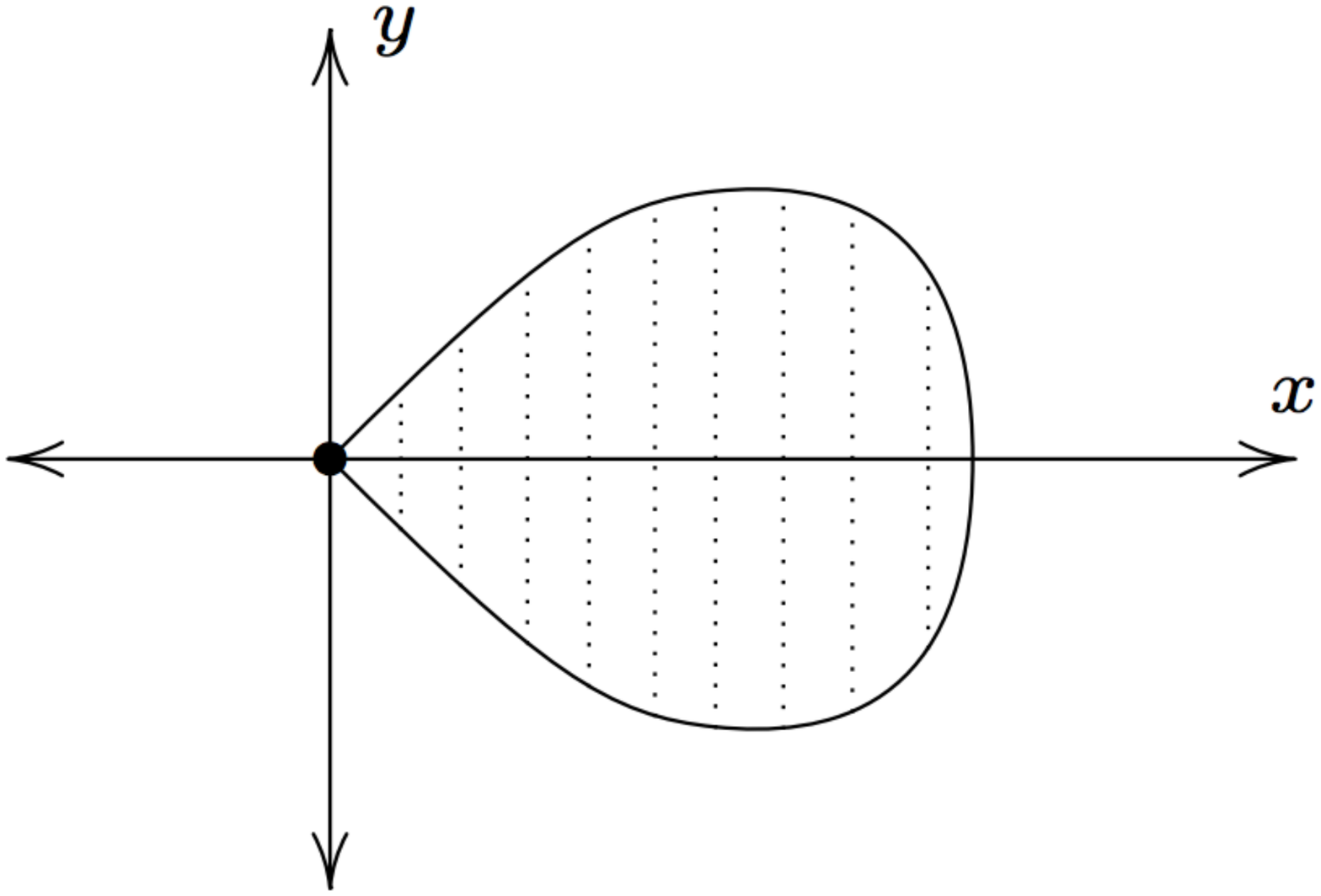} 
\end{center}
\end{exemp}
\begin{defi}
Un fibr\'e est un quadruplet $ (E,S, Z,\phi ),$ o\`u $E$, $S$, $Z$ sont des vari\'et\'es \`a coins et $\phi : E \rightarrow S$ est une application lisse telle que pour tout $s \in S$, il existe un voisinage ouvert $U \subset S$ de $s$, appel\'e ouvert trivialisant, tel qu'on a un diff\'eomorphisme $\Psi : \phi^{-1}(U) \rightarrow U \times Z$ induisant le diagramme commutatif  : 
\begin{displaymath}
\xymatrix@+1pc{  \phi^{-1}(U)
 \ar[r]^{\Psi} \ar[rd]_{\phi}& U \times Z \ar[d]^{ \pr_{1}}\\
   & U.
  }
\end{displaymath} 
o\`u $\pr_{1}$ est la projection sur le premier facteur.
On dira que $(E, S, Z, \phi)$ est d'espace total $E$, de base $S$, de fibre $Z$ et de projection $\phi$. L'application $\Psi$ est appel\'ee trivialisation. Enfin $E_{s} := \phi^{-1}(s)$ est la fibre au\texttt{-}dessus de $s$.
\end{defi}
\begin{defi}
Si $Z$ est un espace vectoriel r\'eel (ou complexe) et si pour toute paire de trivialisations $(U_{i},\Psi_{i})$, $i\in\{1,2\}$ et $s \in U_{1}\cap U_{2}$, l'application ${\Psi_{1}}_{\scriptscriptstyle{\vert E_{s}}} \circ \big({\Psi_{2}}_{\scriptscriptstyle{\vert E_{s}}}\big)^{-1} : Z \rightarrow Z$ est lin\'eaire, alors $(E, S, Z, \phi)$ est un fibr\'e vectoriel r\'eel (ou complexe).	
\end{defi}
\begin{exemp}
Soit $M$ une vari\'et\'e \`a coins. Pour $m\in M,$ on consid\`{e}re l'id\'eal $$I_{m}=\{f\in C^\infty(M) \mid f(m)=0\}.$$ Alors on pose 
$T^{*}_{m}M= I_{m} / I^{2}_{m},$ o\`u $$I^{2}_{m}=\{f\in C^\infty(M) \mid \exists k \in \mathbb{N},\; g_{1}, h_{1},...,  g_{k},h_{k} \in I_{m} \;\textrm{tels\;que}\;  f= \sum \limits_{{i=1}}^k g_{i}h_{i}\}.$$ On peut v\'erifier que le fibr\'e cotangent, donn\'e par $T^{*}M= \bigcup _{m \in M}T^{*}_{m}M,$ et son dual $TM$ le fibr\'e tangent sont deux vari\'et\'es \`a coins constituant l'espace total de deux fibr\'es au sens des vari\'et\'es \`a coins. 
\end{exemp}
Soit $M$ une vari\'et\'e \`a coins de dimension n. 
\begin{defi}\label{larin}
Une alg\`{e}bre de Lie structurale des champs de vecteurs lisses sur $M$ est un sous\texttt{-}espace $\mathcal{V} \subset \mathfrak{X}(M) $ de l'espace vectoriel r\'eel des champs de vecteurs sur $M$ satisfaisant aux propri\'et\'es suivantes :
\begin{enumerate}
\item $\mathcal{V}$ est ferm\'e sous le crochet de Lie des champs de vecteurs;\label{larinn}
\item chaque $V\in \mathcal{V}$ est tangent \`a toutes les faces de $M$ ;\label{larinnn}
\item $\mathcal{V}$ est un $C^\infty(M)$ module localement libre, c'est\texttt{-}\`a\texttt{-}dire que 
        $\forall m \in M$, il existe un voisinage ouvert $U_{m} \subset M$ de $m,$ $ k\in \mathbb{N}$ et une famille de champs de vecteurs $\{X_{1},..., X_{k}\} \subset \mathcal{V}$ de sorte que $\forall Y \in  \mathcal{V}$, il existe une famille de fonctions lisses $\{f_{1},..., f_{k}\} \subset C^\infty(M)$, uniquement d\'etermin\'ee sur $U_{m}$ par $Y$, telle que $\displaystyle Y= \sum \limits_{{i=1}}^k f_{i}X_{i} \; sur \; U_{m}.$ \label{larinnnn}
\end{enumerate}
\end{defi}
\begin{exemp}~\cite{melrose1996differential} \label{bchamp}
Soient $M$ une vari\'et\'e \`a coins et 
\begin{align*}
\mathcal{V}_{b}&=\{V\in \mathfrak{X}(M)\: \textrm{\:tel\:que\:V\:est\:tangent\:\`a\:toutes\:les\:faces\:de\:M}\}\\
                         &=\{ V\in \mathfrak{X}(M) \mid V\rho_{H}=a_{H}\rho_{H},\; a_{H} \in C^\infty(M), \; \forall H \in \mathcal{M}_{1}(M) \},
\end{align*}
o\`u $\rho_{H}$ est une fonction de d\'efinition de l'hypersurface bordante $H$.
Un champ de vecteurs $V\in \mathcal{V}_{b}$ est appel\'e un b\texttt{-}champ de vecteurs.
Par un calcul standard, on a que $\forall \: V_{1}, V_{2} \in  \mathcal{V}_{b}$, $[V_{1}, V_{2}]\rho_{H}\in\rho_{H}C^\infty(M)$. Alors $\mathcal{V}_{b}$ est ferm\'e sous le crochet de Lie des champs de vecteurs. Soit $m$ un \'el\'ement d'une face ferm\'ee $F$ de $M$ de profondeur $l$. Sur $(\rho_{1},...,\rho_{l},y_{1},...,y_{n-l})$ un syst\`{e}me de coordonn\'ees locales centr\'e en $m$, chaque b\texttt{-}champ de vecteurs $V\in \mathcal{V}_{b}$ est de la forme:
$$V=\sum \limits_{{i=1}}^l a_{i}(\rho,y)\rho_{i}\frac{\partial} {\partial \rho_{i}} + \sum \limits_{{i=1}}^{n-l} b_{i}(\rho,y)\frac{\partial }{\partial y_{i}}$$ 
o\`u les coefficients $a_{i}$ et $b_{i}$ sont des fonctions lisses, $y=(y_{1},...,y_{n-l}) \in \mathbb{R}^{n-l}$, et $\rho=(\rho_{1},...,\rho_{l})$ est une famille de fonctions de d\'efinition des hypersurfaces bordantes contenant $F$. Cela montre que l'alg\`{e}bre de Lie des b\texttt{-}champs de vecteurs est engendr\'ee par $\rho_{i} \frac{\partial }{\partial \rho_{i}}$ et $\frac{\partial }{\partial y_{i}}$ pr\`{e}s de $m\in F$, tandis que pr\`{e}s de n'importe quel $m \in \overset{\circ}{M}$, elle est engendr\'ee par le rep\`{e}re habituel $\{\frac{\partial }{\partial x_{1}},...,\frac{\partial }{\partial x_{n}}\}$. Par cons\'equent, $\mathcal {V}_{b}$ est une alg\`{e}bre de Lie structurale des champs de vecteurs puisqu'elle est un $C^\infty$module localement libre de rang $n= \dim M$. D'autre part, on peut associer \`a $\mathcal {V}_{b}$ un fibr\'e vectoriel ${}^{\mathcal{V}_{b}}TM= \bigcup _{m \in M} (\mathcal {V}_{b}/ I_{m} \mathcal {V}_{b})  $, not\'e ${}^bTM$. En effet, il reste \`a v\'erifier que ${}^bTM$ est $C^\infty$, autrement dit, que la matrice de transition entre les changement de cartes est $C^\infty$. Soit $\Phi$ un changement de coordonn\'ees. Si $x=(\rho_{1},...,\rho_{l},y_{l+1},...,y_{n})$ est un syst\`{e}me de coordonn\'ees locales centr\'e en $m \in F$, alors $\Phi(x)$ est un nouveau syst\`{e}me de coordonn\'ees locales de $m$ donn\'e par 
$$\rho'_{i}=\Phi_{i}(x)=\rho_{i}\alpha_{i}(x),\alpha_{i}(x)>0, \forall i \in \{1,...,l\}.$$
Il est clair que les coefficients de la matrice de transition sont $C^\infty$ car :  \\
Si $ i \in \{1,...,l\},$ on calcule que   
$$\rho_{i}\frac{\partial} {\partial \rho_{i}}=(1+ \frac{\rho_{i}}{\alpha_{i}}\frac{\partial \alpha_{i}}{\partial \rho_{i}})\rho'_{i} \frac{\partial}{\partial \rho'_{i}}+ \sum \limits_{\substack{j\leq l \\ i \neq j}}
^{} \frac{\rho_{i}}{\alpha_{j}} \frac{\partial \alpha_{j}}{\partial \rho_{i}} \rho'_{j} \frac{\partial}{\partial \rho'_{j}} +  \sum \limits_{{j=l+1}}^{n} \rho_{i}\frac{\partial\Phi_{j}}{\partial \rho_{i}}\frac{\partial}{\partial\Phi_{j}}.$$
Si $i \in \{l+1,...,n\},$ on a plut\^ot que
$$\frac{\partial}{\partial y_{i}}= \sum \limits_{{j=1}}^{l} \frac{1}{\alpha_{j}} \frac{\partial \alpha_{j}}{\partial y_{i}} \rho'_{j} \frac{\partial}{\partial \rho'_{j}} +  \sum \limits_{{j=l+1}}^{n} \frac{\partial\Phi_{j}}{\partial y_{i}}\frac{\partial}{\partial\Phi_{j}}.$$
D'o\`u ${}^bTM$ est bien d\'efini. De plus, on remarque qu'il existe un morphisme canonique de fibr\'es vectoriels $$\varrho_{\mathcal{V}_{b}}:  {}^bTM \rightarrow TM,$$ qui envoie la classe d'un champ de vecteurs $V\in\mathcal {V}_{b}$ \`a un vecteur tangent $V(m) \in T_{m}M$. Sur un syst\`{e}me de coordonn\'ees locales de $m \in F$, on voit que le champ de vecteurs $\rho_{i} \frac{\partial }{\partial \rho_{i}}$ n'est pas nul dans ${}^bT_{m}M$ car $\rho_{i} \frac{\partial }{\partial \rho_{i}} \notin  I_{m} \mathcal {V}_{b}$. Par contre $\rho_{i} \frac{\partial }{\partial \rho_{i}}$ est nul dans $T_{m}M=\mathfrak{X}(M)/ I_{m} \mathfrak{X}(M)$. On d\'eduit que $\varrho_{\mathcal{V}_{b}}$ est un isomorphisme sur $\overset{\circ}{M}$ et de rang $n-l$ sur l'int\'erieur d'une face de codimension $l,$ $$\rho_{i} \frac{\partial }{\partial \rho_{i}} \mapsto 0 \;\; \;\; \frac{\partial }{\partial y_{j}} \mapsto \frac{\partial }{\partial y_{j}},\; j>l, \:i \leq l.$$ 
\end{exemp}
\begin{remq}
En utilisant la propri\'et\'e~\ref{larinnn} de la D\'efinition~\ref{larin}, on peut voir que chaque alg\`{e}bre de Lie structurale $\mathcal{V}$ est une sous\texttt{-}alg\`{e}bre de $\mathcal{V}_{b}$.
\end{remq}
\begin{remq}
On remarque que la propri\'et\'e~\ref{larinnnn} de la D\'efinition~\ref{larin}  \'equivaut \`a dire que $\mathcal{V}$ est un $C^\infty$module projectif. Ainsi, d'apr\`es le th\'eor\`{e}me Serre\texttt{-}Swan~\cite{karoubi2008k}, il existe un fibr\'e vectoriel $${}^{\mathcal{V}}TM :=   \bigcup _{m \in M} (\mathcal {V}/ I_{m} \mathcal {V}) \rightarrow M$$ tel que $\mathcal{V}\simeq\Gamma( {}^{\mathcal{V}}TM )$. Puisque $\mathcal{V} \subset \mathfrak{X}(M)$, il existe un morphisme canonique de fibr\'es vectoriels $$\varrho_{\mathcal{V}}: {}^{\mathcal{V}}TM \rightarrow TM,$$ appel\'e l'ancre ou encore l'application d'ancrage. 
\end{remq}
\begin{exemp}~\cite{rafe1991elliptic} \label{rasi}
Soient $M$ une vari\'et\'e \`a bord $\partial M$, qui est l'espace total d'un fibr\'e de vari\'et\'es lisses $\pi : \partial M \rightarrow B$, et $$\mathcal {V}_{e}=\{V \in \mathfrak{X}(M) \; tel \; que \; V \; est \; tangent \; \textrm{\`a}\; toutes \;les \;fibres\;de\;\pi \}.$$
Soient $V_{1}, V_{2} \in \mathcal {V}_{e}.$ Puisque $[V_{1},V_{2}]_{\scriptscriptstyle{\vert F}}=[{V_{1}}_{\scriptscriptstyle{\vert F}},{V_{2}}_{\scriptscriptstyle{\vert F}}]$ pour $F$ une fibre de $\pi$, le commutateur est bien tangent aux fibres de $\pi$. Soit $(\rho,y,z)$ un syst\`{e}me de coordonn\'ees locales pr\`{e}s du bord, o\`u $\rho$ est une fonction de d\'efinition du bord, $y$ est une famille de variables sur la base $B$ de $\pi$ et $z$ est une famille de variables sur les fibres de $\pi.$ Alors dans ce syst\`{e}me de coordonn\'ees,
$\mathcal {V}_{e}$ est engendr\'e par $\rho\frac{\partial }{\partial \rho}$, $\rho \frac{\partial }{\partial y}$ et $\frac{\partial }{\partial z}$. Il s'agit donc d'un $C^\infty$module projectif. On d\'eduit que $\mathcal{V}_{e}$ est une alg\`{e}bre de Lie structurale et qu'on peut lui associer un fibr\'e vectoriel not\'e  ${}^{e}TM$. En particulier, lorsque $B:=\partial M$ et $\pi=\Id,$ on retrouve l'alg\`{e}bre de Lie des 0\texttt{-}champs de vecteurs de Mazzeo\texttt{-}Melrose~\cite{mazzeo1987meromorphic} sur une vari\'et\'e \`a bord, 
$$\mathcal {V}_{0}=\{V \in \mathfrak{X}(M) \mid \; V_{\scriptscriptstyle{\vert \partial M}}=0 \}.$$ 
Dans ce cas, on d\'enote ${}^{e}T_{}M$ par ${}^{0}T_{}M.$
\end{exemp}
Nous pr\'esentons maintenant quelques d\'efinitions et r\'esultats standards concernant
les alg\'ebro\"ides de Lie, dont il sera fait usage dans la suite. Des r\'ef\'erences sur cette notion sont~\cite{mackenzie1987lie} et~\cite{crainic2003integrability}. \\
Soit $M$ une vari\'et\'e \`a coins.
 \begin{defi}
Un alg\'ebro\"ide de Lie sur $M$ est la donn\'ee d'un fibr\'e vectoriel $A \rightarrow M$, d'un morphisme $\varrho: A \rightarrow TM,$ appel\'e ancre, et d'une structure d'alg\`{e}bre de Lie sur son module de sections globales $\Gamma(A)$ dont le crochet satisfait \`a la r\`{e}gle de Leibniz,
$$[u,f v] = f [u, v] + (\varrho_{\Gamma}(u).f ) v, \; \forall u,v \in \Gamma(A), \;f \in C^{\infty}(M),$$ o\`u $\varrho_{\Gamma} : \Gamma(A) \rightarrow \Gamma(TM)$ est l'application des modules des sections induites par l'ancre $\varrho$.
 \end{defi}
 \begin{remq}
 Par l'antisym\'etrie du crochet, la r\`{e}gle de Leibniz \`a gauche est aussi satisfaite :
 $$[fu, v] = f [u, v] - (\varrho_{\Gamma}(v).f ) u, \: \forall u,v \in \Gamma(A),\;f \in C^{\infty}(M).$$
 \end{remq}
 \begin{prop}
L'application $\varrho_{\Gamma}$ est un morphisme d'alg\`{e}bres de Lie. Autrement dit,\\ $\varrho_{\Gamma}([u,v])= [\varrho_{\Gamma}(u), \varrho_{\Gamma}(v)],$ $\forall u,v \in \Gamma(A),$ o\`u \`a gauche on a le crochet d'alg\'ebro\"ide de Lie et \`a droite le crochet de Lie des champs de vecteurs.
 \end{prop}
 \begin{proof}
Par le r\`{e}gle de Leibniz, on a que $$(\varrho_{\Gamma}(u).f ) v= [u, fv]-f[u, v] \: \forall u,v \in \Gamma(A),\;f \in C^{\infty}(M).$$ Par suite, on obtient que $\forall w \in \Gamma(A),$ 
\begin{align*}
([\varrho_{\Gamma}(u),\varrho_{\Gamma}(v)].f)w & =\big(\varrho_{\Gamma}(u).(\varrho_{\Gamma}(v).f)\big)w-\big(\varrho_{\Gamma}(v).(\varrho_{\Gamma}(u).f)\big)w\\
                                                                             & =[u,(\varrho_{\Gamma}(v).f )w]-(\varrho_{\Gamma}(v).f )[u,w]\\
                                                                             &\;\;\;\;\;\;\;\;\;\;\;\;\;\;\;\;\;\;\;\;\;\;\;\;\;\;\;\;\;\;-[v,(\varrho_{\Gamma}(u).f )w]+(\varrho_{\Gamma}(u).f )[v,w]\\
                                                                             &= [u,[v,fw]] -[u,f[v,w]]-[v,f[u,w]]+f[v,[u,w]]\\
                                                                                 &\;\;\;-[v,[u,fw]] +[v,f[u,w]]+[u,f[v,w]]-f[u,[v,w]]\\
                                                                             &= [u,[v,fw]] - [v,[u,fw]]+ f\big([v,[u,w]]-[u,[v,w]]\big).                                                                              
\end{align*}
En utilisant l'identit\'e de Jacobi et l'antisym\'etrie du crochet, on d\'eduit donc que
\begin{align*}
([\varrho_{\Gamma}(u),\varrho_{\Gamma}(v)].f)w & =- [fw,[u,v]] + f(-[w,[v,u]])\\
                                            &= [[u,v],fw] - f([[u,v],w]) \\
                                            &=(\varrho_{\Gamma}([u,v]).f)w,
\end{align*}
d'o\`u le r\'esultat.
 \end{proof}
 \begin{exemp}
Tout fibr\'e $A \rightarrow M$ en alg\`{e}bres de Lie est un alg\'ebro\"ide de Lie dont l'ancre est nulle.
 \end{exemp}
 \begin{defi}
 Un alg\'ebro\"ide de Lie sur $M$ est dit bordant si tous les champs de vecteurs de $\varrho_{\Gamma}(\Gamma(A))$ sont tangents \`a toutes les faces de $M$. 
 \end{defi}
Le concept d'une alg\`{e}bre de Lie structurale sera alors \'equivalent au concept d'alg\'ebro\"ide de Lie bordant, si $\varrho_{\Gamma}$ est injective et $\varrho_{\Gamma}(\Gamma(A)) \subset \mathcal {V}_{b}.$
 \begin{defi}
 Soit $ (A \rightarrow M, \varrho_{A}, [\cdot,\cdot]_{A})$ et $ (B \rightarrow M, \varrho_{B}, [\cdot,\cdot]_{B})$ deux alg\'ebro\"ides de Lie au\texttt{-}dessus de la m\^eme base. Un morphisme d'alg\'ebro\"ides de Lie au\texttt{-}dessus de la m\^eme base est la donn\'ee d'un morphisme $\Phi: A \rightarrow B$ entre les fibr\'es vectoriels $A \rightarrow M$ et $B \rightarrow M$ qui commute aux ancres et aux crochets au sens o\`u
$${(\varrho_{A})}_{\Gamma}={(\varrho_{B})}_{\Gamma} \circ \Phi, \;\; \Phi([u,v]_{A})= [\Phi(u), \Phi(v)]_{B}, \;\forall u,v \in \Gamma(A).$$
 \end{defi}
 \begin{defi} \label{larz}
Soit $ (A \rightarrow M, \varrho_{A}, [.,.])$ un alg\'ebro\"ide de Lie. On dit qu'un sous\texttt{-}fibr\'e vectoriel $B \rightarrow N$ de $A_{\scriptscriptstyle{\vert N }} $, o\`u $N$ est une sous\texttt{-}vari\'et\'e de $M$, est un sous\texttt{-}alg\'ebro\"ide de Lie de $A$ si : 
\begin{equation}\label{item}
\forall p \in N, u \in B_{p}, \; on\;a\;que\;\varrho_{A}(u) \in T_{p}N;
\end{equation}
\begin{equation}\label{itemmm}
\forall u,v \in \Gamma(A)\; telles\;que\; u_{\scriptscriptstyle{\vert N }},v_{\scriptscriptstyle{\vert N }} \in  \Gamma(B),\; on\;a\;que\;[u,v]_{\scriptscriptstyle{\vert N }} \in \Gamma(B).
\end{equation}
\end{defi}
\begin{prop}
Le sous\texttt{-}alg\'ebro\"ide de Lie $B \rightarrow N$ est ancr\'e par la restriction de $\varrho_{A}$ \`a $B$ et est \'equip\'e par le crochet de $A$ restreint \`a $\Gamma(B)$. 
Muni de cette ancre et de ce crochet, $B \rightarrow N$ est bien un alg\'ebro\"ide de Lie, not\'e aussi par $B$.
\end{prop}
\begin{proof}
D'abord, on v\'erifie que le crochet $[.,.]_{\scriptscriptstyle{\vert N }}$ est bien d\'efini. En effet, soient $u_{i} \in \Gamma(B)$ et $\tilde{u}_{i},\tilde{v}_{i} \in \Gamma(A)$ deux prolongements pour $u_{i},$ pour chaque $i \in \{1,2\}$. Autrement dit, $\displaystyle \tilde{v}_{i_{\scriptscriptstyle{\vert N }}}= \tilde{u}_{i_{\scriptscriptstyle{\vert N }}}= u_{i}.$ On a donc que 
$ \tilde{u}_{i}-\tilde{v}_{i} = f_{i} \tilde{w}_{i},$
o\`u $\tilde{w}_{i} \in \Gamma(A)$ et $f_{i}\in C^\infty(M)$ est tel que $f_{i_{\scriptscriptstyle{\vert N}}}=0.$ De plus, on voit que $$[\tilde{u}_{1},\tilde{u}_{2}]-[\tilde{v}_{1},\tilde{v}_{2}]= [\tilde{u}_{1}-\tilde{v}_{1},\tilde{u}_{2}]+[\tilde{v}_{1},\tilde{u}_{2}-\tilde{v}_{2}].$$
Or, $$\displaystyle [\tilde{u}_{1}-\tilde{v}_{1},\tilde{u}_{2}]= f_{1}[ \tilde{w}_{1},\tilde{u}_{2}]-\big( {(\varrho_{A})}_{\Gamma}(\tilde{u}_{2}). f_{1}\big) \tilde{w}_{1}.$$ Par la condition~\eqref{item}, on obtient que $\displaystyle {{\big((\varrho_{A})}_{\Gamma}(\tilde{u}_{2}). f_{1}\big)}_{\scriptscriptstyle{\vert N }}=0.$ Par suite, on a que $[\tilde{u}_{1}-\tilde{v}_{1},\tilde{u}_{2}]_{\scriptscriptstyle{\vert N}}=0,$ puisque $f_{1_{\scriptscriptstyle{\vert N }}}=0.$
 De m\^eme, on d\'eduit que $[\tilde{v}_{1},\tilde{u}_{2}-\tilde{v}_{2}]_{\scriptscriptstyle{\vert N}}=0.$ et donc que $[\tilde{u}_{1},\tilde{u}_{2}]_{\scriptscriptstyle{\vert N}}=[\tilde{v}_{1},\tilde{v}_{2}]_{\scriptscriptstyle{\vert N}}.$

Il reste finalement \`a montrer que l'application d'ancrage $\varrho_{B}:={\varrho_{A}}_{\scriptscriptstyle{\vert B}}$ satisfait \`a la r\`{e}gle de Leibniz. Soient $u_{i}\in \Gamma(B),$ et $\tilde{u}_{i} \in \Gamma(A)$ un prolongement de $u_{i},$ pour chaque $i\in\{1,2\}.$ Soit $f \in C^\infty(N)$ et $\tilde{f} \in C^\infty(M)$ un prolongement. On a alors que 
\begin{align*}
[fu_{1}, u_{2}] &=[\tilde{f} \tilde{u}_{1}, \tilde{u}_{2}]_{\scriptscriptstyle{\vert N }} \\
    &={\tilde{f}}_{\scriptscriptstyle{\vert N }} [\tilde{u}_{1}, \tilde{u}_{2}]_{\scriptscriptstyle{\vert N }} - \big({(\varrho_{A})}_{\Gamma}(\tilde{u}_{2}). \tilde{f} )\tilde{u}_{1} \big)_{\scriptscriptstyle{\vert N }} \\
    &=f[u_{1}, u_{2}] - \big({{(\varrho_{B})}_{\Gamma}} (u_{2}). f \big)u_{1}.
\end{align*}
\end{proof}
\begin{exemp}\label{lola}
Soient $M$ une vari\'et\'e \`a coins et $(A \rightarrow M, \varrho, [.,.])$ un alg\'ebro\"ide de Lie bordant. Ainsi, pour chaque face ferm\'ee $F$ de $M$, la restriction de $A$ \`a $F$ est un alg\'ebro\"ide de Lie bordant. En effet, $A_{\scriptscriptstyle{\vert F}}$ est un sous\texttt{-}fibr\'e vectoriel et alors la propri\'et\'e~\eqref{itemmm} est trivialement satisfaite. Puisque l'alg\'ebro\"ide de Lie $A$ est bordant, la propri\'et\'e~\eqref{item} est aussi satisfaite. 
\end{exemp}
\begin{defi}
Une structure de Lie \`a l'infini sur une vari\'et\'e $\overset{\circ}{M}$ est un couple $(M,\mathcal{V})$, o\`u $M$ est une vari\'et\'e \`a coins compacte dont son int\'erieur est $\overset{\circ}{M}$ et $\mathcal{V}$ est une alg\`{e}bre de Lie structurale des champs de vecteurs sur $M$ telle que son ancre $\varrho_{\mathcal{V}}$ est un isomorphisme sur $\overset{\circ}{M}$ (donc ${}^{\mathcal{V}}TM _{\scriptscriptstyle{\vert \overset{\circ}{M}}} \simeq T\overset{\circ}{M}$). 
\end{defi}
Nous allons mettre en \'evidence la d\'efinition ci\texttt{-}dessus par des exemples illustratifs.
\begin{exemp} \label{bcchamp}
Un exemple fondamental est l'alg\`{e}bre structurale $\mathcal{V}_{b},$ vue dans l'Exemple~\ref{bchamp}, qui est bien une structure de Lie \`a l'infini.  
\end{exemp}
\begin{exemp}~\cite{melrose1995geometric} \label{scs}
Soit $M$ une vari\'et\'e \`a bord compacte et $\rho$ une fonction de d\'efinition du bord.
L'espace vectoriel des champs de vecteurs $\mathcal {V}_{sc}= \rho\mathcal{V}_{b}$ est appel\'e l'alg\`{e}bre de Lie de diffusion. Un champ de vecteurs de $\mathcal {V}_{sc}$ est appel\'e un champs de vecteurs de diffusion ou encore un sc\texttt{-}champ de vecteurs. Le couple $(M,\mathcal{V}_{sc})$ est alors une structure de Lie \`a l'infini. En effet, on a que
$\forall \: V_{1}, V_{2} \in  \mathcal{V}_{sc}$,
 $$[V_{1}, V_{2}]=[\rho V'_{1},\rho V'_{2}]=\rho^{2}[V'_{1}, V'_{2}]+\rho(V'_{1}\rho)V'_{2}- \rho(V'_{2}\rho)V'_{1},$$ o\`u  $V'_{1}, V'_{2} \in  \mathcal{V}_{b}.$
Soit $m$ un \'el\'ement de $\partial M$. Dans un syst\`{e}me de coordonn\'ees locales $(\rho,y)$ centr\'e en $m$, chaque sc\texttt{-}champ de vecteurs $V \in \mathcal{V}_{sc}$ est donn\'ee localement par
$$V= a(\rho,y)\rho^{2} \frac {\partial } {\partial \rho}+ \sum \limits_{{i=1}}^{n-1}  a_{i}(\rho,y) \rho \frac {\partial } {\partial y_{i}},$$ o\`u les coefficients $a$ et $a_{i}$ sont des fonctions lisses et $y=(y_{1},...,y_{n-1}) \in \mathbb{R}^{n-1}$.\\
L'alg\`{e}bre de Lie $\mathcal {V}_{sc}$ est engendr\'ee par $\rho^{2}\frac{\partial }{\partial \rho}$ et $\rho \frac{\partial }{\partial y}$ pr\`{e}s du bord.  
Comme $\mathcal {V}_{sc}$ est un $C^\infty$module projectif, on peut lui associer un fibr\'e vectoriel not\'e ${}^{sc}TM$. 
\end{exemp}
\begin{exemp}~\cite{mazzeo1998pseudodifferential} \label{mazzza}
Soit $M$ une vari\'et\'e \`a bord $\partial M$, qui est l'espace total d'un fibr\'e de vari\'et\'es lisses $\Phi : \partial M \rightarrow S$ de fibre typique $Z$. Soit $\rho\in C^{\infty}(M)$ une fonction de d\'efinition du bord. On d\'efinit alors une alg\`{e}bre de Lie de champs de vecteurs par $$\mathcal {V}_{\Phi} := \{V \in \mathfrak{X}(M) \; tel \; que \; V\rho \in \rho^{2}C^{\infty}(M)  \; \textrm{et}\; \Phi_{*}(V_{\scriptscriptstyle{\vert \partial M }}) =0 \}.$$
Soit $(\rho,y,z)$ un syst\`{e}me de coordonn\'ees locales pr\`{e}s du bord, o\`u $y$ est une famille de variables sur la base $S$ de $\Phi$ et $z$ est une famille de variables sur les fibres de $\Phi.$ Alors dans ce syst\`{e}me de coordonn\'ees,
$\mathcal {V}_{\Phi}$ est engendr\'e par $\rho^{2} \frac{\partial }{\partial \rho}$, $\rho\frac{\partial }{\partial y}$ et $\frac{\partial }{\partial z}$. Il s'agit donc d'un $C^\infty$module projectif. On d\'eduit que $\mathcal{V}_{\Phi}$ est une structure de Lie \`a l'infini et qu'on peut lui associer un fibr\'e vectoriel not\'e  ${}^{\Phi}TM$. 
\end{exemp}
\begin{exemp} \label{feuill}
Une g\'en\'eralisation de l'exemple pr\'ec\'edent est la structure Lie \`a l'infini des champs de vecteurs \`a pointes feuillet\'ees au sens de~\cite{rochon2012pseudodifferential}. En effet, soient $M$ une vari\'et\'e compacte \`a bord $\partial M$ qui est muni d'un feuilletage lisse $\mathcal{F}$ et $\rho$ une fonction de d\'efinition du bord. On d\'efinit l'espace vectoriel des champs de vecteurs \`a pointes feuillet\'es (ou tout simplement $\mathcal{F}$\texttt{-}champs de vecteurs) par $$\mathcal {V}_{\mathcal{F}} := \{V \in \mathfrak{X}(M) \; tel \; que \; V\rho \in \rho^{2}C^{\infty}(M)  \; \textrm{et}\; V_{\scriptscriptstyle{\vert \partial M }} \in \Gamma (T\mathcal{F})\}.$$ Alors, on peut voir comme auparavant que $\mathcal {V}_{\mathcal{F}}$ est une structure de Lie \`a l'infini.
\end{exemp}
\begin{exemp} \label{jaw}
Soit $M$ une vari\'et\'e \`a coins compacte ayant $k$ hypersurfaces bordantes $H_{1},...,H_{k}$. Pour chaque $i\in \{1,...k\},$ on suppose que $H_{i}$ admet une structure de fibr\'e, $\pi_{i}: H_{i} \rightarrow S_{i}$, o\`u la base $S_{i}$ et les fibres sont des vari\'et\'es \`a coins compactes. On d\'enote la famille de fibr\'es $(\pi_{i})_{i\in \{1,...k\}}$ par $\pi$. On dit que $(M, \pi)$ est une vari\'et\'e \`a coins fibr\'es s'il existe un ordre partiel sur les $k$ hypersurfaces bordantes tel que :
\begin{enumerate}
\item pour chaque $I \subset \{1,...,k\}$ avec $\bigcap _{ i \in I} H_{i} \neq 0$, la famille $(H_{i})_{i\in I}$ est totalement ordonn\'ee;
\item si $H_{i}<H_{j}$ alors $H_{i} \cap H_{j} \neq 0,$ $\pi_{i_{\scriptscriptstyle{ \vert H_{i} \cap H_{j}}}} :  H_{i} \cap H_{j} \rightarrow S_{i} $ est une submersion surjective, \\ $S_{ji}:=\pi_{j}(H_{i} \cap H_{j} )$ est une hypersurface bordante de $S_{j}$ et il existe une submersion surjective\\ $\pi_{ji}: S_{ji} \rightarrow S_{i}$ satisfaisant $\pi_{ji} \circ \pi_{j}= \pi_{i}$ sur $H_{i} \cap H_{j};$
\item les hypersurfaces bordantes de $S_{i}$ sont donn\'ees par $S_{ij}$ pour $H_{j}<H_{i}.$
\end{enumerate} 
Pour plus de d\'etails, on se r\'ef\`{e}re le lecteur aux~\cite{albin2011resolution} et~\cite{albin2012signature}.
\`a partir de cette d\'efinition, il d\'ecoule directement que chaque base $S_{j}$ est naturellement une vari\'et\'e \`a coins fibr\'es avec la structure fibr\'ee de ses hypersurfaces $(S_{ji})_{i\in \{1,...k\}}$ induite par les fibr\'es $(\pi_{ji}: S_{ji} \rightarrow S_{i})$ avec $H_{i}<H_{j}.$ 
De m\^eme, chaque fibre de $\pi_{i}: H_{i} \rightarrow S_{i}$ est aussi une vari\'et\'e \`a coins fibr\'es. Maintenant, pour $\rho_{i}$ un choix de fonction d\'efinissant $H_{i}$ compatible avec $\pi,$ \`a savoir que pour chaque $j$ tel que $H_{i}<H_{j}$ la restriction $\rho_{i}$ sur $H_{j}$ est constante le long des fibres de $\pi_{j} : H_{j} \rightarrow S_{j}$,
on d\'efinit (voir~\cite{conlon2019quasi})
\begin{multline*}
\mathcal{V}_{QFB}=\{ V\in \mathcal{V}_{b}  \textrm{\:tel\:que\:} \forall i\in \{1,...k\},\;V_{\scriptscriptstyle{ \vert H_{i} }}\:\textrm{est\:tangent\:\`a\:toutes\:les\:fibres\:de\:} \pi_{i}\\
 \textrm{et}\: V \upsilon_{i} \in \upsilon_{i}^{2} C^\infty(M)\: \textrm{o\`u}\: \upsilon_{i}=\prod_{H_{i}<H_{j}} \rho_{j}\}.
\end{multline*}
Un champ de vecteurs $V\in \mathcal{V}_{QFB}$ est appel\'e un champ de vecteurs quasi fibr\'e au bord ou simplement un QFB\texttt{-}champ de vecteurs. 
On peut voir ais\'ement que $\mathcal{V}_{QFB}$ est ferm\'e sous le crochet de Lie des champs de vecteurs. Soit $m$ un \'el\'ement d'une face ferm\'ee $F$ de $M$ de profondeur $l$. Apr\`{e}s un nouvel \'etiquetage, si n\'ecessaire, on peut supposer que $F=H_{1}\cap...\cap H_{l}$ de sorte que $H_{1}<H_{2}<...<H_{l}.$ On prend alors un petit voisinage ouvert de $m$ tel que $\pi_{i}$ est trivial pour chaque $i\in\{1,...,l\}$. On consid\`{e}re le $k_{i}$\texttt{-}tuplet des fonctions $y_{i}=(y_{i}^{1},...,y_{i}^{k_{i}})$ et $z=(z_{1},...,z_{q})$ tel que 
$(\rho_{1},y_{1},...,\rho_{l},y_{l},z)$ d\'efinit le syst\`{e}me de coordonn\'ees locales centr\'e de $m$ qui v\'erifie que sur chaque $H_{i},$ $(\rho_{1},y_{1},...,\rho_{i-1},y_{i-1},y_{i})$ induit des coordonn\'ees sur la base $S_{i}$ avec $\pi_{i}$ correspondant \`a l'application
$$(\rho_{1},y_{1},...,\widehat{\rho_{i}},y_{i},...,\rho_{l},y_{l},z) \mapsto(\rho_{1},y_{1},...,\rho_{i-1},y_{i-1},y_{i}),$$
o\`u le symbole $\:\widehat{}\:$ surmontant une lettre indique qu'il faut l'omettre. 
Et par suite, on peut v\'erifier que dans ce syst\`{e}me de coordonn\'ees locales, $\mathcal {V}_{QFB}$ est engendr\'ee par 
$$\upsilon_{1}\rho_{1}\frac{\partial }{\partial \rho_{1}}, \upsilon_{1} \frac{\partial }{\partial y_{1}^{n_{1}}}, \upsilon_{2}\rho_{2}\frac{\partial }{\partial \rho_{2}}- \upsilon_{1}\frac{\partial }{\partial \rho_{1}}, \upsilon_{2} \frac{\partial }{\partial y_{2}^{n_{2}}},
...,\upsilon_{l}\rho_{l}\frac{\partial }{\partial \rho_{l}}- \upsilon_{l-1}\frac{\partial }{\partial \rho_{l-1}},\upsilon_{l} \frac{\partial }{\partial y_{l}^{n_{l}}},\frac{\partial }{\partial z_{1}},...,\frac{\partial }{\partial z_{q}}$$
pour $n_{i}\in\{1,...,k_{i}\}$ et $\upsilon_{i}= \prod_{p=i}^{l} \rho_{p}$. Il s'agit donc d'un $C^\infty$module projectif et on peut lui associer un fibr\'e vectoriel not\'e ${}^{QFB}TM$. On d\'eduit que $(M, \mathcal {V}_{QFB})$ est une structure de Lie \`a l'infini. 
\end{exemp}
\begin{exemp}\label{eppst}
L'alg\`{e}bre de Lie $(M, \mathcal {V}_{S})$ apparue dans~\cite{debord2015pseudodifferential} est une structure de Lie \`a l'infini car elle est en tout point similaire \`a celle de la structure $(M, \mathcal {V}_{QFB}),$ sauf qu'on exige que $V\rho_{i} \in \rho_{i}^{2}C^\infty(M)$ pour chaque $i$ au lieu de demander que $V \upsilon_{i} \in \upsilon_{i}^{2} C^\infty(M)$.
L'exemple d'alg\`{e}bre de Lie des $\phi$\texttt{-}champs de vecteurs, introduite dans~\cite{mazzeo1998pseudodifferential}, d\'efinit un cas particulier de $(M, \mathcal {V}_{QFB})$ en prenant $M$ une vari\'et\'e \`a bord $\partial M$ qui est l'espace total d'un fibr\'e $\pi:=\phi : \partial M \rightarrow S.$ Elle correspond \`a l'Exemple~\ref{scs} si $S=\partial M$ et $\pi=\Id$.
\end{exemp}
\begin{exemp}\label{eplpst}
Soit $(M, \pi)$ une vari\'et\'e \`a coins fibr\'es. Si $S_{i}= H_{i}$ et $\pi_{i}=\Id$ pour chaque hypersurface bordante maximale $H_{i},$ alors un QFB\texttt{-}champ de vecteurs est dit un QAC\texttt{-}champs de vecteurs ou encore un champs de vecteurs quasi asymptotiquement conique. L'ensemble des QAC\texttt{-}champs de vecteurs est d\'enot\'e par $\mathcal{V}_{QAC}.$ D'o\`u $(M, \mathcal {V}_{QAC})$ est une structure de Lie \`a l'infini.
\end{exemp}
\begin{exemp} \label{epst}
Soient $M$ une vari\'et\'e compacte avec bord $\partial M$ et $\rho$ une fonction de d\'efinition du bord. Soit une 1\texttt{-}forme lisse $\Theta \in C^\infty(M; T^{*}M)$ telle que $\iota^{*} \Theta $ ne s'annule nulle part sur $\partial M$, o\`u $\iota : \partial M \rightarrow M$ est l'inclusion canonique. Soit 
$$\mathcal {V}_{\Theta}:=\{V \in \mathcal {V}_{b} \mid V_{\scriptscriptstyle{\vert \partial M}}=0 \; et \;\Theta(V) \in  \rho^{2} C^\infty(M) \}.$$ 
Alors, $\mathcal {V}_{\Theta}$ est une structure de Lie \`a l'infini, appel\'e $\Theta$\texttt{-}structure. 
Pour la description d'une base locale pr\`{e}s du bord, et pour plus d\'etails, on se r\'ef\`ere \`a l'article~\cite{epstein1991resolvent}.
 \end{exemp}
 \begin{defi}
Une structure de Lie \'evanescente \`a l'infini est une structure de Lie \`a l'infini $\mathcal{V}$ telle que pour tout $V \in  \mathcal{V},$ on a que $V_{\scriptscriptstyle{\vert \partial M }}=0.$
 \end{defi}
 \begin{exemp}
On consid\`{e}re, par exemple $M$ une vari\'et\'e \`a coins compacte de dimension $n$ et   
\begin{align*}
\mathcal{V}_{0}&:=\{ V \in \mathcal{V}_{b}\; \textrm{tel\;que\;pour chaque} \;H \in  \mathcal{M}_{1}(M), \; V(m)_{\scriptscriptstyle{\vert H }}=0 \;  \forall m\in H \}\\
                           &=\rho_{1}...\rho_{k} \mathfrak{X}(M),
\end{align*}
o\`u $\rho_{i}$ est une fonction de d\'efinition associ\'ee \`a l'hypersurface bordante $H_{i}$ avec $ \partial M = \bigcup _{i=1}^k H_{i}$. Il est clair que chaque champ de vecteurs de $\mathcal{V}_{0}$ s'annule sur toutes les hypersurfaces bordantes de $M$. Ce sont les\\ 0\texttt{-}champs de vecteurs de Mazzeo\texttt{-}Melrose~\cite{mazzeo1987meromorphic}.
Pour prouver que $(M,\mathcal{V}_{0})$ est effectivement une structure de Lie \`a l'infini, il faut v\'erifier que $\mathcal{V}_{0}$ est un $C^\infty$module projectif ferm\'e sous le crochet de Lie des champs de vecteurs. En effet, soit $m$ un \'el\'ement de $F$ une face ferm\'ee de $M$ de profondeur $l$. Sur $(\rho_{1},...,\rho_{l},y_{1},....,y_{n-l})$ un syst\`{e}me de coordonn\'ees locales en $m$, un champ de vecteurs $V\in \mathcal{V}_{0}$ est de la forme:
$$V=\sum \limits_{{i=1}}^l a_{i}\rho_{1}...\rho_{k}\frac{\partial} {\partial \rho_{i}} + \sum \limits_{{i=1}}^{n-l} b_{i}\rho_{1}...\rho_{k}\frac{\partial }{\partial y_{i}},$$ 
o\`u les coefficients $a_{i}$ et $b_{i}$ sont des fonctions lisses. 
Par un calcul simple, on a que $[V,V'] \in \mathcal{V}_{0}$, pour tout $V, V' \in \mathcal{V}_{0}.$ De plus, sur un voisinage ouvert de $m$, l'alg\`{e}bre de Lie $\mathcal{V}_{0}$ est engendr\'ee par $\rho_{1}...\rho_{k} \frac{\partial }{\partial \rho_{i}}$, $i \in \{1,...,l \}$ et $ \rho_{1}...\rho_{k} \frac{\partial }{\partial y_{i}}$, $i \in \{l,...,n-l \}$. Ainsi, $\mathcal {V}_{0}$ est $C^\infty$module projectif.
\end{exemp}
\begin{exemp}
Il y a plusieurs exemples de vari\'et\'es \`a bord avec une structure de Lie \'evanescente \`a l'infini.
On prend par exemples, $\mathcal{V}_{0}$ ci\texttt{-}haut, $\mathcal{V}_{sc}$ (Exemple~\ref{scs}) et $\mathcal{V}_{de}$ apparue dans~\cite{lauter2001fredholm}. Un autre exemple important est $\mathcal{V}_{\mathbb{K}\theta}$ (voir l'article~\cite{epstein1991resolvent}), le cas o\`u $\mathbb{K}=\mathbb{C}$ est exactement l'Exemple~\ref{epst}. 
\end{exemp} 
Maintenant, nous allons restreindre notre attention \`a des structures de Lie qui induisent dans un sens pr\'ecis des structures de fibr\'es au bord. Cette restriction sera tr\`{e}s utile par la suite. De plus, elle a le m\'erite d'\^etre tr\`{e}s flexible, puisqu'elle inclut tous les exemples de structures de Lie d\'ecrits pr\'ec\'edemment.
 \begin{defi}
Soient $M$ et $S$ deux vari\'et\'es \`a coins compactes. Un fibr\'e d'alg\'ebro\"ides de Lie associ\'e \`a un fibr\'e $\phi : M \rightarrow S$ est un alg\'ebro\"ide de Lie $A$ tel que son ancre $\varrho: A \rightarrow TM$ a son image dans le tangent vertical $T(M/S) \subset T(M)$. 
\end{defi}
\begin{exemp}
Soient $S, Z$ deux vari\'et\'es \`a coins compactes et $\pr: S\times Z \rightarrow Z$ la projection canonique. Soit $\pi : A \rightarrow Z$ un alg\'ebro\"ide de Lie sur $Z$ avec $\varrho: A \rightarrow TZ$ l'ancre associ\'ee. Le fibr\'e d'alg\'ebro\"ides de Lie trivial sur $S\times Z$ induit par $A$ est l'alg\'ebro\"ide de Lie $S\times A$  donn\'e par la projection
$\pi'=\Id \times \pi : S \times A \rightarrow  S\times Z$ 
avec l'ancre associ\'ee $\pr^{*}\varrho  : S \times A \rightarrow   T (S \times Z)$ d\'efinie par $$\pr^{*}\varrho (s,a)=(0,\varrho(a)) \in T_{s}S\times T_{\pi(a)}Z = T_{\pi'(s,a)}(S\times Z).$$
\end{exemp}
 \begin{defi} \label{lpalp}
Un fibr\'e de structures de Lie \`a l'infini associ\'e \`a un fibr\'e $\phi : M \rightarrow S$ est un fibr\'e d'alg\'ebro\"ides de Lie qui induit par restriction une structure de Lie \`a l'infini sur chaque fibre telle que pour chaque $s\in S,$ il existe un voisinage ouvert $U$ et une trivialisation $\Psi_{U} : \phi^{-1}(U) \rightarrow U \times Z$ qui induit un isomorphisme de fibr\'es d'alg\'ebro\"ides de Lie, o\`u $U\times Z$ est muni du fibr\'e d'alg\'ebro\"ides de Lie trivial induit par un certain alg\'ebro\"ide de Lie sur $Z.$ On d\'enotera un tel fibr\'e de structures de Lie \`a l'infini par $(M, S, \phi, \mathcal{V})$, o\`u $\mathcal{V} \subset C^{\infty}(M; T(M/S))$ est l'alg\`ebre de Lie des champs de vecteurs verticaux induite par l'application d'ancrage. 
 \end{defi} 
 \begin{defi}  \label{aaahana}
Soit $M$ une vari\'et\'e \`a coins compacte. Une structure de Lie fibr\'ee \`a l'infini $(M, \mathcal{V}_{SF})$ est une structure de Lie \`a l'infini telle que pour toute face ferm\'ee $F,$ la restriction $\mathcal{V}_{SF}$ sur $F$, not\'ee $\mathcal{V}_{F},$ induit un fibr\'e de structures de Lie \`a l'infini associ\'e \`a un certain fibr\'e $\phi_{F} : F \rightarrow S_{F}.$
\end{defi} 
\begin{prop} \label{asd}
Si $(M, \mathcal{V}_{SF})$ est une structure de Lie fibr\'ee \`a l'infini, alors :
\begin{enumerate}
\item Si $F_{1}$ et $F_{2}$ sont deux faces ferm\'ees de $M$ avec $F_{1} \subset F_{2},$ alors il existe une submersion surjective $\phi_{12}$ de $S_{F_{1}}$ dans $\phi_{F_{2}}(F_{1})$ telle que ${\phi_{F_{2}}}_{\scriptscriptstyle{\vert F_{1}}}   :=\phi_{12} \circ  \phi_{F_{1}};$ \label{chleka}
\item Pour toute face ferm\'ee $F$ de $M,$ il existe un morphisme surjectif d'alg\'ebro\"ides de Lie $$\displaystyle \pi_{F} : {}^{\mathcal {V}_{SF}} T_{F}M \twoheadrightarrow {}^{\mathcal {V}_{F}} TF,$$ \label{almiir}
o\`u ${}^{\mathcal {V}_{SF}} T_{F}M= {}^{\mathcal {V}_{SF}} TM_{\scriptscriptstyle{\vert F}}.$
\end{enumerate}
\end{prop}
\begin{proof}
On a que, pour $F_{1} \subset F_{2},$ 
${\mathcal {V}_{F_{2}}}_{\scriptscriptstyle{\vert F_{1}}}={{\mathcal {V}_{SF}}_{\scriptscriptstyle{\vert F_{2}}}}_{\scriptscriptstyle{\vert F_{1}}}=\mathcal {V}_{ F_{1}},$ donc les \'el\'ements de $\mathcal {V}_{F_{1}}$ doivent aussi \^etre tangents au fibres de ${\phi_{F_{2}}}_{\scriptscriptstyle{\vert F_{1}}},$ de sorte que chaque fibre de $\phi_{F_{1}}$ doit \^etre incluse dans une fibre de ${\phi_{F_{2}}}_{\scriptscriptstyle{\vert F_{1}}}$. En effet, si tel n'\'etait pas le cas, il existerait $s_{1}\in S_{1}$, $p\in \phi^{-1}_{F_{1}}(s_{1})$ et $\xi \in {}^{{\mathcal{V}_{F}}_{1}}T_{p}F_{1} \subset {}^{{\mathcal{V}_{F}}_{2}}T_{p}F_{2}$ transverse \`a la fibre de $\phi_{F_{2}}: F_{2} \rightarrow S_{F_{2}}$ passant par $p,$ contredisant la d\'efinition de $\phi_{F_{2}}.$ Pour la propri\'et\'e~\ref{chleka} de la Proposition~\ref{asd},
on peut donc d\'efinir $\phi_{12}(s_{1})$ comme \'etant l'\'el\'ement $s_{2} \in {\phi_{F_{2}}}_{\scriptscriptstyle{\vert F_{1}}} \subset S_{2}$ tel que $\phi_{F_{1}}^{-1}(s_{1}) \subset \phi_{F_{2}}^{-1}(s_{2}).$
Pour la propri\'et\'e~\ref{almiir} de la Proposition~\ref{asd}, 
on sait que ${\mathcal {V}_{SF}}_{\scriptscriptstyle{\vert F}}=\mathcal {V}_{F}$ et ${\varrho_{\mathcal{V}_{SF}}}_{\scriptscriptstyle{\vert F}}:= \varrho_{\mathcal{V}_{F}}$ induisant le diagramme commutatif :  
\begin{displaymath}
\xymatrix@+1pc{ \Gamma({}^{\mathcal {V}_{SF}} TM) 
 \ar[r]^{{(\pi_{F})}_{\Gamma}} \ar[rd]_{{(\varrho_{\mathcal{V}_{SF}})}_{\Gamma}}& \Gamma({}^{\mathcal {V}_{F}} TF) \ar[d]^{ {(\varrho_{\mathcal{V}_{F}})}_{\Gamma}}\\
   & \Gamma(TF).
  }
\end{displaymath} 
o\`u ${(\pi_{F})}_{\Gamma}$ est le morphisme de restriction de structures de Lie \`a l'infini qui pr\'eserve le $C^{\infty}$module projectif. Par le th\'eor\`{e}me Serre\texttt{-}Swan~\cite{karoubi2008k}, ce morphisme induit un morphisme surjectif d'alg\'ebro\"ides de Lie \\$\pi_{F} : {}^{\mathcal {V}_{SF}} T_{F}M \twoheadrightarrow {}^{\mathcal {V}_{F}} TF.$ 
\end{proof}
\begin{remq}
Toute structure de Lie \'evanescente \`a l'infini 
est un cas particulier de structure de Lie fibr\'ee \`a l'infini en prenant $\phi_{F} : F \rightarrow F$ l'application d'identit\'e pour toute face $F$, puisque $ \mathcal {V}_{SF_{\scriptscriptstyle{\vert F}}} = 0$. 
\end{remq}
\begin{defi}  \label{trivill}
Un fibr\'e de structures de Lie fibr\'ees \`a l'infini par rapport \`a $\phi : M \rightarrow S$ est un fibr\'e de structures de Lie \`a l'infini induisant par restriction une structure de Lie fibr\'ee \`a l'infini sur chaque fibre de sorte que pour chaque face $F$ de $M$ telle que $\phi_{\scriptscriptstyle{\vert F}} : F \rightarrow S$ est surjectif :
\begin{enumerate}
\item on a un fibr\'e $\phi_{F} : F \rightarrow S_{F}$ et une submersion surjective $\alpha_{F} :  S_{F} \rightarrow S$ de sorte que $\alpha_{F} \circ \phi_{F} = \phi_{\scriptscriptstyle{\vert F}};$ \label{cndi}
\item pour tout $s\in S,$ le fibr\'e sur $(\phi_{\scriptscriptstyle{\vert F}})^{-1} (s)$ induit par la structure de Lie fibr\'ee \`a l'infini sur $\phi^{-1}(s),$ est pr\'ecis\'ement la restriction de $\phi_{F}$ \`a $(\phi_{\scriptscriptstyle{\vert F}})^{-1} (s).$ \label{qwq}
\end{enumerate}
\end{defi}
\begin{remq}
Une structure de Lie fibr\'ee \`a l'infini $(M, \mathcal {V}_{SF})$ aussi est un fibr\'e de structures de Lie fibr\'ees \`a l'infini $(M, S, \phi_{M}, \mathcal{V})$, associ\'e \`a la fibration $\phi_{M} : M \rightarrow S$ où $S$ est un point. En fait, un fibr\'e de structures de Lie associ\'e  \`a un fibr\'e $\phi_{M} : M \rightarrow S$ engendre une structure de Lie \`a l'infini $\mathcal{V}$ sur $M$ si et seulement si $S$ est un point, puisque $M$ doit \^etre \'egale  \`a la fibre. Et alors $\mathcal{V}= \mathcal{V}_{SF}$. 
\end{remq}
\begin{prop}\label{muni}
Si $(M, \mathcal {V}_{SF})$ est une structure de Lie fibr\'ee \`a l'infini alors $(F, \mathcal {V}_{F})$ est un fibr\'e de structures de Lie fibr\'ees \`a l'infini.
\end{prop}
 \begin{proof}
Il suffit de v\'erifier que chaque fibre de $\phi_{F} : F \rightarrow S_{F},$ munie de la restriction $\mathcal {V}_{F},$ est une structure de Lie fibr\'ee \`a l'infini sur chaque fibre. 
Prenant $s\in S_{F},$ on note par $\mathcal {V}_{F_{s}}$ la structure de Lie \`a l'infini de fibre $F_{s}:=\phi^{-1}_{F}(s)$ induite par restriction.
Soit $\Omega$ une face ferm\'ee de $F$ et posons $\Omega_{s}:= \Omega \cap F_{s}.$ Alors soit $\Omega_{s}= F_{s},$ soit $\Omega_{s}$ est une face ferm\'ee de $F_{s}$ de codimension strictement positive. Dans ce dernier cas, la propri\'et\'e~\ref{chleka} de la Proposition~\ref{asd} appliqu\'ee \`a $F_{1}= \Omega$ et $F_{2}=F$ nous donne une submersion surjective $\alpha_{\Omega} : S_{\Omega} \rightarrow S_{F}$ telle que ${\phi_{F}}_{\scriptscriptstyle{\vert \Omega}} = \alpha_{\Omega} \circ \phi_{\Omega} $ comme dans~\ref{cndi} de la d\'efinition pr\'ec\'edente de sorte que ${\phi_{\Omega}}_{\scriptscriptstyle{\vert \Omega_{s}}}$ est le fibr\'e sur $\Omega_{s}$ qui conf\`{e}re \`a $F_{s}$ une structure de Lie fibr\'ee \`a l'infini. 
\end{proof}
\begin{exemp}
L'Exemple~\ref{mazzza} est une structure de Lie fibr\'ee \`a l'infini, ce qui n'est pas le cas pour la structure $\mathcal {V}_{\mathcal{F}}$ de l'Exemple~\ref{feuill} lorsque $\mathcal{F}$ est un feuilletage qui ne provient pas d'une structure de fibr\'e. 
\end{exemp}
\begin{exemp}
La structure de Lie \`a l'infini $\mathcal {V}_{b}$ est une structure de Lie fibr\'ee \`a l'infini telle que chaque fibr\'e $\phi_{F}: F \rightarrow \{0\}$ est une projection canonique de la base $F$ dans un point. De plus, on a une suite exacte 
\begin{equation} 
0\rightarrow {}^{b} NF \hookrightarrow {}^{b} T_{F}(M) \twoheadrightarrow {}^{b} TF \rightarrow 0,
\end{equation} 
o\`u  ${}^{b} T_{F}(M)={}^{b} T(M)_{\scriptscriptstyle{\vert F}}$ et ${}^{b} NF :=\Ker(\pi)$ est le noyau de la projection canonique $\pi :  {}^{b} T_{F}(M) \twoheadrightarrow {}^{b} TF$ par restriction. 
\end{exemp}
\begin{exemp}
Dans l'Exemple~\ref{jaw}, on suppose que $M$ est une vari\'et\'e compacte \`a coins fibr\'es de profondeur $l$ associ\'ee \`a une famille $(H_{i})_{i \in \{1,...,l\}}$ d'hypersurfaces bordantes.
On peut voir que chaque fibre de $\pi_{i} : H_{i} \rightarrow S_{i}$ est une vari\'et\'e \`a coins fibr\'es. Pr\'ecis\'ement, si $s\in S_{i}$ alors $H_{i,s}:=\pi^{-1}_{i}(s)$ est une vari\'et\'e compacte \`a coins fibr\'es, o\`u les hypersurfaces bordantes sont $H_{i,s} \cap H_{j}$ avec $H_{i}<H_{j}$ et les fibrations associ\'ees sont obtenues par restriction des $\pi_{j}.$
Par cons\'equent, la restriction d'un QFB\texttt{-}champ de vecteurs de $(\mathcal {V}_{QFB}, M)$ sur chaque fibre $H_{i,s}:=\pi^{-1}_{i}(s)$ nous donne aussi un QFB\texttt{-}champ de vecteurs sur la fibre.
Maintenant, soit $F:=H_{1}\cap...\cap H_{k}$ une face ferm\'ee de profondeur $k\in \{1,...,l\}$ de sorte que $H_{1}<H_{2}<...<H_{k},$ apr\`{e}s un r\'e\texttt{-}\'etiquetage, si n\'ecessaire. Il suffit de prendre $\phi_{F}:={\pi_{k}}_{\scriptscriptstyle{\vert F}}$ pour voir que la structure $\mathcal {V}_{QFB}$ est une structure de Lie fibr\'ee \`a l'infini, pourvu bien sûr que pour chaque fibr\'e la condition de trivialisation locale de D\'efinition~\ref{lpalp} soit satisfaite.
\end{exemp}
\begin{prop}
Le produit cart\'esien de deux structures de Lie fibr\'ees \`a l'infini est une structure de Lie fibr\'ee \`a l'infini. 
\end{prop}
\begin{proof}
\'Etant donn\'ees $(M_{1}, \mathcal {V}_{1})$ et $(M_{2}, \mathcal {V}_{2})$ deux structures de Lie fibr\'ees \`a l'infini, leur produit cart\'esien est clairement une structure de Lie \`a l'infini $$\big(M=M_{1} \times M_{2},\mathcal {V} = \Gamma( \pr^{*}_{1}  {}^{\mathcal{V}_{1}}TM_{1})  \otimes_{C^{\infty}(M)} \Gamma(\pr^{*}_{2} {}^{\mathcal{V}_{2}}TM_{2} )\big),$$ o\`u $\pr_{i} : M_{1} \times M_{2} \rightarrow M_{i}$ est la projection canonique. Pour chaque face ferm\'ee $F:=F_{1} \times F_{2}$ de $M,$ il suffit alors de prendre$$\phi_{F}:=\phi_{F_{1}} \times \phi_{F_{2}} : F \rightarrow S_{F}:=S_{F_{1}}\times S_{F_{2}}$$ avec $S_{F_{i}}$ un point lorsque $F_{i}=M_{i}.$
\end{proof}
La Proposition~\ref{muni} peut \^etre g\'en\'eralis\'ee de la façon suivante aux fibr\'es de structures de Lie fibr\'ees \`a l'infini.
\begin{prop}\label{munim}
Soit $(M, S, \phi_{M}, \mathcal{V})$ un fibr\'e de structures de Lie fibr\'ees \`a l'infini associ\'e au fibr\'e $\phi_{M} : M \rightarrow S.$ Soit $F_{1}$ une face ferm\'ee de $M$. Soit $F_{2}$ la plus petite face ferm\'ee de $M$ contenant $F_{1}$ telle que $\phi_{M}(F_{2})=S$.  Soient $\phi_{F_{2}} : F_{2}\rightarrow S_{F_{2}}$ et $\alpha_{F_{2}}: S_ {F_{2}}\rightarrow S$ les applications données par la D\'efinition~\ref{trivill}. Soit $\mathcal {V}_{F_{1}}$ la restriction de $\mathcal {V}$ \`a $F_{1}$ et posons finalement 
$$\phi_{F_{1}} := {\phi_{F_{2}}}_{\scriptscriptstyle{\vert F_{1}}},\; S_{F_{1}}=\phi_{F_{2}}(F_{1}).$$
Alors $(F_{1}, S_{F_{1}}, \phi_{F_{1}}, \mathcal {V}_{F_{1}})$ est un fibr\'e de structures de Lie fibr\'ees \`a l'infini.
\end{prop}
 \begin{proof}
 On doit consid\'erer trois cas.
\begin{itemize}
 \item[Cas 1 :] $ F_{1}=\phi^{-1}_{M}(S_{1})$ pour $S_{1}$ une face ferm\'ee de $S$. Dans ce cas,  $\phi_{F_{1}}={\phi_{M}}_{\scriptscriptstyle{\vert F_{1}}}$ et $S_{F_{1}}=S_{1}.$ La structure de fibré de structures de Lie fibr\'ees \`a l'infini de $(F_{1}, S_{F_{1}}, \phi_{F_{1}}, \mathcal {V}_{F_{1}})$ est donc directement induite par restriction de celle de $(M, S, \phi_{M}, \mathcal{V})$.
  \item[Cas 2 :] $F_{1}=F_{2}$, c'est-\`a-dire que $\phi_{M}(F_{1})=S$. Dans ce cas, $(F_{1}, S_{F_{1}}, \phi_{F_{1}}, \mathcal {V}_{F_{1}})$ est clairement un fibr\'e de structure de Lie \`a l'infini. De plus, par la Proposition~\ref{muni} appliqu\'ee \`a chaque fibre de $\phi_{M} : M \rightarrow S$, on voit que les structures de Lie \`a l'infini des fibres de   $\phi_{F_{1}}$ sont bien fibr\'ees \`a l'infini. Il reste donc \`a v\'erifier les conditions~\ref{cndi} et~\ref{qwq} de la D\'efinition~\ref{trivill}. Soit donc $F_{0}$ une face ferm\'ee de $F_{1}$ telle que $\phi_{F_{1}}(F_{0})=S_{F_{1}}.$ Comme ${\phi_{M}}_{\scriptscriptstyle{\vert F_{1}}} = \alpha_{F_{1}} \circ \phi_{F_{1}} : F_{1} \rightarrow S$ est surjective, il en est de m\^eme  pour $$\alpha_{F_{1}}  : S_{F_{1}} \rightarrow S.$$ Ainsi, on aura que $$\phi_{M}(F_{0})=\alpha_{F_{1}} \circ \phi_{F_{1}} (F_{0})=\alpha_{F_{1}}(S_{F_{1}})= S.$$ Il existe donc des applications $\phi_{F_{0}} : F_{0}\rightarrow S_{F_{0}}$ et $\alpha_{F_{0}} : S_{F_{0}} \rightarrow S$ satisfaisant aux conditions~\ref{cndi} et~\ref{qwq} de la D\'efinition~\ref{trivill} pour le fibr\'e de structures de Lie fibr\'ees \`a l'infini $(M, S, \phi_{M}, \mathcal{V}).$ En appliquant la Proposition~\ref{asd} aux faces ferm\'ees $F_{0}\cap \phi^{-1}(s)$ et $F_{1}\cap \phi^{-1}(s)$ pour chaque $s\in S,$ on voit aussi qu'il existe une subersion surjective $\alpha_{F_{0}F_{1}} : S_{F_{0}} \rightarrow S_{F_{1}}$ telle que $$\alpha_{F_{0}F_{1}} \circ \phi_{F_{0}} = {\phi_{F_{1}}}_{\scriptscriptstyle{\vert F_{0}}};$$
de sorte que les applications $\phi_{F_{0}}$ et $\alpha_{F_{0}F_{1}}$ satisfont \`a la condition~\ref{cndi} de la D\'efintion~\ref{trivill} pour \\$(F_{1}, S_{F_{1}}, \phi_{F_{1}}, \mathcal {V}_{F_{1}}).$ D'autre part, comme $\phi_{F_{0}}$ satisfait \`a la condition~\ref{qwq} pour $(M, S, \phi_{M}, \mathcal{V}),$ on voit par restriction qu'elle est aussi automatiquement satisfaite pour $(F_{1}, S_{F_{1}}, \phi_{F_{1}}, \mathcal {V}_{F_{1}}).$ Cela montre que $(F_{1}, S_{F_{1}}, \phi_{F_{1}}, \mathcal {V}_{F_{1}})$ est bien un fibr\'e de structures de Lie fibr\'ees \`a l'infini.
   \item[Cas 3 :]
 En g\'en\'eral, si $F_{1}$ ne tombe ni dans le cas 1, ni dans le cas 2, alors par le cas 2, on sait au moins que $(F_{2}, S_{F_{2}}, \phi_{F_{2}}, \mathcal {V}_{F_{2}})$ est un fibr\'e de structures de Lie fibr\'ees \`a l'infini.  Ainsi, en prenant $(F_{2}, S_{F_{2}}, \phi_{F_{2}}, \mathcal{V}_{F_{2}})$ plut\^ot que $(M, S, \phi_{M}, \mathcal{V})$ comme fibr\'e de structures de Lie fibr\'ees \`a l'infini, on peut se ramener au cas 1 pour montrer que $(F_{1}, S_{F_{1}}, \phi_{F_{1}}, \mathcal {V}_{F_{1}})$ est un fibr\'e de structures de Lie fibr\'ees \`a l'infini. 
\end{itemize}
\end{proof}
Soient $(M,\mathcal{V})$ une structure de Lie \`a l'infini et $\varrho_{\mathcal{V}}: {}^{\mathcal{V}}TM \rightarrow TM$ l'ancre associ\'ee. 
\begin{defi}
Une m\'etrique de $\overset{\circ}{M}$ compatible avec une structure de Lie \`a l'infini $(M,\mathcal{V})$ est une m\'etrique riemannienne sur $\overset{\circ}{M}$ d\'efinie par 
$$g=(\varrho_{\mathcal{V}}^{-1})^{*}(h_{\scriptscriptstyle{\vert \overset{\circ}{M}}}),$$ pour une certaine m\'etrique euclidienne $h \in \Gamma (\Sym^{2}({}^{\mathcal{V}}T^{*}M))$.
On dit alors que $(\overset{\circ}{M},g)$ est une vari\'et\'e riemannienne avec une structure de Lie \`a l'infini.
\end{defi} 
\begin{exemp} (Vari\'et\'e \`a bout asymptotiquement cylindrique).
Soit $\overset{\circ}{M}$ une vari\'et\'e riemannienne compl\`{e}te non compacte de dimension $n$ munie d'une m\'etrique compl\`{e}te $g$ telle qu'il existe un compact $K \subset \overset{\circ}{M}$ et une vari\'et\'e compacte riemannienne $(\partial M, g_{\partial M}) $ de sorte que $\overset{\circ}{M} \setminus K$ peut \^etre param\'etris\'e par un voisinage tubulaire du bord $\partial M \times (0, +\infty ) \simeq \overset{\circ}{M} \setminus K$. On suppose que le cylindre $\partial M \times (0, +\infty )$ soit muni de la m\'etrique produit $g_{cyl}= dx^{2}+ g_{\partial M}$. Par un recollement \`a l'infini de $\overset{\circ}{M} \setminus K$ par $\partial M$, on obtient une compactification $M := \overset{\circ}{M} \cup \partial M$ qui est une vari\'et\'e \`a bord compacte. Et par suite, le changement de variable $\rho=e^{-x}$ donne une fonction de d\'efinition pour $\partial M$ telle que $\frac{\partial }{\partial x}= -\rho \frac{\partial }{\partial \rho} $. Ceci sugg\`{e}re de consid\'erer $ {}^bTM$. 
On dit alors que $(\overset{\circ}{M}, g)$ est une vari\'et\'e \`a bout asymptotiquement cylindrique si, pr\`{e}s d'un voisinage tubulaire du bord $\partial M \times [0, 1],$ il existe $ \gamma>0$ tel que
$$g-(\frac{d\rho^{2}}{\rho^{2}}+g_{\partial M}) \in \rho^{\gamma} C_{b}^{\infty}(M; \Sym^{2} ({}^bT^{*}M))= \rho^{\gamma} C_{b}^{\infty}(M) \otimes_{ C^{\infty}(M)} \Gamma(\Sym^{2} ({}^bT^{*}M)),$$ o\`u
$C_{b}^{\infty}(M)=\{f \in C^{\infty}(\overset{\circ}{M})\mid \forall k \in \mathbb{N}_{0}, \;\;\{V_{1},...,V_{k}\} \subset \mathcal{V}_{b}, \; \;\sup_{\overset{\circ}{M}}|V_{1}...V_{k}f|<\infty\}.$
Lorsque $\displaystyle g-(\frac{d\rho^{2}}{\rho^{2}}+g_{\partial M}) \in \rho \,C_{b}^{\infty}(M; \Sym^{2} ({}^bT^{*}M)),$ c'est un exemple de m\'etrique compatible avec la structure de Lie \`a l'infini $\mathcal{V}_{b}.$
\end{exemp}
\begin{exemp}(Vari\'et\'e \`a bout asymptotiquement conique).
Soit $g_{b}$ une b\texttt{-}m\'etrique qui est une m\'etrique compatible avec une structure de Lie \`a l'infini $(M,\mathcal{V}_{b})$, o\`u $M$ est une vari\'et\'e \`a bord compacte. On appelle m\'etrique de diffusion une m\'etrique compl\`{e}te de $\overset{\circ}{M}$ compatible avec une structure de Lie \`a l'infini $(M,\mathcal{V}_{sc}),$ donn\'ee par $\displaystyle g_ {sc}:=\frac{g_ {b}}{\rho ^{2}}.$  On la dit une m\'etrique de diffusion de type produit s'il existe un voisinage tubulaire du bord $c: \partial M \times [0, \varepsilon ) \rightarrow M$ tel que $$c^{*}g_{sc} = \frac{d \rho^{2}} {\rho ^{4}} + \frac {g_{\partial M}}{\rho ^{2}}, \;\; g_{\partial M} \in C^{\infty}\big(\partial M; \Sym^{2}(T\partial M)\big),$$ ou de mani\`{e}re \'equivalente,  $$c^{*}g_{sc} =dt^{2} +t^{2} g_{\partial M}, \; \;t \in (-log \varepsilon, +\infty ).$$
On dit alors que $(\overset{\circ}{M}, g)$ est une vari\'et\'e \`a bout asymptotiquement conique si \\
$g-g_{p}\in \rho \,C^{\infty}\big(M; Sym^{2} ({}^{sc}T^{*}M)\big),$ o\`u $g_{p}$ est une m\'etrique de diffusion de type produit. Dans ce cas, $g$ est compatible avec $\mathcal{V}_{sc}.$
Un exemple simple est le cas euclidien: la compactification radiale de $\mathbb{R}^{n}$ avec bord la sph\`{e}re $\mathbb{S}^{n-1}$. Cette compactification est donn\'ee par la projection st\'er\'eographique $SP$ d\'efinie par $$SP: \mathbb{R}^{n} \rightarrow \mathbb{S}_{+}^{n}:=\{x=(x_{0},...,x_{n})  \in \mathbb{R}^{n+1} \mid \left|x\right|=1\; et \; x_{0} \geq 0 \}, $$ $$ SP(x)=\big(\frac{1} { (1+\left|x\right|^{2})^{\frac{1}{2}}},\frac{x}{(1+\left|x\right|^{2})^{\frac{1}{2}}}\big).$$
Puisque $\mathbb{S}_{+}^{n}$ est une vari\'et\'e \`a bord compacte, $SP$ permet d'identifier $\mathbb{R}^{n}$ \`a $( \mathbb{S}_{+}^{n} )_{0}:= \mathbb{S}_{+}^{n}\setminus \partial \mathbb{S}_{+}^{n}$. La m\'etrique euclidienne est une m\'etrique de diffusion de $\mathbb{R}^{n}$ qui est donn\'ee par 
$$ \left|dx\right|^{2}=dr^{2} + r^{2} \left|dw\right|^{2} =  \frac{d \rho^{2}} {\rho ^{4}} + \frac {\left|dw\right|^{2}}{\rho ^{2}}, \; \left|x\right|=r=\frac{1}{\rho}, w=\frac{x}{\left|x\right|}, $$ 
o\`u $\left|dw\right|^{2}$ est la m\'etrique standard de la sph\`{e}re $\mathbb{S}^{n-1} = \partial \mathbb{S}_{+}^{n}$.
\end{exemp}
\begin{prop} \label{hassna}
Deux m\'etriques $g_{1}$ et $g_{2}$ de $\overset{\circ}{M}$ compatibles avec une structure de Lie \`a l'infini $(M,\mathcal{V})$ sont bi\texttt{-}Lipschitz \'equivalentes, c'est\texttt{-}\`a\texttt{-}dire qu'il existe une constante $C>0$ telle que 
$$ C^{-1}g_{2}(X,X) \leq  g_{1}(X,X)\leq C g_{2}(X,X) \; \; \forall X \in T\overset{\circ}{M}.$$
En particulier, $ C^{-1}d_{2} \leq  d_{1}\leq C d_{2}$, o\`u $d_{i}$ est la distance sur $\overset{\circ}{M}$ correspondante \`a $g_{i}$.
\end{prop}
\begin{proof}
Sur chaque fibre de ${}^{\mathcal{V}}TM$, on a une trivialisation locale. L'\'equivalence des normes de $\mathbb{R}^{n}$ nous assure l'in\'egalit\'e escompt\'ee localement. Comme $M$ est compacte, on obtient le r\'esultat voulu globalement sur $M$, et donc sur $\overset{\circ}{M}$ par restriction.
\end{proof}
On d\'enote par $\vol_{n}$ la forme volume de la m\'etrique riemannienne d'une vari\'et\'e riemannienne $\overset{\circ}{M}$ avec une structure de Lie \`a l'infini $(M,\mathcal{V})$.
\begin{prop} (Proposition 4.1 dans ~\cite{ammann2004geometry})
Soit $f\geq 0$ une fonction lisse sur $M$. Si $ \int_{\overset{\circ}{M}} f \vol_{n} <  +\infty$ alors $f$ s'annule sur chaque hypersurface bordante de $M$. En particulier, le volume de chaque vari\'et\'e riemannienne non compacte avec une structure de Lie \`a l'infini est infini. 
\end{prop}
\begin{proof}
Sans perte de g\'en\'eralit\'e, on suppose que $M$ est une vari\'et\'e \`a bord compacte. Soit $\vol_{n}'$ une forme volume sur $M$ associ\'ee \`a une autre m\'etrique sur $M$ qui est lisse jusqu'au bord. On a alors que $\vol_{n} \geq C \rho^{-1} \vol_{n}'$ avec $\rho$ une fonction de d\'efinition du bord et $C$ une constante positive.
D'o\`u, si $f$ est non nulle sur $\partial M$, alors $$\int_{\overset{\circ}{M}} f \vol_{n} \geq \int_{\overset{\circ}{M}}  C \rho^{-1} f \vol_{n}'> + \infty.$$
\end{proof}
\begin{defi}
Soit $E \rightarrow M$ un fibr\'e vectoriel. Une ${}^{\mathcal{V}}TM$\texttt{-}connexion sur $E$, not\'ee aussi $\nabla$, est une application 
$$\nabla :  \Gamma({}^{\mathcal{V}}TM) \times \Gamma(E)  \rightarrow  \Gamma(E) \:$$  $$\: (X,\mu) \mapsto \nabla_{X} \, \mu,$$ 
qui satisfait aux propri\'et\'es suivantes:
\begin{enumerate}
\item $\nabla_{fX+gY} \, \mu = f \,\nabla_{X} \, \mu +g \, \nabla_{Y} \, \mu,  \; \; \forall f,g \in  C^\infty(M), \;  \forall X, Y \in  \Gamma({}^{\mathcal{V}}TM),$
\item $\nabla_{X} \, (a\mu + a' \mu')= a\,\nabla_{X} \, \mu +b' \, \nabla_{X} \, \mu', \; \; \forall a,a' \in  \mathbb{R},$
\item $\nabla_{X} \, (f\mu)= \varrho_{\mathcal{V}}(X)(f) \mu +f \, \nabla_{X} \, \mu, \; \; \forall f\in  C^\infty(M).$
\end{enumerate}
\end{defi}
\begin{prop} (Lemme 4.2 dans ~\cite{ammann2004geometry}) \label{kozz}
Soit $(\overset{\circ}{M}, g)$ une vari\'et\'e riemannienne avec une structure de Lie \`a l'infini $(M,\mathcal{V})$. Alors
la connexion de Levi\texttt{-}Civita sur $T\overset{\circ}{M}$, $$\nabla :  \mathfrak{X}(\overset{\circ}{M}) \times  \mathfrak{X}(\overset{\circ}{M}) \rightarrow  \mathfrak{X}(\overset{\circ}{M}),$$ s'\'etend en une connexion affine, not\'ee aussi $$\nabla :  \Gamma({}^{\mathcal{V}}TM)   \times  \Gamma({}^{\mathcal{V}}TM) \rightarrow \Gamma({}^{\mathcal{V}}TM),$$ ayant les m\^emes propri\'et\'es de la connexion usuelle de Levi\texttt{-}Civita (sym\'etrique et compatible). Autrement dit, la connexion de Levi\texttt{-}Civita sur $T\overset{\circ}{M}$ s'\'etend en une $ {}^{\mathcal{V}}TM$\texttt{-}connexion de Levi\texttt{-}Civita sur $ {}^{\mathcal{V}}TM$. 
\end{prop}
\begin{proof}
C'est une cons\'equence directe de la formule de Koszul :
$$2\langle\nabla_{X}Y,Z\rangle\:=\:\langle [X,Y],Z\rangle -\langle [Y,Z],X\rangle+\langle [Z,X],Y\rangle +X\langle Y,Z\rangle+Y\langle Z,X\rangle-Z\langle X,Y\rangle,$$$ \forall X, Y \in  \Gamma({}^{\mathcal{V}}TM).$
\end{proof}
\begin{remq}
Le m\^eme raisonnement implique que la connexion de Levi\texttt{-}Civita sur $T^{*}\overset{\circ}{M}$ (respectivement  $(\otimes^{k} \;T^{*}\overset{\circ}{M}) \otimes (\otimes^{l} \;T\overset{\circ}{M})$) s'\'etend en une $ {}^{\mathcal{V}}TM$\texttt{-}connexion sur $ {}^{\mathcal{V}}T^{*}M$ (respectivement $(\otimes^{k} \;  {}^{\mathcal{V}}T^{*}M) \otimes (\otimes^{l} \;  {}^{\mathcal{V}}TM) $). 
\end{remq}
\begin{cor} (Corollaire 4.3 dans~\cite{ammann2004geometry}) \label{nidhal}
Soient $(\overset{\circ}{M}, g)$ une vari\'et\'e riemannienne avec une structure de Lie \`a l'infini $(M,\mathcal{V})$ et $\nabla$ la $ {}^{\mathcal{V}}TM$\texttt{-}connexion de Levi\texttt{-}Civita. $\forall X, Y \in  \Gamma({}^{\mathcal{V}}TM)$, l'endomorphisme de courbure de Riemann $R(X,Y)$ s'\'etend en un endomorphisme sur $  {}^{\mathcal{V}}TM$. De plus, chaque d\'eriv\'ee contravariante, $$\nabla^{k}R \in  \Gamma\big((\otimes^{k}\; T^{*}\overset{\circ}{M} )\otimes (\wedge^{2}\;T^{*}\overset{\circ}{M})\otimes \End(T\overset{\circ}{M})\big),$$ est born\'ee sur $\overset{\circ}{M}$ et s'\'etend en une section de $$(\otimes^{k} \; {}^{\mathcal{V}}T^{*}M) \otimes (\wedge^{2} \; {}^{\mathcal{V}}T^{*}M)  \otimes \End( {}^{\mathcal{V}}TM).$$
\end{cor}
\begin{proof}
Soit $ X, Y \in \Gamma({}^{\mathcal{V}}TM)$. Par d\'efinition, on sait que 
$$R(X,Y):=\nabla_{X} \nabla_{Y}-\nabla_{Y} \nabla_{X} - \nabla_{[X,Y] }.$$
On a que $\nabla_{X}$, $\nabla_{Y}$ et $\nabla_{[X,Y]}$ sont des op\'erateurs diff\'erentiels sur ${}^{\mathcal{V}}TM$, par la Proposition~\ref{kozz}. Alors $R(X, Y)$ est un op\'erateur diff\'erentiel sur $ {}^{\mathcal{V}}TM$, qui est un tenseur sur $ T\overset{\circ}{M}$. Puisque $\overset{\circ}{M}$ est dense dans $M$, on a bien que $$R\in \Gamma \big(( \wedge^{2}\; {}^{\mathcal{V}}T^{*}M  )\otimes \End( {}^{\mathcal{V}}TM)\big).$$ On applique la $ {}^{\mathcal{V}}TM$\texttt{-}connexion de Levi\texttt{-}Civita sur $\otimes ^{k} \; {}^{\mathcal{V}}T^{*}M$ pour obtenir 
$$\nabla^{k}R \in \Gamma\big((\otimes ^{k}\; {}^{\mathcal{V}}T^{*}M )\otimes (\wedge^{2} \; {}^{\mathcal{V}}T^{*}M) \otimes \End( {}^{\mathcal{V}}TM)\big).$$
Le contr\^ole de $\nabla^{k}R$ r\'esulte du fait que $M$ est compacte.
\end{proof}

Avant de conclure, nous rappelons qu'un exemple de g\'eom\'etrie born\'ee est une vari\'et\'e riemannienne compl\`{e}te dont le rayon d'injectivit\'e est strictement positif et les d\'eriv\'ees contravariantes du tenseur de courbure sont born\'ees. Ainsi, toute vari\'et\'e avec une structure de Lie \`a l'infini est un exemple de g\'eom\'etrie born\'ee d\`{e}s que le rayon d'injectivit\'e est strictement positif. Par un r\'esultat de Amann, Lauter et Nistor (voir le Corollaire 4.20~\cite{ammann2004geometry}), \'etant donn\'e une structure de Lie \`a l'infini, soit le rayon d'injectivit\'e est strictement positif pour toutes les m\'etriques compatibles, soit il ne l'est pour aucune. De plus, Ammann, Lauter et Nistor donnent un crit\`{e}re suffisant pour que le rayon d'injectivit\'e soit strictement positif. Comme, il n'y a aucun exemple connu de m\'etrique compatible avec une structure de Lie \`a l'infini ayant un rayon d'injectivit\'e nul, il est conjectur\'e que le rayon d'injectivit\'e est toujours strictement positif pour une m\'etrique compatible avec une structure de Lie \`a l'infini, voir la Conjecture 4.11~\cite{ammann2004geometry}. Cette conjecture n'a pas encore \'et\'e prouv\'ee sauf pour des m\'etriques particuli\`{e}res comme les m\'etriques QAC et QFB (voir la Proposition 1.27~\cite{conlon2019quasi}). 
\section{\'equation parabolique et polyhomog\'en\'eit\'e des solutions} \label{}
Dans la pr\'esente section, nous g\'en\'eralisons les travaux de~\cite{rochon2015polyhomogeneite} pour certaines \'equations $\mathcal{V}_{SF}$\texttt{-}paraboliques lin\'eaires, \`a savoir des \'equations paraboliques lin\'eaires d\'etermin\'ees par une structure de Lie fibr\'ee \`a l'infini. Plus pr\'ecis\'ement, nous montrons que les solutions de ces \'equations admettent un d\'eveloppement polyhomog\`{e}ne \`a l'infini pourvu que certaines conditions naturelles soient satisfaites.  

Dans ce qui suit, soit $(\overset{\circ}{M}, g)$ une vari\'et\'e riemannienne non compacte compl\`{e}te avec une structure de Lie \`a l'infini $(M,\mathcal{V})$. La connexion de Levi\texttt{-}Civita $\nabla$ sur $T\overset{\circ}{M}$ induit une connexion de Levi\texttt{-}Civita sur $\Gamma(T^{(k,l)}\overset{\circ}{M}):= \Gamma\big((\otimes ^{k} \; T^{*} \overset{\circ}{M})\otimes( \otimes ^{l} \; T\overset{\circ}{M})\big)$, not\'ee aussi $\nabla$. Plus g\'en\'eralement, si $E \rightarrow M$ est un fibr\'e vectoriel euclidien avec une m\'etrique induite sur $\Gamma(T^{(k,l)}\overset{\circ}{M} \otimes E)$, not\'ee aussi $g$, alors un choix de $ {}^{\mathcal{V}}TM$\texttt{-}connexion pour $E$ et la connexion de Levi\texttt{-}Civita sur $\Gamma(T^{(k,l)}\overset{\circ}{M})$ induisent une connexion sur $\Gamma(T^{(k,l)}\overset{\circ}{M}\otimes E) $, not\'ee aussi $\nabla$. 
\begin{defi}
On d\'efinit l'espace de Banach $C_{\mathcal{V}}^{k}(\overset{\circ}{M}; E)$ par l'ensemble des sections de classe $C^{k}$ de $E$ telles que toutes ses d\'eriv\'ees contravariantes jusqu'\`a l'ordre $k$ sont uniform\'ement born\'ees, c'est\texttt{-}\`a\texttt{-}dire que $$C_{\mathcal{V}}^{k}(\overset{\circ}{M}; E) :=\big\{ \mu \in C^{k}(\overset{\circ}{M}; E) \mid \sup_{m \in \overset{\circ}{M}}{\left|\nabla^{j} \mu\right|}_{g} < +\infty\;\forall j\in\{0,...,k\}\big\}.$$ 
Il est muni de la norme d\'efinie par
$$\left\| \mu \right\|_{k}= \sum \limits_{j=0}^{k} \sup_{m \in \overset{\circ}{M}}{\left|\nabla^{j} \mu\right|}_{g}.$$
L'espace de Fr\'echet associ\'e est d\'efini par $$C_{\mathcal{V}}^{\infty}(\overset{\circ}{M}; E)=\bigcap _{k \in \mathbb{N}} C_{\mathcal{V}}^{k}(\overset{\circ}{M}; E)$$
et a pour semi\texttt{-}normes $\left\| \cdot \right\|_{k}$ pour $k \in \mathbb{N}$.
\end{defi}
\begin{remq} \label{lisse}
On a que $C^{\infty} (M;E) \subsetneq C_{\mathcal{V}}^{\infty}(\overset{\circ}{M};E)$. Autrement dit, chaque section de $C_{\mathcal{V}}^{k}(\overset{\circ}{M};E)$ n'exige que le contr\^ole de toutes ses d\'eriv\'ees contravariantes, mais pas qu'elle s'\'etend continûment au bord. Par exemple, 
si $\rho_{H}$ est une fonction de d\'efinition d'une hypersurface bordante $H$ de $M$ telle que $\rho_{H}(m)<1, \; \forall m \in M$, alors la fonction $(\log \rho_{H})^{-\eta},$ o\`u $\eta \in \mathbb{N},$ n'est pas une fonction lisse sur $M$, mais est un \'el\'ement de $C_{\mathcal{V}}^{\infty}(\overset{\circ}{M})$. 
\end{remq}
\begin{defi}
\'Etant donn\'e $\alpha \in (0,1]$, on d\'efinit l'espace de H\"older $C_{\mathcal{V}}^{0,\alpha}(\overset{\circ}{M}; E)$ comme \'etant l'espace de Banach $$\{ \mu \in C_{\mathcal{V}}^{0}(\overset{\circ}{M}; E) \mid \left\| \mu \right\|_{0,\alpha}  < +\infty\}$$ avec norme $\left\| \cdot \right\|_{0,\alpha}$ donn\'ee par $$\left\| \mu \right\|_{0,\alpha} : = \left\| \mu \right\|_{0} +\sup \big\{\frac{\left| \psi_{\gamma}(\mu(\gamma(0)))-\mu(\gamma(1))\right|_{g} }{l(\gamma)^{\alpha}} \mid \gamma \in C^{\infty} ([0,1];\;\overset{\circ}{M})\;et\; \gamma(0)\neq \gamma(1)\big\}$$
o\`u $\psi_{\gamma} : E_{\scriptscriptstyle{\vert \gamma(0) }} \rightarrow E_{\scriptscriptstyle{\vert \gamma(1) }}$ est le transport parall\`{e}le le long de $\gamma$ et $l(\gamma)$ est la longueur de $\gamma$ associ\'ee \`a la m\'etrique $g$.\\ 
Pour $k \in  \mathbb{N}$, l'espace de H\"older $C_{\mathcal{V}}^{k,\alpha}(\overset{\circ}{M}; E)$ est l'espace de Banach d\'efini par $$\{ \mu \in C_{\mathcal{V}}^{k}(\overset{\circ}{M}; E) \mid \nabla^{k} \mu \in  C_{\mathcal{V}}^{0,\alpha}(\overset{\circ}{M}; T^{(k,0)}\overset{\circ}{M} \otimes E) \}$$ avec norme donn\'ee par $$\left\| \mu \right\|_{k,\alpha} : = \left\| \mu \right\|_{k-1} + \left\| \nabla^{k} \mu \right\|_{0,\alpha}. $$
\end{defi}
\begin{remq}
Par la Proposition~\ref{hassna} et le Corollaire~\ref{nidhal}, la d\'efinition de ces espaces et leur topologie d\'ependent seulement de la structure de Lie \`a l'infini,  pas du choix de m\'etrique compatible. Les normes cependant d\'ependent clairement du choix de la m\'etrique $g.$
\end{remq} 
Maintenant, voici une version parabolique de ces espaces. D\'enotons aussi par $E$ le tir\'e en arri\`{e}re de $E$ par la projection $[0, T ] \times M \rightarrow M$. 
\begin{defi} L'espace de Banach $C_{\mathcal{V}}^{k}([0,T]\times \overset{\circ}{M}; E)$ est d\'efini par l'ensemble 
\begin{multline*}
\big\{ \mu \in C^{k}([0,T]\times\overset{\circ}{M}; E) \mid \forall i,j \in \mathbb{N}_{0} \; avec \; 2i+j \leq k, \\\; (\frac{\partial}{\partial t})^{i} \nabla^{j} \mu \in  C^{0}([0,T]\times \overset{\circ}{M}; T^{(j,0)}\overset{\circ}{M}\otimes E) \; et \; \sup_{t \in [0,T]}\sup_{m \in \overset{\circ}{M}}{\left| (\frac{\partial}{\partial t} )^{i} \nabla^{j} \mu (t,m)\right|}_{g} < +\infty\big\},
\end{multline*}
avec norme donn\'ee par $$ \left\| \mu \right\|_{k} : =  \sum \limits_{2i+j \leq k} \sup_{t \in [0,T]} \sup_{m \in \overset{\circ}{M}}{\left| (\frac{\partial}{\partial t} )^{i} \nabla^{j} \mu (t,m) \right|}_{g}.$$
L'espace de Fr\'echet associ\'e est d\'efini par $$C_{\mathcal{V}}^{\infty}([0,T]\times \overset{\circ}{M}; E)=\bigcap _{k \in \mathbb{N}_{0} } C_{\mathcal{V}}^{k}([0,T]\times \overset{\circ}{M}; E).$$
\'Etant donn\'e $\alpha \in (0,1],$ l'espace de H\"older parabolique $C_{\mathcal{V}}^{0,\alpha}([0,T]\times \overset{\circ}{M}; E)$ est l'espace de Banach d\'efini par $$\{ \mu \in C_{\mathcal{V}}^{0}([0,T]\times \overset{\circ}{M}; E) \mid \left\| \mu \right\|_{0,\alpha}  < +\infty \}$$ avec norme $ \left\| \cdot \right\|_{0,\alpha}$ donn\'ee par
\begin{multline*}
\left\| \mu \right\|_{0,\alpha} : = \left\| \mu \right\|_{0} 
         +\sup_{t \in [0,T]}\sup \big\{\frac{\left| \psi_{\gamma}(\mu(\gamma(0)))-\mu(\gamma(1))\right|_{g} }{l(\gamma)^{\alpha}} \mid \gamma \in C^{\infty} ([0,1];\;\{t\} \times \overset{\circ}{M})\;et\; \gamma(0)\neq \gamma(1)\big\}\\
         +\sup_{m \in \overset{\circ}{M}}  \sup_{t \neq t'} \frac{\left|\mu(t,m)-\mu(t',m)\right|_{g}}{\left| t-t'\right|^{\frac{\alpha}{2}}}. 
\end{multline*}
Pour $k \in  \mathbb{N}$, l'espace de H\"older parabolique $C_{\mathcal{V}}^{k,\alpha}([0,T]\times \overset{\circ}{M}; E)$ est l'espace de Banach d\'efini par $$\{ \mu \in C_{\mathcal{V}}^{k}([0,T]\times \overset{\circ}{M}; E) \mid \nabla^{k} \mu \in  C_{\mathcal{V}}^{0,\alpha}([0,T]\times \overset{\circ}{M}; T^{(k,0)}\overset{\circ}{M}\otimes E) \}$$ avec norme donn\'ee par $$\left\| \mu \right\|_{k,\alpha} : = \left\| \mu \right\|_{k-1} +  \sum \limits_{2i+j = k} \left\| (\frac{\partial}{\partial t} )^{i} \nabla^{j} \mu \right\|_{0,\alpha}. $$
\end{defi}

Pour des sections dans $C_{\mathcal{V}}^{k,\alpha}( \overset{\circ}{M}; E)$, il existe parfois un d\'eveloppement lisse pr\`{e}s du bord $\partial M,$ \`a savoir un d\'eveloppement similaire aux s\'eries de Taylor pr\`es du bord  de la forme $ \displaystyle \sum_{n\in \mathbb{N}_{0}} \rho^{n} u_{n} $ o\`u $\rho$ est une fonction de d\'efinition du bord et $u_{n}$ est une section lisse sur $M.$ Sans avoir un d\'eveloppement lisse, on a toutefois souvent une expansion polyhomog\`{e}ne avec des termes de la forme $\rho^{z} (\log \rho)^{k}$ au lieu de $\rho^{-n}$, o\`u le couple $(z, k)$ appartient \`a un sous\texttt{-}ensemble de $\mathbb{C} \times \mathbb{N}_{0}$, appel\'e  ensemble indiciel d\'efini comme suit : 
\begin{defi}
Un ensemble indiciel $G$ est un sous\texttt{-}ensemble discret de $\mathbb{C} \times \mathbb{N}_{0}$ tel que 
\begin{enumerate}
\item $(z_{j},k_{j}) \in G, \; \left|(z_{j},k_{j})\right| \rightarrow \infty \Longrightarrow \Re z_{j} \rightarrow \infty;$
\item $(z,k) \in G  \Longrightarrow (z+p,k) \in G \; \;\forall p\in \mathbb{N};$
\item $(z,k) \in G  \Longrightarrow (z,p) \in G \;\; \forall p\in \{0,...,k\}.$
\end{enumerate}
L'ensemble indiciel $G$ est dit positif si  $\mathbb{N}_{0} \times \{0\} \subset G$ et 
$$(z,k) \in G  \Longrightarrow \Im z =0, \; \Re z\geq 0,$$
$$(0,k) \in G  \Longrightarrow k=0.$$
Une famille indicielle $\mathcal{G}$ de $M$ est la donn\'ee d'un ensemble indiciel $\mathcal{G}(H)$ pour chaque hypersurface bordante $H$ de $M$. Elle est dite positive si chacun de ses ensembles indiciels est positif. Si $\mathcal{G}$ est une famille indicielle et $F$ est une face de $M$, alors $\mathcal{G}_{\scriptscriptstyle{\vert F }}$ d\'enote la famille indicielle de $F$ qui associe \`a l'hypersurface bordante $F \cap H' $ de $F$ l'ensemble indiciel $\mathcal{G}(H') \in \mathcal{G}$. De m\^eme, si $U$ est un ouvert de $M,$  $\mathcal{G}_{\scriptscriptstyle{\vert U }}$ d\'enote la famille indicielle de $U$ qui associe \`a l'hypersurface bordante $U\cap H'$ de $U$ l'ensemble indiciel $\mathcal{G}(H') \in \mathcal{G}.$
 \end{defi}
 \begin{remq} \label{field}
Si $G_{1}$ et $G_{2}$ sont deux ensembles indiciels d'une hypersurface bordante $H$ de $M$, alors leur somme $$G_{1} + G_{2} := \{(z_{1}+z_{2}, k_{1}+k_{2}) \in \mathbb{C} \times \mathbb{N}_{0} \mid (z_{1}, k_{1}) \in G_{1}, (z_{2}, k_{2}) \in G_{2} \}$$ est aussi un ensemble indiciel de $H$. De m\^eme, si $\mathcal{G}_{1}$ et $\mathcal{G}_{2}$ sont deux familles indicielles de $M$, alors 
$\mathcal{G}_{1} + \mathcal{G}_{2}$
est une famille indicielle qui associe \`a l'hypersurface bordante $H$ l'ensemble indiciel $\mathcal{G}_{1}(H)+\mathcal{G}_{2}(H)$. On v\'erifie aussi que $\mathcal{G}_{1} \cup \mathcal{G}_{2}$ est une famille indicielle de $M$ qui associe \`a l'hypersurface bordante $H$ l'ensemble indiciel $\mathcal{G}_{1}(H) \cup \mathcal{G}_{2}(H)$. Si $\mathcal{G}_{1}$ et $\mathcal{G}_{2}$ sont positives alors $\mathcal{G}_{1} \cup \mathcal{G}_{2} \subset \mathcal{G}_{1} + \mathcal{G}_{2}$.
En particulier, si $\mathcal{G}$ est une famille indicielle positive, on peut lui associer la famille indicielle $$ \mathcal{G}_{\infty}= \sum \limits_{j=1}^{\infty} \mathcal{G}=\bigcup _{n=1}^{\infty} \sum \limits_{j=1}^{n} \mathcal{G},$$ et par construction, on a que $\mathcal{G}_{\infty}+\mathcal{G}=\mathcal{G}_{\infty}$.
 \end{remq} 
\begin{defi}
On suppose que $M$ est une vari\'et\'e compacte \`a bord, $\rho$ une fonction de d\'efinition du bord $\partial M$ et $G$ un ensemble indiciel de $M.$
L'espace $\mathcal{A}^{G}_{\phg}(M)$ est l'espace des fonctions polyhomog\`{e}nes $f \in  C^\infty(\overset{\circ}{M})$ ayant un d\'eveloppement asymptotique pr\`{e}s de $\partial M$ de la forme 
 $$f \sim \sum \limits_{(z,k) \in G} a(z,k) \rho^{z}(\log\rho )^{k}, \;\;a(z,k)\in C^\infty(M),$$ o\`u
 le symbole $\sim$ signifie que $\forall N \in \mathbb{N}$, on a
 $$f  - \sum \limits_{\substack{(z,k) \in G \\ \Re z \leq N}} a(z,k) \rho^{z}(\log\rho )^{k} \in \rho^{N} C_{b}^{N}(\overset{\circ}{M}) \;\;\textrm{avec}\;\; C_{b}^{N}(\overset{\circ}{M}):=C_{\mathcal{V}_{b}}^{N}(\overset{\circ}{M}),$$ 
 o\`u $(M, \mathcal{V}_{b})$ est la structure de Lie \`a l'infini de l'Exemple~\ref{bcchamp}.
\end{defi}
\begin{defi}\label{extension}
Soient $\pr_{1} : H \times [0,\varepsilon) \rightarrow H$ la projection canonique et
 $\eta_{H}\in C^\infty_{c}(M)$ une fonction de coupure lisse \`a support compact dans un voisinage tubulaire de $H,$ $c : H \times [0,\varepsilon) \leftrightarrow c\big( H \times [0,\varepsilon)\big) \subset M,$ tel que $\eta_{H} \equiv 1$ pr\`{e}s de $H$.
Une application d'extension des fonctions sur $H$ \`a des fonctions sur $M$ est d\'efinie par
$$\Xi_{H}(a)=  \eta_{H}{(c^{-1})}^{*}\pr_{1}^{*} a \;\;\textrm{v\'erifiant} \;\;\Xi_{H} (a)_{\scriptscriptstyle{\vert H}}= a.$$ 
\end{defi}
 \begin{defi}
Soit $\mathcal{G} = \{\mathcal{G}(H) \mid H \in \mathcal{M}_{1}(M) \}$ une famille indicielle pour une vari\'et\'e \`a coins $M.$ L'espace $\mathcal{A}^{\mathcal{G}}_{\phg}(M)$ est d\'efini r\'ecursivement sur la profondeur de la vari\'et\'e \`a coins par l'ensemble des fonctions $f \in  C^\infty(\overset{\circ}{M})$ satisfaisant, pour chaque hypersurface bordante $H$ de $M$,  
 $$f \sim \sum \limits_{(z,k) \in \mathcal{G}(H) } \Xi _{H} \big(a(z,k)\big) \rho_{H}^{z}(\log\rho_{H} )^{k}, a(z,k)\in \mathcal{A}^{\mathcal{G}_{\scriptscriptstyle{\vert H }}}_{\phg}(H),$$ 
 o\`u $\rho_{H}$ est une fonction de d\'efinition associ\'ee \`a $H$.
 Le symbole $\sim$ signifie que $\forall N \in \mathbb{N}$, on a
$$f  - \sum \limits_{\substack{(z,k) \in \mathcal{G}(H)  \\ \Re z \leq N}} \Xi _{H} \big(a(z,k)\big) \rho_{H}^{z}(\log\rho_{H})^{k} \in \rho_{H}^{N} \big(\displaystyle \prod _{\underset {H' \in \mathcal{M}_{1}(M)} {H' \neq H} } \rho_{H'}^{m_{H'}}(\log\rho_{H'})^{b_{H'}} \big)C_{\mathcal{V}_{b}}^{\infty}(\overset{\circ}{M}),$$ o\`u $(m_{H'},b_{H'}) \in  \mathcal{G}(H')$ est tel que si $(m,b) \in  \mathcal{G}(H')$ alors $\Re(m) \geq \Re(m_{H'})$ et $\Re(m) = \Re(m_{H'})$ implique $b \leq b_{H'}.$
La r\'ecursion se termine \'eventuellement et $\mathcal{A}^{\mathcal{G}}_{\phg}(M)$ est bien d\'efini lorsque les coefficients $a(z,k)$ sont dans $ \mathcal{A}^{*}_{\phg}(Y) \equiv C^{\infty}(Y)$, o\`u $Y$ est une vari\'et\'e ferm\'ee d\'etermin\'ee par une intersection maximale d'hypersurfaces bordantes de $M$.
\end{defi}
\begin{remq} \label{eva}
\'Etant donn\'ees $\mathcal{G}_{1}$ et $\mathcal{G}_{2}$ deux familles indicielles de $M$, si $f_{1} \in \mathcal{A}^{\mathcal{G}_{1}}_{\phg}(M)$ et $f_{2} \in \mathcal{A}^{\mathcal{G}_{2}}_{\phg}(M)$, alors $f_{1}+f_{2} \in  \mathcal{A}^{\mathcal{G}_{1} \cup \mathcal{G}_{2}}_{\phg}(M)$ et $f_{1}f_{2} \in  \mathcal{A}^{\mathcal{G}_{1} + \mathcal{G}_{2}}_{\phg}(M)$. 
 \end{remq}
\begin{remq}
L'espace $\mathcal{A}^{\mathcal{G}}_{\phg}(M)$ ne d\'epend pas du choix de fonction de d\'efinition du bord $\rho$ contrairement aux coefficients $a(z,k)$. Si chaque ensemble indiciel de $\mathcal{G}$ est $\mathbb{N}_{0} \times \{0\}$, alors $\mathcal{A}^{\mathcal{G}}_{\phg}(M)=C^\infty(M)$. Et si chaque ensemble indiciel de $\mathcal{G}$ est $\emptyset,$ alors $\mathcal{A}^{\mathcal{G}}_{\phg}(M)=\dot{C}^\infty(M)$ est l'espace des fonctions lisses sur $M$ s'annulant sur $\partial M$ ainsi que toutes leurs d\'eriv\'ees.
On remarque aussi que l'espace $\mathcal{A}^{\mathcal{G}}_{\phg}(M)$ est un $C^\infty$module.
 \end{remq}
 \begin{defi}
Soit $E \rightarrow M$ un fibr\'e vectoriel. L'espace des sections polyhomog\`{e}nes de $E$ correspondant \`a une famille indicielle $\mathcal{G}$ est d\'efini par 
 $$\mathcal{A}^{\mathcal{G}}_{\phg}(M;E)=\mathcal{A}^{\mathcal{G}}_{\phg}(M) \otimes_{C^\infty(M)} \Gamma(E).$$
 \end{defi}
Il est parfois n\'ecessaire de consid\'erer une classe de m\'etriques un peu plus grande, car la notion de m\'etrique riemannienne compatible avec une structure de Lie \`a l'infini $(M,\mathcal{V})$ est un peu stricte. 
 \begin{defi}
Une $\mathcal{V}$\texttt{-}m\'etrique est une m\'etrique riemannienne sur $\overset{\circ}{M}$ bi\texttt{-}Lipschitz \'equivalente \`a une m\'etrique compatible avec une structure de Lie \`a l'infini $(M,\mathcal{V})$ sur $\overset{\circ}{M}$. 
 \end{defi}
 \begin{defi}
Une $\mathcal{V}$\texttt{-}m\'etrique polyhomog\`{e}ne est une $\mathcal{V}$\texttt{-}m\'etrique $g$ sur $\overset{\circ}{M}$ d\'efinie par 
$$g=(\varrho_{\mathcal{V}}^{-1})^{*}(h_{\scriptscriptstyle{\vert \overset{\circ}{M}}}),$$ pour une certaine m\'etrique euclidienne $h \in \mathcal{A}^{\mathcal{G}}_{\phg}\big(M;\;\Sym^{2}({}^{\mathcal{V}}T^{*}M)\big)$ avec $\mathcal{G}$ une famille indicielle positive.
\end{defi}
Plusieurs des $\mathcal{V}$\texttt{-}m\'etriques ne sont pas typiquement lisses jusqu'au bord mais elles admettent souvent un d\'eveloppement asymptotique polyhomog\`{e}ne. Par exemple, les $\mathcal{V}_{\mathbb{C}\Theta}$\texttt{-}m\'etriques de Cheng Yau~\cite{cheng1980existence}, les m\'etriques Poincar\'e\texttt{-}Einstein~\cite{chrusciel2005boundary} et les m\'etriques Calabi\texttt{-}Yau asymtotiquement cylindriques ou coniques~\cite{conlon2015moduli}.

Nous examinerons certains aspects de la th\'eorie des op\'erateurs diff\'erentiels d\'etermin\'es par une structure particuli\`{e}re de Lie \`a l'infini. Soit $(\overset{\circ}{M}, g)$ une vari\'et\'e riemannienne avec une structure de Lie \`a l'infini $(M,\mathcal{V})$. 
Gr\^ace \`a la section pr\'ec\'edente, nous savons d\'ej\`a que la connexion de Levi\texttt{-}Civita sur $T\overset{\circ}{M}$ s'\'etend en une $ {}^{\mathcal{V}}TM$\texttt{-}connexion $\nabla$ de Levi\texttt{-}Civita sur $ {}^{\mathcal{V}}TM$. Cela implique aussi que la connexion de Levi\texttt{-}Civita sur $\Gamma(T^{(k,l)}\overset{\circ}{M}):=\Gamma \big((\otimes ^{k} \; T^{*}\overset{\circ}{M}) \otimes( \otimes^{l} \;T\overset{\circ}{M})\big)$ s'\'etend en une $ {}^{\mathcal{V}}TM$\texttt{-}connexion sur
$\Gamma( {}^{\mathcal{V}}TM^{(k,l)}),$ 
not\'ee aussi $\nabla$. Plus g\'en\'eralement, si $E \rightarrow M$ est un fibr\'e vectoriel, alors une $ {}^{\mathcal{V}}TM$\texttt{-}connexion sur $E$ induit une $ {}^{\mathcal{V}}TM$\texttt{-}connexion sur $ \Gamma( {}^{\mathcal{V}}TM^{(k,l)} \otimes E)$, not\'ee aussi $\nabla$. 
\begin{defi}
L'espace $\Diff^{*}_{\mathcal{V}}(M)$ est la $C^{\infty}(M)$\texttt{-}alg\`{e}bre universelle enveloppante de l'alg\`{e}bre de Lie $\mathcal{V}$. Autrement dit, l'espace $\Diff^{k}_{\mathcal{V}}(M)$ des op\'erateurs diff\'erentiels d'ordre k est l'espace des op\'erateurs diff\'erentiels engendr\'es par $C^{\infty}(M)$ et la composition d'au plus k \'el\'ements de $\mathcal{V}$.
\end{defi}
\begin{defi}
Puisque $\Diff^{k}_{\mathcal{V}}(M)$ est un $C^{\infty}$module, on d\'efinit l'ensemble des op\'erateurs diff\'erentiels de $\Diff^{k}_{\mathcal{V}}(M)$ agissant sur des sections d'un fibr\'e vectoriel $E \rightarrow M$ par 
  $$\Diff^{k}_{\mathcal{V}}(M;E)=  \Diff^{k}_{\mathcal{V}}(M) \otimes_{C^{\infty}(M)} \Gamma\big(\End(E)\big).$$
Donc chaque $L \in \Diff^{k}_{\mathcal{V}}(M;E)$ s'\'ecrit sous la forme 
$$L\mu =\sum \limits_{j=0}^{k} \zeta_{j}\, \cdot\, \nabla ^{j} \mu, \;\;   \nabla ^{j}  \mu \in \Gamma({}^{\mathcal{V}}TM^{(j,0)} \otimes E), \; \zeta_{j} \in \Gamma\big( {}^{\mathcal{V}}TM^{(0,j)} \otimes \End(E)\big),$$
o\`u \texttt{<<} $\cdot$ \texttt{>>} indique la contraction naturelle d'indices. \\ On d\'efinit aussi 
$\Diff^{k}_{\mathcal{V}}(\overset{\circ}{M};E)$ comme \'etant l'ensemble des op\'erateurs diff\'erentiels $L$ de la forme
$$L\mu=\sum \limits_{j=0}^{k} \zeta_{j}\, \cdot\, \nabla ^{j} \mu, \;\;   \nabla ^{j} \mu \in C_{\mathcal{V}}^{\infty}(\overset{\circ}{M};T^{(j,0)}\overset{\circ}{M}\otimes E), \;  \zeta_{j} \in C_{\mathcal{V}}^{\infty}\big(\overset{\circ}{M};T^{(0,j)}\overset{\circ}{M} \otimes \End(E)\big).$$
Si on change la r\'egularit\'e au\texttt{-}dessus $C_{\mathcal{V}}^{\infty}$ par $C_{\mathcal{V}}^{\ell,\alpha},$ l'ensemble des op\'erateurs diff\'erentiels $L$ sera not\'e par $\Diff^{k}_{\mathcal{V},\ell,\alpha}(\overset{\circ}{M};E).$
\end{defi}
\begin{defi}
Soit $\mathcal{G}$ une famille indicielle de $M$.
L'espace $$\Diff^{k}_{\mathcal{V}, \mathcal{G}}(M; E)= \mathcal{A}^{\mathcal{G}}_{\phg}(M) \otimes_{C^{\infty}(M)} \Diff^{k}_{\mathcal{V}}(M; E) $$ est l'ensemble des op\'erateurs diff\'erentiels de $\Diff^{k}_{\mathcal{V}}(M; E)$ polyhomog\`{e}nes par rapport \`a $\mathcal{G}$.
\end{defi}
\begin{remq}
De mani\`{e}re analogue \`a la Remarque~\ref{eva}, \'etant donn\'ees $\mathcal{G}$ et $\mathcal{K}$ deux familles indicielles de $M$, si $u \in \mathcal{A}^{\mathcal{G}}_{\phg}(M; E)$ et $P \in \Diff^{k}_{\mathcal{V}, \mathcal{K}}(M; E)$, alors $Pu \in  \mathcal{A}^{\mathcal{K}+ \mathcal{G}}_{\phg}(M; E)$.
\end{remq}
\begin{remq}
Les trois d\'efinitions pr\'ec\'edentes ont un sens m\^eme si $\mathcal{V}$ est une alg\`{e}bre de Lie structurale qui n'est pas une structure de Lie \`a l'infini.
  \end{remq}
\begin{remq}
On voit que $\Diff^{k}_{\mathcal{V}}(M; E)  \subset \Diff^{k}_{\mathcal{V}}(\overset{\circ}{M}; E) \subset \Diff^{k}_{\mathcal{V},\ell,\alpha}(\overset{\circ}{M};E).$ 
\end{remq}
\begin{defi}
Le symbole principal d'un op\'erateur $L \in  \Diff^{k}_{\mathcal{V}}(M; E)$ (respectivement $\Diff^{k}_{\mathcal{V}}(\overset{\circ}{M}; E)$) est l'application $\sigma_{k}(L) :  {}^{\mathcal{V}}T^{*}M \rightarrow \End(E)$ homog\`{e}ne de degr\'e $k$ sur les fibres d\'efinie par $\sigma_{k}(L)(\xi)=i^{k} \zeta_{k}\cdot \xi^{k}$, o\`u $\zeta_{k}$ est le coefficient du terme d'ordre $k$ et $\xi^{k} = \underbrace{\xi \otimes \xi ... \otimes \xi}_{k \text{ fois}}  \in {}^{\mathcal{V}}TM^{(k,0)}.$
\end{defi}
\begin{remq}
Le symbole principal induit une application $\sigma_{k} : \Diff^{k}_{\mathcal{V}}(M; E) \rightarrow C^{\infty}\big( {}^{\mathcal{V}}T^{*}M, \pi^{*}\End(E)\big),$ o\`u $\pi :  {}^{\mathcal{V}}T^{*}M \rightarrow M$ est la projection canonique, telle que $\sigma_{k_{1}+k_{2}}(L_{1} L_{2})=\sigma_{k_{1}}(L_{1})\sigma_{k_{2}}(L_{2}),$ pour $L_{i} \in \Diff^{k_{i}}_{\mathcal{V}}(M; E)$ avec $i\in\{1,2\}.$
\end{remq}
\begin{defi}
Un op\'erateur $L \in \Diff^{2}_{\mathcal{V}}(\overset{\circ}{M}; E) $ est dit uniform\'ement $\mathcal{V}$\texttt{-}elliptique si le symbole principal $\sigma_{2}(L)$ v\'erifie que $\forall m \in \overset{\circ}{M}$ et $\forall \xi \in T^{*}_{m}\overset{\circ}{M} \setminus \{0\}$ 
$$\sigma_{2}(L)(\xi) = - g^{ij}\xi_{i} \xi_{j},$$ pour une certaine $\mathcal{V}$\texttt{-}m\'etrique $g$. 
Autrement dit, $L$ est uniform\'ement $\mathcal{V}$\texttt{-}elliptique si son
symbole principal est le m\^eme que celui d'un laplacien associ\'e \`a une $\mathcal{V}$\texttt{-}m\'etrique.
De m\^eme, si $\{ L_{t}: t\in [0,T] \}$ est une famille lisse d'op\'erateurs uniform\'ement $\mathcal{V}$\texttt{-}elliptiques dans $\Diff^{2}_{\mathcal{V}}(M; E)$ (respectivement $\Diff^{2}_{\mathcal{V}}(\overset{\circ}{M}; E)$), on dira que $\frac{\partial}{\partial t}-L_{t}$ est un op\'erateur uniform\'ement $\mathcal{V}$\texttt{-}parabolique. 
\end{defi}
\begin{prop} \label{fibre}
Soient $(M,\mathcal{V}_{SF})$ une structure de Lie fibr\'ee \`a l'infini et $\mathcal{G}$ une famille indicielle positive de $M.$ On suppose que pour un fibr\'e associ\'e $\phi_{F}$ d'une face ferm\'ee $F$ de $M,$ il existe $s\in  S_{F}$ tel que la fibre $F_{s}:=\phi^{-1}_{F}(s)$ est de dimension non nulle. Si $L \in \Diff^{2}_{\mathcal{V}_{SF}, \mathcal{G}}(M; E)$ est un op\'erateur uniform\'ement\\ $\mathcal{V}_{SF}$\texttt{-}elliptique polyhomog\`{e}ne par rapport \`a $\mathcal{G}$, alors la restriction de $L$ au coefficient d'ordre $0$ sur la fibre $F_{s}$ est un op\'erateur $L^{0}_{s} \in \Diff^{2}_{\mathcal{V}_{F_{s}}, \mathcal{G}_{\scriptscriptstyle{\vert F_{s} }}}(F_{s}; E)$ uniform\'ement $\mathcal{V}_{F_{s}}$\texttt{-}elliptique. 
\end{prop}
\begin{proof}
Gr\^ace \`a la propri\'et\'e~\ref{almiir} de la Proposition~\ref{asd}, on a une suite exacte courte : 
\begin{equation} 
0\rightarrow \Ker(\pi_{F_{\scriptscriptstyle{\vert F_{s}}}}) \hookrightarrow {}^{\mathcal {V}_{SF}} T_{F_{s}}(M) \twoheadrightarrow {}^{\mathcal{V}_{F_{s}}} T(F_{s}) \rightarrow 0,
\end{equation}
o\`u $\mathcal{V}_{F_{s}}$ est la structure de Lie fibr\'ee \`a l'infini de $F_{s}.$ La suite exacte courte duale est donc
\begin{equation} 
0\rightarrow  {}^{\mathcal{V}_{F_{s}}} T^{*}(F_{s})  \hookrightarrow {}^{\mathcal {V}_{SF}} T^{*}_{F_{s}}(M) \twoheadrightarrow \big(\Ker(\pi_{F_{\scriptscriptstyle{\vert F_{s}}}})\big )^{*} \rightarrow 0,
\end{equation}
Par suite, la premi\`{e}re application $\iota : {}^{\mathcal {V}_{F_{s}}} T^{*}(F_{s})  \hookrightarrow {}^{\mathcal {V}_{SF}} T^{*}_{F_{s}}(M)$ induit un diagramme commutatif 
\begin{displaymath}
\xymatrix@+1pc{
\Diff^{2}_{\mathcal{V}_{SF}, \mathcal{G}}(M; E) \ar[r]^{\sigma_{2}} \ar[d]_{{}_{\scriptscriptstyle{\vert F_{s} }}} & C^{\infty}\big( {}^{\mathcal{V}_{SF}}T^{*}M, \pi^{*}\End(E)\big) \ar[d] ^{\iota^{*}}\\
   \Diff^{2}_{\mathcal{V}_{SF}, \mathcal{G}_{\scriptscriptstyle{\vert F_{s} }}}(F_{s}; E) \ar[r]^{ \sigma_{2}} &C^{\infty}\big( {}^{\mathcal{V}_{SF}}T^{*}F_{s}, \pi^{*}\End(E)\big),
  }
\end{displaymath} 
Donc $\sigma_{2}(L^{0}_{s}) = \iota^{*} \circ  \sigma_{2}(L)_{\scriptscriptstyle{\vert F_{s}}}.$
Si $L$ est $\mathcal{V}_{SF}$\texttt{-}uniform\'ement elliptique, on voit donc que $\sigma_{2}(L^{0}_{s}) $ sera \\$\mathcal{V}_{F_{s}}$\texttt{-}uniform\'ement elliptique.
\end{proof}

Nous \'etudions maintenant la polyhomog\'en\'eit\'e globale des solutions d'\'equations paraboliques lin\'eaires d\'etermin\'ees par une structure de Lie fibr\'ee \`a l'infini. Commen\c{c}ons par un petit lemme. 
\begin{lemme} \label{qwsa}
Soient $B(0,2r):=\{x\in \mathbb{R}^{n} : \left\| x - x_{0} \right\| \leq 2r \}$ une boule ferm\'ee centr\'ee en $0$ de rayon $2r$ et $\{L_{t} : t \in [0,T]\} $ une famille d'op\'erateurs uniform\'ement 
elliptiques de $\Diff^{2}_{ k, \alpha}(B(0,2r))$. Soient $f \in C^{k,\alpha}([0,T]\times B(0,2r))$ et $u_{0} \in C^{k+2,\alpha}(B(0,2r))$.  Si  $u \in C^{k+2,\alpha}([0,T]\times B(0,2r))$ est la solution unique de l'\'equation 
$$\frac{\partial u}{\partial t}-L_{t} u = f, \; u_{\scriptscriptstyle{\vert t=0 }}=u_{0},$$
alors il existe une constante $\kappa >0$ d\'ependante des normes des coefficients de l'op\'erateur $L_{t},$ la r\'egularit\'e $k,$ la dimension $n,$ le rayon $r$ et la constante d'ellipticit\'e de la famille $L_{t}$ telle que 
\begin{equation} \label{sdchauder}
\left\| u \right\|_{C^{k+2,\alpha}([0,T]\times B(0,r))}\leq \kappa ( \left\| u_{0} \right\|_{C^{k,\alpha}([0,T]\times \overset{\circ}{B}(0,2r))} + \left\| f \right\|_{C^{k,\alpha}([0,T]\times  \overset{\circ}{B}(0,2r))}).
\end{equation}
\end{lemme}
\begin{proof}
C'est un r\'esultat standard. Voici une preuve s'appuyant sur le livre~\cite{ladyzenskaja1967nn}).
Soit une fonction lisse \`a support compact $\varphi \in C_{c}^{\infty}(B(0,2r))$ telle que $\varphi \equiv 1$ sur $B(0,r)$ et $ \sup_{B(0,2r)} \left|\varphi \right|=1.$ On sait que $L_{t}$ est un op\'erateur uniform\'ement elliptique de $\Diff^{2}_{ k, \alpha}(B(0,2r))$ de la forme 
\begin{equation} 
L_{t}(\mu)=\sum \limits_{j=0}^{2} \zeta_{j} \cdot \nabla ^{j} \mu,
\end{equation}
o\`u $\nabla$ est la connexion de Levi\texttt{-}Civita pour la m\'etrique euclidienne sur $\mathbb{R}^{n}$. Ainsi, apr\`{e}s un calcul simple, on obtient que,  
$$
\frac{\partial (\varphi u)}{\partial t}-L_{t} (\varphi u)= -( \zeta_{2}\cdot\nabla ^{2} \varphi)u - 2\, \zeta_{2} ( \nabla \varphi, \nabla u)  
- \zeta_{1}\cdot (\nabla \varphi) u +\varphi  f, \;\; {\varphi u}_{\scriptscriptstyle{\vert \partial B(0,2r) \times[0,T] }}=0.$$

Par l'\'equation 5.3 du Th\'eor\`{e}me 5.2 dans~\cite{ladyzenskaja1967nn} \`a la page 320 (en prenant $\Phi = 0$), on obtient une estimation de $\varphi u$ sur $C^{k+2,\alpha}([0,T]\times  \overset{\circ}{B}(0,2r)).$  On d\'eduit ainsi l'estimation escompt\'ee~\eqref{sdchauder}. 
\end{proof}
Voici des g\'en\'eralisations naturelles de deux \'enonc\'es apparus dans l'article de Rochon~\cite{rochon2015polyhomogeneite} pour des\\ $\mathbb{C} \Theta$\texttt{-}m\'etriques (voir l'Exemple~\ref{epst}).
\begin{prop}\label{propiii}
Soit $(M, \mathcal{V})$ une structure de Lie \`a l'infini telle que le rayon d'injectivit\'e des m\'etriques compatibles est strictement positif. Soit $\{L_{t} : t \in [0,T]\} $ une famille d'op\'erateurs uniform\'ement $\mathcal{V}$\texttt{-}elliptiques de $\Diff^{2}_{\mathcal{V}, k, \alpha}(\overset{\circ}{M}; E)$. Si $u_{0} \in C_{\mathcal{V}}^{k+2,\alpha}(\overset{\circ}{M}; E)$ et $f \in C_{\mathcal{V}}^{k,\alpha}([0,T]\times \overset{\circ}{M}; E)$, alors l'\'equation 
$$\frac{\partial u}{\partial t}-L_{t} u = f, \; u_{\scriptscriptstyle{\vert t=0 }}=u_{0},$$
poss\`{e}de une solution $u \in C_{\mathcal{V}}^{k+2,\alpha}([0,T]\times \overset{\circ}{M}; E)$.
De plus, il existe une constante $\kappa >0$ d\'ependante de la famille $L_{t}$ telle que 
\begin{equation} \label{schauder}
\left\| u \right\|_{k+2,\alpha}\leq \kappa ( \left\| u_{0} \right\|_{k+2,\alpha} + \left\| f \right\|_{k,\alpha}).
\end{equation}
\end{prop}
\begin{proof}
D'abord, on remplace $u$ par $u-u_{0}$ pour se ramener au cas o\`u $u_{0}=0.$ Soit $(U_{p})_{p \in \mathbb{N}}$ un recouvrement de $\overset{\circ}{M}$ qui est donn\'e par une suite d'ouverts relativement compacts avec bord lisse tels que $U_{p} \subset U_{p+1}, \; \forall p \in \mathbb{N}.$ Soit une suite de fonctions lisses \`a supports compacts $\varphi_{p} \in C_{c}^{\infty}(U_{p+1})$ telle que $\varphi_{p} \equiv 1$ sur $U_{p}$ et $ \sup_{U_{p+1}} \left|\varphi_{p}\right|=1.$ Par un r\'esultat standard, voir par exemple le Th\'eor\`{e}me 7.1 apparu dans la section 7 du chapitre 7 de~\cite{ladyzenskaja1967nn}, pour chaque $p\in \mathbb{N},$ l'\'equation 
$$ \frac{\partial \omega_{p}}{\partial t}-L_{t} \omega_{p} = \varphi_{p} f, \; \omega_{p}(0,\cdot)=0, \; {\omega_{p}}_{\scriptscriptstyle{\vert \partial U_{p+1} \times[0,T] }}=0$$
poss\`{e}de une unique solution $\omega_{p} \in C^{k+2,\alpha}([0,T]\times U_{p+1}; E).$\\
L'estimation~\eqref{sdchauder} induit des estimations de Schauder locales ind\'ependantes de $p.$ En effet, comme $(\overset{\circ}{M}, g)$ est un exemple de g\'eom\'etrie born\'ee, il existe $r>0$ et $C_{k} >0$ avec $k \in \mathbb{N}_{0}$ tel que, pour tout $ p \in M,$ il existe une carte $\varphi : U_{p} \rightarrow B(0,2r) \subset \mathbb{R}^{n}$ avec $n = \dim M$ de sorte que 
$$\frac{g_{E}}{C_{0}}\leq \varphi_{*}(g)\leq C_{0}g_{E} \;\; \textrm{et}\; \;\sup_{B(0,2r)}\left| {(\nabla^{E})}^{k}\varphi_{*}(g) \right|_{g_{E}} \leq C_{k},$$
o\`u $g_{E}$ est la m\'etrique euclidienne sur $\mathbb{R}^{n}$ et $\nabla^{E}$ est la connexion de Levi\texttt{-}Civita associ\'ee. 
On applique par suite le Th\'eor\`{e}me d'Arzel\`a\texttt{-}Ascoli, pour extraire une sous\texttt{-}suite ${(\omega_{p_{q}})}_{q \in \mathbb{N}}$ qui converge uniform\'ement vers une solution $u \in C_{\mathcal{V}}^{k+2}([0,T]\times \overset{\circ}{M}; E)$. On v\'erifie alors que $\left\| u \right\|_{k+2,\alpha} <\infty.$
En effet, on a pour $2i+j=k+2,$
\begin{align*}
 \frac{\left|(\frac{\partial}{\partial t} )^{i} \nabla^{j}u(t',m)-(\frac{\partial}{\partial t} )^{i} \nabla^{j}u(t,m)\right|_{g}}{\left| t'-t\right|^{\frac{\alpha}{2}}} &\leq \frac{\left|(\frac{\partial}{\partial t} )^{i} \nabla^{j}u(t',m)-(\frac{\partial}{\partial t} )^{i} \nabla^{j}\omega_{p_{q}}(t',m)\right|_{g}}{\left| t'-t\right|^{\frac{\alpha}{2}}}\\
   &\;\;\;\;\;+\frac{\left|(\frac{\partial}{\partial t} )^{i} \nabla^{j}\omega_{p_{q}}(t',m)-(\frac{\partial}{\partial t} )^{i} \nabla^{j}\omega_{p_{q}}(t,m)\right|_{g}}{\left| t'-t\right|^{\frac{\alpha}{2}}}\\
   &\;\;\;\;\;+\frac{\left|(\frac{\partial}{\partial t} )^{i} \nabla^{j}\omega_{p_{q}}(t,m)-(\frac{\partial}{\partial t} )^{i} \nabla^{j}u(t,m)\right|_{g}}{\left| t'-t\right|^{\frac{\alpha}{2}}}.
\end{align*}
Soit maintenant $\varepsilon> 0$ fix\'e. Comme $\omega_{p_{q}}$ converge uniform\'ement vers $u$ dans $C_{\mathcal{V}}^{k+2}([0,T]\times \overset{\circ}{M}; E),$  on prend $q$ suffisamment grand d\'ependant de $t$ et $t'$ tel que
 $$\frac{\left|(\frac{\partial}{\partial t} )^{i} \nabla^{j}u(t',m)-(\frac{\partial}{\partial t} )^{i} \nabla^{j}u(t, m)\right| _{g}}{\left| t'-t\right|^{\frac{\alpha}{2}}}  \leq  
\varepsilon  + \frac{\left|(\frac{\partial}{\partial t} )^{i} \nabla^{j}\omega_{p_{q}}(t',m)-(\frac{\partial}{\partial t} )^{i} \nabla^{j}\omega_{p_{q}}(t,m)\right|_{g}}{\left| t'-t\right|^{\frac{\alpha}{2}}}+\varepsilon \leq C+2 \varepsilon,   $$
o\`u $C$ est une constante telle que $\left\| \omega_{p} \right\|_{k+2,\alpha}\leq C,$ $\forall p.$ En prenant le supremum sur $t\neq t'$, on obtient alors 
 $$
\sup_{t \neq t'} \frac{\left|(\frac{\partial}{\partial t} )^{i} \nabla^{j}u(t',m)-(\frac{\partial}{\partial t} )^{i} \nabla^{j}u(t, m)\right| _{g}}{\left| t'-t\right|^{\frac{\alpha}{2}}}  \leq  C+2\varepsilon<\infty. 
$$
De m\^eme, 
\begin{align*}
\frac{\left| \psi_{\gamma}\big((\frac{\partial}{\partial t} )^{i} \nabla^{j}u(\gamma(0))\big)-(\frac{\partial}{\partial t} )^{i} \nabla^{j}u(\gamma(1))\right|_{g} }{l(\gamma)^{\alpha}} &\leq \frac{\left| \psi_{\gamma}\big((\frac{\partial}{\partial t} )^{i} \nabla^{j}u(\gamma(0))\big)-\psi_{\gamma}\big((\frac{\partial}{\partial t} )^{i} \nabla^{j}\omega_{p_{q}}(\gamma(0))\big)\right|_{g} }{l(\gamma)^{\alpha}} \\
   &\;\;\;\;\; + \frac {\left| \psi_{\gamma}\big((\frac{\partial}{\partial t} )^{i} \nabla^{j}\omega_{p_{q}}(\gamma(0))\big)-(\frac{\partial}{\partial t} )^{i} \nabla^{j}\omega_{p_{q}}(\gamma(1))\right|_{g} }{l(\gamma)^{\alpha}} \\
   &\;\;\;\;\; + \frac {\left| (\frac{\partial}{\partial t} )^{i} \nabla^{j}\omega_{p_{q}}(\gamma(1))-(\frac{\partial}{\partial t} )^{i} \nabla^{j} u(\gamma(1))\right|_{g} }{l(\gamma)^{\alpha}},
\end{align*}
ce qui pour $q$ suffisamment grand donne
$$
\sup_{\gamma}\frac{\left| \psi_{\gamma}\big((\frac{\partial}{\partial t} )^{i} \nabla^{j}u(\gamma(0))\big)-(\frac{\partial}{\partial t} )^{i} \nabla^{j}u(\gamma(1))\right|_{g} }{l(\gamma)^{\alpha}} \leq C+2 \varepsilon.
$$
Comme par hypoth\`{e}se $(\overset{\circ}{M}, g)$ est un exemple de g\'eom\'etrie born\'ee, les estimations de Schauder paraboliques locales peuvent \^etre consid\'er\'ees pour donner l'estimation voulue pour $\left\| u \right\|_{k+2,\alpha}.$
\end{proof}
\begin{cor}\label{corr}
Soient $(H_{i})_{i \in \{1,...,l\}}$ la famille des hypersurfaces bordantes d'une vari\'et\'e \`a coins compacte $M$ et $\{L_{t} : t\in [0,T]\}$ une famille d'op\'erateurs uniform\'ement $\mathcal{V}$\texttt{-}elliptiques de $\Diff^{2}_{\mathcal{V}}(\overset{\circ}{M}; E)$. Soient $\rho_{i}$ une fonction de d\'efinition de $H_{i},$ pour $i \in \{1,...,l\},$ $f \in \rho^{\beta} (\log \rho)^{\eta} C_{\mathcal{V}}^{k,\alpha}([0,T]\times \overset{\circ}{M}; E)$ et $u_{0} \in \rho^{\beta} (\log \rho)^{\eta}C_{\mathcal{V}}^{k+2,\alpha}(\overset{\circ}{M}; E),$ o\`u $\rho^{\beta}=\prod_{i=1}^{l}\rho^{\beta_{i}}_{i},$ $(\log \rho)^{\eta}= \prod_{i=1}^{l}(\log \rho_{i})^{\eta_{i}},$ $\beta=(\beta_{1},...,\beta_{l})$ et $\eta=(\eta_{1},...,\eta_{l})$ avec $\beta_{i} >0$ et $\eta_{i} \in \mathbb{N}_{0}$ (ou $\beta_{i}=\eta_{i}=0$). 
Alors l'\'equation 
$$\frac{\partial u}{\partial t}-L_{t} u = f, \; u_{\scriptscriptstyle{\vert t=0 }}=u_{0},$$
poss\`{e}de une unique solution $u \in \rho^{\beta} (\log \rho)^{\eta}C_{\mathcal{V}}^{k+2,\alpha}([0,T]\times \overset{\circ}{M}; E).$
\end{cor}
\begin{proof}
D'abord, on consid\`{e}re 
l'op\'erateur $$\tilde{L}_{t} :=\rho^{-\beta} (\log \rho)^{-\eta} \circ L_{t} \circ \rho^{\beta} (\log \rho)^{\eta}.$$ On a que $\tilde{L}_{t} \in \Diff^{2}_{\mathcal{V}}(\overset{\circ}{M}; E)$ est une famille lisse d'op\'erateurs par la Remarque~\ref{lisse}, et aussi clairement uniform\'ement $\mathcal{V}$\texttt{-}elliptiques car
$$\sigma_{2}(\tilde{L}_{t})(\xi) =\sigma_{2}(L_{t})(\xi) = - \rho^{-\beta} (\log \rho)^{-\eta}  \rho^{\beta} (\log \rho)^{\eta}g^{ij}\xi_{i} \xi_{j},\;\; \forall \xi \in T^{*}_{m}\overset{\circ}{M} \setminus \{0\},$$ pour une certaine $\mathcal{V}$\texttt{-}m\'etrique $g$. 
On peut se ramener alors \`a une \'equation \'equivalente \`a l'\'equation d\'ecrite ci\texttt{-}dessus, \`a savoir
$$\frac{\partial \tilde{u}}{\partial t}-\tilde{L}_{t} \tilde{u} = \tilde{f}, \; \tilde{u}_{\scriptscriptstyle{\vert t=0 }}=\tilde{u}_{0}, \; \textrm{avec} \;
\tilde{u}=\rho^{-\beta} (\log \rho)^{-\eta} u, \; \tilde{u}_{0}=\rho^{-\beta} (\log \rho)^{-\eta} u_{0}, \; \tilde{f}=\rho^{-\beta} (\log \rho)^{-\eta} f,$$
et l'existence de la solution d\'ecoule de la proposition pr\'ec\'edente.
L'unicit\'e de la solution se d\'emontre \`a l'aide du principe du maximum en v\'erifiant que si $u_{1}$ et $u_{2}$ sont deux solutions, alors $u_{1}=u_{2}.$ 
D'abord, en rempla\c{c}ant $u_{1}, \;u_{2},\; L_{t}$ et $f$ par $\rho^{\alpha} u_{1},\; \rho^{\alpha}u_{2}, \; \tilde{L}_{t} = \rho^{\alpha} L_{t} \rho^{-\alpha}$ et $ \rho^{\alpha}f,$ avec $\alpha=(\alpha_{1},...,\alpha_{l}), \alpha_{i}>0,$ on peut se ramener au cas o\`u $\left|u_{1} \right|_{g(t)}$  et $\left|u_{2} \right|_{g(t)}$ d\'ecroissent vers z\'ero \`a l'infini. 
Maintenant, on consid\`{e}re $u= u_{1}-u_{2}$ qui satisfait \`a l'\'equation  
$$\frac{\partial u}{\partial t}=L_{t} u, \; u_{\scriptscriptstyle{\vert t=0 }}=0.$$
Comme $L_{t}$ s'\'ecrit sous la forme
\begin{equation}
L_{t}(\mu)=\sum \limits_{j=0}^{2} \zeta_{j}\cdot \nabla ^{j} \mu = \Delta_{g(t)} \mu + \zeta_{1}\cdot \nabla \mu + \zeta_{0}\cdot\mu,
\end{equation}
 on a successivement 
\begin{align*}
 \frac{\partial}{\partial t}\left|u \right|^{2} & = \Delta \left|u \right|^{2} -2\left|\nabla u \right|^{2}+ 2 {\langle u,  \zeta_{1}\cdot \nabla u \rangle} + 2\langle u, \zeta_{0} \cdot u  \rangle\\
                                                              & = \Delta \left|u \right|^{2} -2\left|\nabla u \right|^{2}+ 2 {\langle \zeta^{*}_{1} \cdot u,  \nabla u \rangle}+ 2\langle u, \zeta_{0} \cdot u  \rangle\\
                                                              & \leq \Delta \left|u \right|^{2} -2\left|\nabla u \right|^{2}+ \left| \zeta^{*}_{1}\cdot u \right|^{2} + \left|  \nabla u \right|^{2} + 2\langle u, \zeta_{0} \cdot u  \rangle\\
                                                              &  \leq \Delta \left|u \right|^{2} + C\left|u \right|^{2},
\end{align*}
o\`u $C>0$ est une constante d\'ependante de $\zeta_{0}$ et $\zeta^{*}_{1}.$ On conclut finalement que $u\equiv0$ par le principe du maximum. 
\end{proof}
Dans l'article~\cite{rochon2015polyhomogeneite}, le r\'esultat de la polyhomog\'en\'eit\'e des solutions utilise le principe du maximum. L'argument pourrait en principe s'\'etendre plus g\'en\'eralement aux structures de Lie \'evanescentes \`a l'infini. Notre prochain objectif sera d'\'etablir plus g\'en\'eralement la polyhomog\'en\'eit\'e globale de solutions d'\'equations paraboliques lin\'eaires d\'etermin\'ees par une structure de Lie fibr\'ee \`a l'infini. En fait, afin de d\'eterminer un candidat pour le coefficient du terme d'un certain ordre dans le d\'eveloppement polyhomog\`{e}ne, nous serons oblig\'e d'\'etudier plus g\'en\'eralement une famille param\'etr\'ee de telles \'equations sur une vari\'et\'e \`a coins $M$, via la restriction au bord. Cette famille sera d\'etermin\'ee par un fibr\'e de structures de Lie fibr\'ees \`a l'infini associ\'e \`a un fibr\'e $\phi_{M}$. Pour cette raison, nous consid\'erons d'embl\'ee une famille d'\'equations paraboliques, plut\^ot qu'une seule \'equation parabolique. Cela nous permettra d'utiliser le principe d'induction sur la profondeur de la vari\'et\'e dans la preuve de la polyhomog\'en\'eit\'e.
Avant d'\'etablir ce cas g\'en\'eral, nous allons consid\'erer un cas particulier facile \`a \'etudier, \`a savoir lorsque $\phi_{M}= \Id.$

Pour $T>0$ et $\mathcal{G}$ est une famille indicielle associ\'ee \`a une vari\'et\'e \`a coins compacte $M$, l'espace $\mathcal{A}_{\phg}^{\mathcal{G}}([0,T]\times M)$ est constitu\'e des fonctions polyhomog\`enes sur $[0,T]\times M$ avec d\'eveloppement lisse en $\{0\}\times M$ et $\{T\}\times M$ et d\'eveloppement polyhomog\`ene sp\'ecifi\'e par $\mathcal{G}(M)$ en $[0,T]\times M$ pour $H$ une hypersurface bordante de $M$. De m\^eme, pour $E \rightarrow M$ un fibr\'e vectoriel, on posera $$\mathcal{A}_{\phg}^{\mathcal{G}}([0,T]\times M; E)= \mathcal{A}_{\phg}^{\mathcal{G}}([0,T]\times M) \otimes_ {C^{\infty}([0,T]\times M)} C^{\infty}([0,T]\times M; E).$$
\begin{prop}\label{priooop}
Soient $M$ une vari\'et\'e \`a coins compacte et $\mathcal{G},$ $\mathcal{G}_{1}$ et $\mathcal{G}_{2}$ des familles indicielles positives de $M.$ 
Alors pour $ l_{t} \in \mathcal{A}_{\phg}^{\mathcal{G}}([0,T]\times M; \End(E)),$ $u_{0} \in \mathcal{A}^{\mathcal{G}_{1}}_{\phg}(M; E)$ et $f \in \mathcal{A}^{\mathcal{G}_{2}}_{\phg}([0,T]\times M; E)$, l'\'equation diff\'erentielle ordinaire lin\'eaire d'ordre $1$ de la forme 
\begin{equation}\label{mneeq}
\frac{\partial }{\partial t}u-l_{t}u = f,\; u_{\scriptscriptstyle{\vert t=0 }}=u_{0}, 
\end{equation}
poss\`{e}de une unique solution $u \in \mathcal{A}^{\mathcal{K}}_{\phg}([0,T] \times M; E)$, o\`u $\mathcal{K}$ est la famille indicielle positive donn\'ee par $ \displaystyle \mathcal{K}=\mathcal{G}_{1} \cup (\mathcal{G}_{\infty}+  \mathcal{G}_{2} ) + \mathcal{G}_{\infty}.$
\end{prop}
\begin{proof}
L'\'equation~\eqref{mneeq} poss\`{e}de une solution unique d\'efinie par 
\begin{equation}\label{meq}
u(t)=  \exp (-P_{t})(u_{0} + \int_{0}^{t}  \exp(P_{\tau}) f(\tau) \, \mathrm{d}\tau), \textrm{\:o\`u\:} P_{t}= -\int_{0}^{t} l_{\tau} \, \mathrm{d}\tau.
\end{equation}
Puisque l'int\'egrale en $t$ d'une fonction polyhomog\`{e}ne reste aussi une fonction polyhomog\`{e}ne, on a que $P_{t} \in \mathcal{A}_{\phg}^{\mathcal{G}}([0,T]\times M;  E)$. Ensuite, comme $ \mathcal{G}_{\infty}= \mathcal{G}_{\infty} + \mathcal{G}$ par la Remarque~\ref{field}, on a automatiquement $$\exp (-P_{t})=  \sum \limits_{j \in \mathbb{N}_{0}} \displaystyle \frac{(-P_{t})^{j}}{\fact j} \in\mathcal{A}_{\phg}^{{\mathcal{G}_{\infty}}}([0,T]\times M;  E).$$ \`A partir de~\eqref{meq}, on d\'eduit que $u \in \mathcal{A}_{\phg}^{{{\mathcal{K}}}}([0,T]\times M ;E).$
\end{proof}
Dans la suite, nous allons introduire la notion d'\'eclatement logarithmique (voir la section 5.14 dans le livre de Melrose~\cite{melrose1996differential} et l'article~\cite{rochon2012asymptotics} autour de l'\'equation 1.14) afin d'obtenir la polyhomog\'en\'eit\'e d'une famille de solutions d'une famille d'\'equations paraboliques. Un \'eclatement logarithmique $[M, H]_{\log}$ d'une hypersurface bordante $H$ est la vari\'et\'e \`a coins compacte $[M, H]_{\log}$ identifi\'ee topologiquement avec $M$, mais avec\\ $C^{\infty}$ structure \`a coins engendr\'ee par $C^{\infty}(M)$ et la nouvelle fonction de d\'efinition  
$$y_{H}= \frac{-1}{\log \rho_{H}},$$
o\`u $\rho_{H}$ est une fonction de d\'efinition de $H$ dans $M$.
Sans perte de g\'en\'eralit\'e, on suppose que $\rho_{H}<1.$ Par le Lemme 5.14.1~\cite{melrose1996differential}, l'\'eclatement logarithmique $[M, H]_{\log}$ est ind\'ependant du choix de fonction de d\'efinition $\rho_{H}$ et est tel que l'identit\'e $\Id : [M, H]_{\log} \rightarrow M$ est $C^{\infty}.$ Clairement, les diff\'erents \'eclatements logarithmiques entre les hypersurfaces bordantes commutent. Ceci nous permet de d\'efinir l'\'eclatement logarithmique total. Si $H_{1}, ... ,H_{l}$ est une liste des hypersurfaces bordantes de $M,$ alors on pose
$$M_{\log}=M_{l},\; \textrm{avec}\;M_{0}=M \; \textrm{et}\; M_{j}=[M_{j-1}, H]_{\log}, \; j\in\{1,..,l\}.$$
On voit que, pour $\mathcal{G}$ une famille indicielle positive de $M,$ $\displaystyle \mathcal{A}^{\mathcal{G}}_{\phg}(M) \subset C^{\infty}(M_{\log}),$ via\\ $\Id^{*} : \mathcal{A}_{\phg}^{\mathcal{G}}(M) \rightarrow C^{\infty}(M_{\log}).$ De plus, on obtient que 
\begin{equation} \label{fdd}
C_{\mathcal{V}_{pc}}^{\infty}(M_{\log}) = C_{b}^{\infty}(\overset{\circ}{M}),
\end{equation}
o\`u 
$\mathcal{V}_{pc}:=\{ V \in \mathfrak{X}(M_{\log}) \mid Vy_{H}\in y^{2}_{H}C^{\infty}(M_{\log}) \;\forall H\in \mathcal{M}_{1}(M_{\log})\}.$
En effet, on a que 
\begin{align*}
\rho_{H} \frac{\partial}{\partial \rho_{H}}&=\rho_{H} \frac{\partial y_{H}}{\partial \rho_{H}} \frac{\partial}{\partial y_{H}}\\
&=\rho_{H} \frac{1}{\log^{2}(\rho_{H})}\frac{1}{ \rho_{H}} \frac{\partial}{\partial y_{H}}\\
&=y^{2}_{H} \frac{\partial}{\partial y_{H}}.
\end{align*}  
De plus, \'etant donn\'ee $f\in C_{\mathcal{V}_{pc}}^{\infty}(M_{\log})$, sur $(y_{H},x_{1},...,x_{n-1})$ un syst\`{e}me de coordonn\'ees locales centr\'e en $m\in H,$ on a que 
\begin{align*}
\rho_{H} \frac{\partial}{\partial \rho_{H}}\big(f(y_{H},x_{1},...,x_{n-1})\big)&=\rho_{H} \frac{\partial}{\partial \rho_{H}} (y_{H}) ( \frac{\partial}{\partial y_{H}}f)(y_{H},x_{1},...,x_{n-1})\\
&=(y^{2}_{H} \frac{\partial}{\partial y_{H}}f)(y_{H},x_{1},...,x_{n-1}).
\end{align*}  
\begin{prop}\label{prooop}
Soit $(M, S, \phi_{M}, \mathcal{V})$ un fibr\'e de structures de Lie fibr\'ees \`a l'infini associ\'e au fibr\'e $\phi_{M} : M \rightarrow S$ tel que le rayon d'injectivit\'e des m\'etriques compatibles sur les fibres soit strictement positif. Soient $\Psi_{U} : \phi_{M}^{-1}(U) \rightarrow U \times Z$ une trivialisation locale du fibr\'e de structures de Lie fibr\'ees \`a l'infini, pour $U$ un ouvert de $S,$ et $(Z,\mathcal{V}_{Z})$ la structure de Lie fibr\'ee \`a l'infini de la fibre $Z$. Soient $\mathcal{G}_{1}$ et $\mathcal{G}_{2}$ deux familles indicielles positives de $S,$ $L_{t}(s)$ une famille polyhomog\`{e}ne sur $[0,T] \times U$ (lisse en t) d'op\'erateurs uniform\'ement $\mathcal{V}_{Z}$\texttt{-}elliptiques dans $\Diff^{2}_{\mathcal{V}_{Z}}( Z; E)$ par rapport ${\mathcal{G}_{1}}_{\scriptscriptstyle{\vert U}},$ $$f\in 
\mathcal{A}^{{\mathcal{G}_{2}}_{\scriptscriptstyle{\vert  U  }}}_{\phg}\big(U,C^{k,\alpha}_{\mathcal{V}_{Z}}([0,T]\times \overset{\circ}{Z}; E)\big)  + x_{\scriptscriptstyle{\vert U }}^{\vartheta} C_{b}^{\infty}\big( \overset{\circ}{U}; C^{k,\alpha}_{\mathcal{V}_{Z}}([0,T]\times \overset{\circ}{Z}; E)\big),$$ pour un certain $\vartheta>0,$ o\`u $x$ est le produit des fonctions de d\'efinitions des hypersurfaces bordantes de $S,$ et $u_{0} \in C_{b}^{\infty}\big(\overset{\circ}{U},C^{k+2,\alpha}_{\mathcal{V}_{Z}}(\overset{\circ}{Z}; E)\big).$
La famille d'\'equations $\mathcal{V}_{Z}$\texttt{-}paraboliques param\'etr\'ee par $s \in U,$
\begin{equation}\label{eee}
\frac{\partial u}{\partial t}-L_{t} u = f, \; u_{\scriptscriptstyle{\vert t=0 }}=u_{0},
\end{equation}
poss\`{e}de alors une unique solution $u \in C_{b}^{\infty}\big(\overset{\circ}{U},C^{k+2,\alpha}_{\mathcal{V}_{Z}}([0,T]\times \overset{\circ}{Z}; E)\big)$.
\end{prop}
\begin{proof}
Si $\dim Z=0,$ alors~\eqref{eee} est une EDO et le r\'esultat d\'ecoule de la Proposition~\ref{priooop}. On peut donc supposer que $\dim Z\geq 1.$ De plus, en rempla\c{c}ant $u$ par $u-u_{0},$ on peut se ramener au cas o\`u $u_{0}=0.$ Pour $s\in U$ fix\'e, par le Corollaire~\ref{corr}, chaque \'equation de~\eqref{eee} poss\`{e}de une unique solution $u(s) \in C_{\mathcal{V}_{Z}}^{k+2,\alpha}([0,T]\times \overset{\circ}{Z}; E).$  

On commen\c{c}e par le cas $U \subset \overset{\circ}{S}.$ Pour $k \geq 0$ arbitraire, on consid\`{e}re la fonction\\ $F: U \times C_{\mathcal{V}_{Z}}^{k+2, \alpha}([0,T]\times \overset{\circ}{Z}; E) \rightarrow  C_{\mathcal{V}_{Z}}^{k, \alpha}([0,T]\times \overset{\circ}{Z}; E)$ d\'efinie par $$F(s,v):= \frac{\partial v}{\partial t}-L_{t}(s) v - f(s,\cdot).$$ Il est clair que $F$ est de classe $C^{\infty}$ sur $U$ et que $F(s,\cdot)$ est de classe $C^{\infty}$. En fait, $d_{v}F_{(s,v)} =\frac{\partial}{\partial t}-L_{t}(s)$ est lin\'eaire et bijective gr\^ace au Corollaire~\ref{corr}, puisque on sait que $\forall f(s,\cdot) \in C_{\mathcal{V}_{Z} }^{k, \alpha}([0,T]\times \overset{\circ}{Z}; E)$ il existe une unique solution $u(s,\cdot) \in C_{\mathcal{V}_{Z}}^{k+2, \alpha}([0,T] \times \overset{\circ}{Z}; E)$ de~\eqref{eee}. En appliquant le th\'eor\`{e}me des fonctions implicites~\cite{krantz2012implicit}, on aura que $u \in C^{\infty}\big(U,C^{k+2,\alpha}_{\mathcal{V}_{Z}}([0,T]\times \overset{\circ}{Z}; E)\big).$
 
Pour traiter le cas $U \subset S$ tel que $U\cap \partial S \neq 0,$ l'id\'ee est de se ramener au cas pr\'ec\'edent en prolongeant l'\'equation au\texttt{-}del\`a de la vari\'et\'e \`a coins $U$. Pour y arriver, il faut cependant utiliser l'\'eclatement logarithmique.  
Sans perte de g\'en\'eralit\'e, on peut supposer que $U$ est un voisinage ouvert de l'origine dans $\mathbb{R}^{n}_{l},$ o\`u\\ $n=\dim S.$ Alors les coordonn\'ees $(x_{1},...,x_{n})$ de $\mathbb{R}^{n}_{l}$ sont telles que $x_{1},...,x_{l}$ sont les fonctions de d\'efinitions des hypersurfaces bordantes de $U$ : 
$$H_{i}=U \cap \{x_{i}=0\}, \;\;\; i \in \{1,...,l\}.$$
Soit $U_{\log}$ l'\'eclatement logarithmique total de $U$ obtenu en utilisant les nouvelles fonctions de d\'efinitions des hypersurfaces bordantes $$\displaystyle y_{i} := \frac{-1}{\log x_{i}}, \;\;\;i\in \{1,...,l\}.$$
Alors $U_{\log}$ est aussi un ouvert de $\mathbb{R}^{n}_{l}$ contenant  l'origine, mais en utilisant les coordonn\'ees $(y_{1},...,y_{l},x_{l+1},...,x_{n})$ dans $\mathbb{R}^{n}_{l}.$ Soit $\Omega_{\log}\subset \mathbb{R}^{n}$ un ouvert tel que $U_{\log}=\Omega_{\log} \cap  \mathbb{R}^{n}_{l}.$ On veut alors se ramener au cas pr\'ec\'edent en prolongeant la famille d'\'equations~\eqref{eee} \`a $\Omega_{\log}$. Maintenant pr\`{e}s de $H_{i},$ la famille d'\'equations~\eqref{eee} a un d\'eveloppement asymptotique de la forme
\begin{equation} \label{ilpl}
\frac{\partial u}{\partial t} - \sum \limits_{\substack{(\eta_{1},k_{1}) \in \mathcal{G}_{1}(H_{i}) }}  x_{i}^{\eta_{1}}(\log x_{i})^{k_{1}} \Xi _{H_{i}} (L_{t,\eta_{1},k_{1}}) u \;\sim \sum \limits_{\substack{(\eta_{2},k_{2}) \in \mathcal{G}_{2}(H_{i}) }}  x_{i}^{\eta_{2}}(\log x_{i})^{k_{2}} f_{t,\eta_{2},k_{2}} + x_{i}^{\vartheta} \tilde{f},
\end{equation}
o\`u $\displaystyle \tilde{f} \in (\frac{x_{\scriptscriptstyle{\vert U }}}{x_{i}})^{\vartheta} C_{b}^{\infty}\big( \overset{\circ}{U}; C^{\infty}_{\mathcal{V}_{Z}}([0,T]\times \overset{\circ}{Z}; E)\big),$ $f_{t,\eta_{2},k_{2}}$ est le coefficient du terme d'ordre $x_{i}^{\eta_{2}}(\log x_{i})^{k_{2}}$ de $f$ et $L_{t,\eta_{1},k_{1}}$ est une famille polyhomog\`{e}ne param\'etr\'ee par $s \in H_{i}$ d'op\'erateurs dans $\Diff^{2}_{\mathcal{V}_{Z}}( Z; E).$ Sur $U_{\log},$ on doit plut\^ot utiliser les coordonn\'ees $y_{i},$ ce qui donne le d\'eveloppement asymptotique
\begin{multline*}
\frac{\partial u}{\partial t} - \sum \limits_{\substack{(\eta_{1},k_{1}) \in \mathcal{G}_{1}(H_{i})}} \big(\exp (-\frac{1}{y_{i}} )\big)^{\eta_{1}}y_{i}^{-k_{1}} \Xi _{H_{i}} (L_{t,\eta_{1},k_{1}}) u \;\sim\\ \sum \limits_{\substack{(\eta_{2},k_{2}) \in \mathcal{G}_{2}(H_{i})}} \big(\exp (-\frac{1}{y_{i}} )\big)^{\eta_{2}}y_{i}^{-k_{2}} f_{t,\eta_{2},k_{2}} + \big(\exp (-\frac{1}{y_{i}} )\big)^{\vartheta} \tilde{f}.
\end{multline*}
Or, il est bien connu que, pour un certain couple $(\eta,k)$ dans un tel ensemble indiciel positif, $\displaystyle \big(\exp (-\frac{1}{y_{i}} )\big)^{\eta}y_{i}^{-k},$ initialement d\'efinie par $y_{i}>0,$ se prolonge par z\'ero en une fonction lisse pour tout $y_{i} \in \mathbb{R}$. En prenant $i=l,$ on peut donc prolonger naturellement la famille d'\'equations sur $U_{\log}$ \`a $\Omega_{\log} \cap \mathbb{R}^{n}_{l-1}.$ Cette construction peut \^etre it\'er\'ee successivement pour prolonger la famille d'\'equations \`a $\Omega_{\log} \cap \mathbb{R}^{n}_{l-2},...,\Omega_{\log} \cap \mathbb{R}^{n}_{1}$ et finalement\\ $\Omega_{\log} \cap \mathbb{R}^{n}=\Omega_{\log}.$ \`A nouveau, gr\^ace au th\'eor\`{e}me des fonctions implicites~\cite{krantz2012implicit}, on aura que $$u \in C^{\infty}\big(\Omega_{\log}; C^{k+2,\alpha}_{\mathcal{V}_{Z}}([0,T]\times \overset{\circ}{Z}; E)\big) \subset  C_{\mathcal{V}_{pc}}^{\infty}\big({U}_{\log}; C^{k+2,\alpha}_{\mathcal{V}_{Z}}([0,T]\times \overset{\circ}{Z}; E)\big)$$ et donc par~\eqref{fdd}, $$u \in C_{b}^{\infty}\big(\overset{\circ}{U}; C^{k+2, \alpha}_{\mathcal{V}_{Z}}([0,T]\times \overset{\circ}{Z}; E)\big).$$
\end{proof}
\begin{remq}\label{llilila}
Les constantes de Schauder $\kappa_{s}$ de~\eqref{schauder} correspondantes \`a $L_{t}(s),$ pour tout $s\in S,$ peuvent \^etre choisies d'une fa\c{c}on uniforme donn\'ee par $$\displaystyle \kappa= \sup_{ s \in S}\kappa_{s}= \sup_{ s \in S}\big( \left\| (d_{v}F_{(s,v)})^{-1} \right\|(\left\| L_{t}(s) \right\|+1)+1 \big ).$$ 
En effet, $d_{v}F_{(s,v)} = \frac{\partial}{\partial t}-L_{t}(s)$ est bijective, lin\'eaire et lisse sur chaque ouvert $U$ de $S$. La famille de solutions $u$ peut s'\'ecrire alors comme \'etant $$u(s)=(d_{v}F_{(s,v)})^{-1} \big(f(s,\cdot) + L_{t}(s) u_{0}(s,\cdot)\big) + u_{0}(s,\cdot).$$ Comme l'op\'erateur $L_{t}(s) : C^{k+2, \alpha}_{\mathcal{V}_{Z}}([0,T]\times \overset{\circ}{Z}; E)\rightarrow C^{k, \alpha}_{\mathcal{V}_{Z}}([0,T]\times \overset{\circ}{Z}; E)$ est lisse et lin\'eaire, on a que 
$$\left\|  L_{t}(s) u_{0}(s,\cdot) \right\|_{k,\alpha}\leq \left\| L_{t}(s) \right\|  \left\| u_{0}(s,\cdot) \right\|_{k+2,\alpha}$$
et donc on d\'eduit que 
$$\left\| u(s) \right\|_{k+2,\alpha}\leq \sup_{ s \in S}\big( \left\| (d_{v}F_{(s,v)})^{-1} \right\|(\left\| L_{t}(s) \right\|+1)+1\big)( \left\| f(s,\cdot) \right\|_{k,\alpha} + \left\| u_{0}(s,\cdot) \right\|_{k+2,\alpha}).
$$Par compacit\'e de la vari\'et\'e $S,$ on obtient le supremum voulu.
\end{remq}
\begin{cor}\label{corrrr}
Sous les m\^emes hypoth\`{e}ses de la Proposition~\ref{prooop}, on suppose que $U\times Z$ est une vari\'et\'e \`a coins de profondeur $l$ associ\'ee \`a une famille d'hypersurfaces bordantes $(H_{i})_{i \in J}.$ Soit $\rho_{i}$ une fonction de d\'efinition de $H_{i},$ pour $i \in J$. Si $$f \in \rho^{\beta} (\log \rho)^{\eta} \bigg(\mathcal{A}^{{\mathcal{G}_{2}}_{\scriptscriptstyle{\vert  U  }}}_{\phg}\big(U,C^{k,\alpha}_{\mathcal{V}_{Z}}([0,T]\times \overset{\circ}{Z}; E)\big) + x_{\scriptscriptstyle{\vert U }}^{\vartheta} C_{b}^{\infty}\big( \overset{\circ}{U}; C^{k, \alpha}_{\mathcal{V}_{Z}}([0,T]\times \overset{\circ}{Z}; E)\big)\bigg),$$ o\`u $x$ est comme dans la Proposition~\ref{prooop}, $\rho^{\beta}=\prod_{i=1}^{l}\rho^{\beta_{i}}_{i},$ $( \log \rho)^{\eta}= \prod_{i=1}^{l}(\log \rho_{i})^{\eta_{i}},$ $\beta=(\beta_{1},...,\beta_{l})$ et $\eta=(\eta_{1},...,\eta_{l})$ avec $\beta_{i} >0$ et $\eta_{i} \in \mathbb{N}_{0}$ (ou $\beta_{i}=\eta_{i}=0$), alors $$u \in \rho^{\beta} (\log \rho)^{\eta} C_{b}^{\infty}\big(\overset{\circ}{U}; C^{k+2, \alpha}_{\mathcal{V}_{Z}}([0,T]\times \overset{\circ}{Z}; E)\big).$$
\end{cor}
\begin{proof}
On consid\`{e}re la famille d'op\'erateurs param\'etr\'ee par $s\in U$ $$\tilde{L}_{t}(s):=\rho^{-\beta} (\log \rho)^{-\eta} \circ L_{t}(s) \circ \rho^{\beta} (\log \rho)^{\eta}.$$ On a que $\tilde{L}_{t}(s) \in \Diff^{2}_{\mathcal{V}_{Z}}(\overset{\circ}{Z}; E)$ est une famille lisse d'op\'erateurs par la Remarque~\ref{lisse}, et aussi clairement uniform\'ement $\mathcal{V}_{Z}$\texttt{-}elliptiques car
$$\forall \xi \in T^{*}_{m}\overset{\circ}{Z} \setminus \{0\}, \;\;\sigma_{2}\big(\tilde{L}_{t}(s)\big)( \xi) =\sigma_{2}\big(L_{t}(s)\big)(\xi) = - \rho^{-\beta} (\log \rho)^{-\eta}  \rho^{\beta} (\log \rho)^{\eta}g^{ij}\xi_{i} \xi_{j}, $$
o\`u $g$ est une certaine $\mathcal{V}_{Z}$\texttt{-}m\'etrique. 
On peut se ramener alors \`a une famille d'\'equations \'equivalente \`a la famille d'\'equations d\'ecrite ci\texttt{-}dessus, \`a savoir
$$\frac{\partial \tilde{u}}{\partial t}-\tilde{L}_{t} \tilde{u} = \tilde{f}, \; \tilde{u}_{\scriptscriptstyle{\vert t=0 }}=\tilde{u}_{0}, \; \textrm{avec}
\;\tilde{u}=\rho^{-\beta} (\log \rho)^{-\eta} u, \; \tilde{u}_{0}=\rho^{-\beta} (\log \rho)^{-\eta} u_{0}, \; \tilde{f}=\rho^{-\beta} (\log \rho)^{-\eta} f,$$
et l'existence et l'unicit\'e de la solution d\'ecoule de la proposition pr\'ec\'edente.
\end{proof}
\begin{them} \label{them}
Soit $(M, S, \phi_{M}, \mathcal{V})$ un fibr\'e de structures de Lie fibr\'ees \`a l'infini associ\'e au fibr\'e $\phi_{M} : M \rightarrow S,$ tel que le rayon d'injectivit\'e des m\'etriques compatibles sur les fibres soit strictement positif. Soient $\mathcal{G},$ $\mathcal{G}_{1}$ et $\mathcal{G}_{2}$ des familles indicielles positives de $M$ et $L_{t}(s)$ une famille polyhomog\`{e}ne sur $[0,T] \times S$ d'op\'erateurs de $\Diff^{2}_{\mathcal {V},\: \mathcal{G}}(M; E)$ uniform\'ement $\mathcal{V}_{M_{s}}$\texttt{-}elliptiques sur $M_{s}:=\phi^{-1}_{M}(s), \forall s\in S.$ Alors pour $f \in \mathcal{A}^{\mathcal{G}_{2}}_{\phg}([0,T]\times M; E)$ et $u_{0} \in \mathcal{A}^{\mathcal{G}_{1}}_{\phg}(M; E),$ la famille d'\'equations
\begin{equation}\label{hjaaj}
\frac{\partial u}{\partial t}-L_{t} u = f, \; u_{\scriptscriptstyle{\vert t=0 }}=u_{0}
\end{equation} 
poss\`{e}de alors une unique solution $u \in \mathcal{A}^{\mathcal{K}}_{\phg}([0,T] \times M; E)$, o\`u $\mathcal{K}$ est la famille indicielle positive donn\'ee par $$ \mathcal{K}=\mathcal{G}_{\infty}+(\mathcal{G}_{1} \cup \mathcal{G}_{2}).$$
\end{them} 
\begin{proof}
D'abord, en consid\'erant $u - u_{0}$ au lieu de $u$, on peut toujours se ramener au cas o\`u $u_{0}= 0$. 
On proc\`{e}de par r\'ecurrence sur la profondeur $l$ de la vari\'et\'e $M,$ le cas $l=0$ d\'ecoulant trivialement de la Proposition~\ref{propiii}. 
On suppose donc que le r\'esultat est vrai pour toute vari\'et\'e de profondeur strictement plus petit que $l.$ Et on montre qu'alors, le r\'esultat est vrai au rang $l$. Il faut montrer que \\ $u \in  \mathcal{A}^{\mathcal{K}_{\scriptscriptstyle{\vert U\times Z }}}_{\phg}([0,T] \times U\times Z; E)$ sur chaque ouvert de trivialisation locale $U\times Z$ du fibr\'e de structures de Lie fibr\'ees \`a l'infini. Le cas $\dim Z = 0$ a \'et\'e fait dans la Proposition~\ref{priooop}.  On suppose que $U \times Z$ avec $\dim Z \geq 1$ est une vari\'et\'e \`a coins de profondeur $l$. Soit $(H_{i})_{i \in J}$ les hypersurfaces bordantes de $U \times Z.$
Puisque $\mathcal{K}$ est une famille indicielle positive, on a que $$\mathcal{A}^{\mathcal{K}_{\scriptscriptstyle{\vert U \times Z }}}_{\phg}([0,T] \times  U \times Z; E) \subsetneq C_{b}^{\infty}\big(\overset{\circ}{U}; C^{\infty}_{\mathcal{V}_{Z}}([0,T]\times \overset{\circ}{Z}; E)\big).$$ On sait d\'ej\`a par la Proposition~\ref{prooop} que l'\'equation poss\`{e}de une solution unique \\$u \in C_{b}^{\infty}\big( \overset{\circ}{U}; C^{\infty}_{\mathcal{V}_{Z}}([0,T]\times \overset{\circ}{Z}; E)\big)$. 
On fixe $i \in \{1,...,l\}.$ Alors on d\'efinit $ \{z_{j}\}_ {j \in \mathbb{N}_{0}} $ une suite strictement croissante de nombres r\'eels positifs ou nuls avec 
\begin{equation}
z_{0}=0 \; \textrm{et} \; (z_{j},0) \in \mathcal{K}_{\scriptscriptstyle{\vert U\times Z }}(H_{i})\; \; \forall j \in \mathbb{N},
\end{equation}
\begin{equation}
(z,k) \in \mathcal{K}_{\scriptscriptstyle{\vert U\times Z }}(H_{i}) \Longrightarrow z=z_{j} \textrm{\:pour\:un\:certain}\; j \in  \mathbb{N}_{0}.
\end{equation}
Pour $i \in \{1,...,l\}$ fix\'e, on va montrer par r\'ecurrence sur $j$ que $$\forall j \in \mathbb{N}_{0}, \;\exists u_{j} \in  \mathcal{A}^{\mathcal{K}_{\scriptscriptstyle{\vert U\times Z }}}_{\phg}([0,T] \times U \times Z; E) \textrm{\:tel\:que}\; u_{j}(0, \cdot) \equiv 0 \;  \textrm{et}$$
\begin{equation} \label{monqqq}
u-u_{j} \in \rho_{i}^{\vartheta}  C_{b}^{\infty}\big(\overset{\circ}{U}; C^{\infty}_{\mathcal{V}_{Z}}([0,T]\times \overset{\circ}{Z}; E)\big) \;\; \forall \vartheta < z_{j+1},
\end{equation}
o\`u $\rho_{i}$ est une fonction bordante de $H_{i}.$\\
On suppose maintenant que le r\'esultat est vrai jusqu'\`a l'ordre $j-1$ et alors on a que $$\exists u_{j-1} \in  \mathcal{A}^{\mathcal{K}_{\scriptscriptstyle{\vert U\times Z }}}_{\phg}([0,T] \times U \times Z; E)  \textrm{\:tel\:que}\; u_{j-1}(0, \cdot) \equiv 0 \;  \textrm{et} $$$$
u-u_{j-1} \in \rho_{i}^{\vartheta}  C_{b}^{\infty}\big(\overset{\circ}{U}; C^{\infty}_{\mathcal{V}_{Z}}([0,T]\times \overset{\circ}{Z}; E)\big) \;\; \forall \vartheta < z_{j}. $$
Pour le cas $j=0$, il est sous entendu ici qu'on prend $u_{-1}=0$, de sorte que  $$u-u_{-1} \in C_{b}^{\infty}\big(\overset{\circ}{U}; C^{\infty}_{\mathcal{V}_{Z}}([0,T]\times \overset{\circ}{Z}; E)\big).$$
On pose $v=u-u_{j-1}$ v\'erifiant 
\begin{equation} \label{moneq}
\frac{\partial v}{\partial t}-L_{t} v = f^{j}, \; f^{j}=f-\frac{\partial u_{j-1}}{\partial t}+L_{t} u_{j-1}.
\end{equation}
En utilisant la Remarque~\ref{field}, le fait que $\mathcal{G}+ \mathcal{K}=\mathcal{K}$ et que $f$ est polyhomog\`{e}ne, on a que $$f^{j} \in \mathcal{A}^{\mathcal{K}_{\scriptscriptstyle{\vert U\times Z }}}_{\phg}([0,T] \times U\times Z; E).$$ 
\'Etant donn\'e que $v \in  \rho_{i}^{\vartheta} C_{b}^{\infty}\big(\overset{\circ}{U}; C^{\infty}_{\mathcal{V}_{Z}}([0,T]\times \overset{\circ}{Z}; E)\big),$ $\forall \vartheta < z_{j}$, on obtient aussi $$f^{j} \in \rho_{i}^{\vartheta} C_{b}^{\infty}\big(\overset{\circ}{U}; C^{\infty}_{\mathcal{V}_{Z}}([0,T]\times \overset{\circ}{Z}; E)\big), \; \forall \vartheta < z_{j}.$$
Cela entra\^ine que $$\exists k \in \mathbb{N}_{0},\; f^{j} \in \rho_{i}^{z_{j}} (\log \rho_{i})^{k} C_{b}^{\infty}\big(\overset{\circ}{U}; C^{\infty}_{\mathcal{V}_{Z}}([0,T]\times \overset{\circ}{Z}; E)\big).$$
On suppose maintenant pour un instant que l'on sache d\'ej\`a que  $$v \in \mathcal{A}^{\mathcal{K}_{\scriptscriptstyle{\vert U\times Z }}}_{\phg}([0,T] \times U \times Z; E) \cap \rho_{i}^{z_{j}} (\log \rho_{i})^{k} C_{b}^{\infty}\big(\overset{\circ}{U}; C^{\infty}_{\mathcal{V}_{Z}}([0,T]\times \overset{\circ}{Z}; E)\big).$$
Dans ce cas, on peut restreindre la famille d'\'equations~\eqref{moneq} au coefficient d'ordre $\rho_{i}^{z_{j}} (\log \rho_{i})^{k}$ sur $H_{i}.$ 
On sait d\'ej\`a par la Proposition~\ref{munim}, que $(H_{i}, S_{i}, \phi_{i}, \mathcal {V}_{H_{i}})$ est un fibr\'e de structures de Lie fibr\'ees \`a l'infini par rapport \`a un certain fibr\'e $\phi_{i} : H_{i} \rightarrow S_{i}$ (pas n\'ecessairement trivial) de fibre typique d\'enot\'ee $Z_{i}.$ Si $\dim Z_{i}\geq 1$, en utilisant la Proposition~\ref{fibre} fibre par fibre, l'\'equation~\eqref{moneq} restreinte au coefficient d'ordre $\rho_{i}^{z_{j}} (\log \rho_{i})^{k}$ sur $H_{i}$ sera une famille d'\'equations $\mathcal{V}_{Z_{i}}$\texttt{-}paraboliques lin\'eaires localement de la forme :
\begin{equation}\label{ahana}
\frac{\partial v_{j}(s)}{\partial t}-\Psi_{t}(s) v_{j}(s) = f^{j}_{z_{j},k}(s), \; {v_{j}(s)}_{\scriptscriptstyle{\vert t=0}}  = 0,
\end{equation}
o\`u $v_{j}(s)$ et $f^{j}_{z_{j},k}(s) \in \mathcal{A}^{\mathcal{K}_{\scriptscriptstyle{\vert  Z_{i} }}}_{\phg}([0,T] \times Z_{i}; E)$ sont les coefficients du terme d'ordre $\rho_{i}^{z_{j}} (\log \rho_{i})^{k}$ de $v$ et $f^{j}$ respectivement restreint sur $\overset{\circ}{Z}_{i}$ et $\Psi_{t}(s)\in \Diff^{2}_{\mathcal {V}_{ Z_{i}},\: \mathcal{K}_{\scriptscriptstyle{\vert  Z_{i} }}}(Z_{i}; E)$ est une famille d'op\'erateurs diff\'erentiels $\mathcal{V}_{Z_{i}}$\texttt{-}elliptiques car elle est la restriction \`a l'ordre $0$ de l'op\'erateur $$\tilde{L}_{t}=\rho_{i}^{-z_{j}} (\log \rho_{i})^{-k}\:L_{t}\: \rho_{i}^{z_{j}} (\log \rho_{i})^{k}\; \textrm{satisfaisant} \;\sigma_{2}(\tilde{L}_{t})=\sigma_{2}(L_{t}).$$
Vu que $H_{i}$ est une vari\'et\'e de profondeur $l-1$, par notre hypoth\`{e}se de r\'ecurrence, la famille d'\'equations~\eqref{ahana} poss\`{e}de une unique solution polyhomog\`{e}ne $v_{j} \in \mathcal{A}^{\mathcal{K}_{\scriptscriptstyle{\vert  H_{i} }}}_{\phg}([0,T] \times H_{i}; E)$. \`A strictement parl\'e, il faut savoir que les m\'etriques induites sur les fibres du fibr\'e de structures de Lie de $H_{i}$ ont un rayon d'injectivit\'e positif. Or, si ce n'\'etait pas le cas, comme la courbure et ses d\'eriv\'ees sont born\'ees, on aurait une suite de boucles g\'eod\'esiques dont la longueur tend vers z\'ero. Cette suite induirait une suite correspondante de boucles g\'eod\'esiques pour une famille de m\'etriques sur $\phi^{-1}_{M}(s)$ compatibles avec la structure de Lie \`{a} l'infini, contredisant notre hypoth\`{e}se sur le rayon d'injectivit\'e de telles m\'etriques.
Si plut\^ot $\dim Z_{i}=0,$ alors~\eqref{ahana} devient une famille d'\'equations diff\'erentielles ordinaires et on peut appliquer la Proposition~\ref{priooop} pour conclure que $v_{j} \in \mathcal{A}^{\mathcal{K}_{\scriptscriptstyle{\vert  H_{i} }}}_{\phg}([0,T] \times H_{i}; E).$
Maintenant, m\^eme si on ne sait pas a priori si $v$ est polyhomog\`{e}ne, on remarque que la famille d'\'equation~\eqref{ahana} a malgr\'e tout un sens et que sa solution $v_{j}$ donne le candidat naturel qui est la restriction \`a l'ordre $\rho_{i}^{z_{j}} (\log \rho_{i})^{k}$ de $v$. Pour montrer
que c'est bien le cas, on pose $w = v - w_{j}$ avec $w_{j} = \rho_{i}^{z_{j}} (\log \rho_{i})^{k} \Xi_{H_{i}} (v_{j})$, o\`u $\Xi_{H_{i}}$ est une application d'extension comme dans la D\'efinition~\ref{extension},
de sorte que $w$ satisfait \`a l'\'equation d'\'evolution 
$$\frac{\partial w}{\partial t}-L_{t} w = h^{j},\;  w_{\scriptscriptstyle{\vert t=0}} =0, \; h^{j}=f^{j}-\frac{\partial  w_{j} }{\partial t}+L_{t}  w_{j}.$$
Or, comme $f^{j}$ et $w_{j}$ appartiennent \`a l'espace $$\rho_{i}^{z_{j}} (\log \rho_{i})^{k} C_{b}^{\infty}\big( \overset{\circ}{U}; C^{\infty}_{\mathcal{V}_{Z}}([0,T]\times \overset{\circ}{Z}; E)\big) \cap \mathcal{A}^{\mathcal{K}_{\scriptscriptstyle{\vert  U \times Z }} }_{\phg}([0, T] \times U \times Z; E),$$
on a alors que $h^{j}$ est aussi un \'el\'ement de cet espace. Gr\^ace \`a la d\'efinition de $w_{j}$, on a aussi $h^{j} \in  \rho_{i}^{z_{j}} (\log \rho_{i})^{k-1} C_{b}^{\infty}\big( \overset{\circ}{U}; C^{\infty}_{\mathcal{V}_{Z}}([0,T]\times \overset{\circ}{Z}; E)\big)$ car la restriction $h^{j}$ sur $\overset{\circ}{H}_{i}$ admet un coefficient d'ordre  $\rho_{i}^{z_{j}} (\log \rho_{i})^{k}$ nul. D'apr\`{e}s le Corollaire~\ref{corrrr}, on a alors que $$w \in \rho_{i}^{z_{j}} (\log \rho_{i})^{k-1} C_{b}^{\infty}\big( \overset{\circ}{U}; C^{\infty}_{\mathcal{V}_{Z}}([0,T]\times \overset{\circ}{Z}; E)\big).$$
En r\'ep\'etant cet argument $k$ fois, on obtient alors $$r_{k}, r_{k-1}, ...,r_{0} \in  \mathcal{A}^{\mathcal{K}_{\scriptscriptstyle{\vert  H_{i} }}}_{\phg}([0,T] \times H_{i}; E)$$ avec $r_{k}= v_{j} $ de sorte que $$\hat{v}=v- \sum \limits_{p=0}^{k} \Xi_{H_{i}}(r_{p})\rho_{i}^{z_{j}} (\log \rho_{i})^{p}, \;\;\; r_{p}(0,\cdot) \equiv 0$$ a pour \'equation d'\'evolution
$$\frac{\partial \hat{v}}{\partial t}-L_{t} \hat{v} = \hat{f}^{j}, \;  \hat{v} (0,\cdot) \equiv 0 \; \textrm{avec}\; $$$$\hat{f}^{j} \in \rho_{i}^{z_{j}} C_{b}^{\infty}\big( \overset{\circ}{U}; C^{\infty}_{\mathcal{V}_{Z}}([0,T] \times \overset{\circ}{Z}; E)\big)\cap  \mathcal{A}^{\mathcal{K}_{\scriptscriptstyle{\vert  U \times Z }} }_{\phg}([0, T] \times U \times Z ; E)$$ et $ \rho_{i}^{-z_{j}} \hat{f}^{j}_{\scriptscriptstyle{\vert \overset{\circ}{H}_{i}}}=0.$
Ainsi on a $\hat{f}^{j} \in \rho_{i}^{\vartheta} C_{b}^{\infty}\big( \overset{\circ}{U}; C^{\infty}_{\mathcal{V}_{Z}}([0,T]\times \overset{\circ}{Z}; E)\big) \;\; \forall \vartheta < z_{j+1}$, et par le Corollaire~\ref{corrrr}, on d\'eduit que $\hat{v} \in \rho_{i}^{\vartheta} C_{b}^{\infty}\big( \overset{\circ}{U}; C^{\infty}_{\mathcal{V}_{Z}}([0,T]\times \overset{\circ}{Z}; E)\big) \;\; \forall \vartheta < z_{j+1}$.\\
Il suffit donc de prendre $$\displaystyle u_{j}=u_{j-1} + \sum \limits_{p=0}^{k} \Xi_{H_{i}}(r_{p})\rho_{i}^{z_{j}} (\log \rho_{i})^{p},$$
ce qui termine la construction de suite $u_{j}$ de~\eqref{monqqq} et compl\`{e}te la d\'emonstration. 
\end{proof}
\begin{cor} \label{asass}
Si dans l'\'equation~\eqref{hjaaj}, $L_{t}(s)$ est une famille lisse sur $[0,T] \times S$ d'op\'erateurs uniform\'ement $\mathcal{V}_{M_{s}}$\texttt{-}elliptiques de $\Diff^{2}_{\mathcal {V}_{M_{s}}}(M_{s}; E)$, $u_{0} \in C^{\infty}(M; E)$ et $f \in  C^{\infty}([0,T]\times M; E)$, alors la solution de cette \'equation est dans $ C^{\infty}([0,T] \times M; E)$.
\end{cor}
\begin{proof}
Il suffit de prendre $\mathcal{G} = \mathcal{G}_{1} = \mathcal{G}_{2} = \{\mathbb{N}_{0} \times \{0\} \;\textrm{pour}\; H \in  \mathcal{M}_{1}(M) \}$ dans le Th\'eor\`{e}me~\ref{them}.  
\end{proof}
\begin{cor} \label{laas}
Soit $(M,\mathcal{V})$ une structure de Lie fibr\'ee \`a l'infini telle que le rayon d'injectivit\'e des m\'etriques compatibles est strictement positif. Soient $\mathcal{G},$ $\mathcal{G}_{1}$ et $\mathcal{G}_{2}$ des familles indicielles positives de $M$  et $\{L_{t} : t \in [0,T] \}$ une famille d'op\'erateurs uniform\'ement $\mathcal{V}_{SF}$\texttt{-}elliptiques de $\Diff^{2}_{\mathcal {V}_{SF},\: \mathcal{G}}(M; E)$. Alors pour $u_{0} \in \mathcal{A}^{\mathcal{G}_{1}}_{\phg}(M; E)$ et $f \in \mathcal{A}^{\mathcal{G}_{2}}_{\phg}([0,T]\times M; E),$ la famille d'\'equations 
\begin{equation}\label{hjaj}
\frac{\partial u}{\partial t}-L_{t} u = f, \; u_{\scriptscriptstyle{\vert t=0 }}=u_{0},
\end{equation} 
poss\`{e}de une unique solution $u \in \mathcal{A}^{\mathcal{K}}_{\phg}([0,T] \times M; E)$, o\`u $\mathcal{K}$ est la famille indicielle positive donn\'ee par $ \displaystyle \mathcal{K}=\mathcal{G}_{\infty}+(\mathcal{G}_{1} \cup \mathcal{G}_{2}).$
\end{cor}
\begin{proof}
C'est une cons\'equence directe du Th\'eor\`{e}me~\ref{them} en prenant la base $S$ un point de fibr\'e $\phi_{M}.$
\end{proof}
\begin{cor}
Si dans l'\'equation~\eqref{hjaj}, la famille $\{L_{t} : t\in [0,T] \}$ d'op\'erateurs uniform\'ement $\mathcal{V}_{SF}$\texttt{-}elliptiques est dans $\Diff^{2}_{\mathcal {V}_{SF}}(M; E)$, $u_{0} \in C^{\infty}(M; E)$ et $f \in  C^{\infty}([0,T]\times M; E)$, alors la solution de cette \'equation est dans $ C^{\infty}([0,T] \times M; E)$.
\end{cor}
\begin{proof}
Dans le Corollaire~\ref{laas}, on prend la m\^eme famille indicielle $\mathcal{G}$ que dans la preuve du Corollaire~\ref{asass}.
\end{proof}

\section{Polyhomog\'en\'eit\'e des solutions des \'equations paraboliques quasi\texttt{-}lin\'eaires}

Cette section d\'eveloppe la th\'eorie des \'equations paraboliques quasi\texttt{-}lin\'eaires d\'etermin\'ees par une structure de Lie \`a l'infini. En s'inspirant beaucoup de l'article de Bahuaud~\cite{bahuaud2011ricci}, nous \'etudions l'existence et l'unicit\'e de leurs solutions et nous d\'eterminons quelles contraintes nous assurent que la polyhomog\'en\'eit\'e sera pr\'eserv\'ee localement jusqu'au bord. Ces r\'esultats jouent un r\^ole important dans des probl\`{e}mes d'\'evolution g\'eom\'etrique concrets tel que le flot de Ricci.

Pour se r\'echauffer, consid\'erons d'abord le cas d'\'equations ordinaires non lin\'eaires au lieu d'EDP paraboliques. 
Soient $T>0$ un r\'eel positif et $f : [0,T] \times \mathbb{R}^{n} \rightarrow \mathbb{R}^{n}$ une application continue. 
Pour $x_{0}\in\mathbb{R}^{n},$ consid\'erons l'EDO
\begin{equation} \label{nonlin}
\frac{\partial }{\partial t}x = f(t,x) , \; x(0)=x_{0}.
\end{equation}
Pour une telle \'equation, nous avons le r\'esultat bien connu d'existence et d'unicit\'e suivant (voir par exemple Th\'eor\`{e}me A~\cite{simmons2016differential} \`a la page 626).
\begin{them}\label{thee} (Cauchy\texttt{-}Lipschitz)
Si $f(t,\cdot )$ est lipschitzienne (uniform\'ement en $t$) sur la boule ferm\'ee centr\'ee en $x_{0}$ de rayon $r,$ $B(x_{0},r):=\{x\in \mathbb{R}^{n} : \left\| x - x_{0} \right\| \leq r \},$ et que 
$$c:=\sup \{ \left\| f(t,x) \right\| : (t,x) \in  [0,T] \times B(x_{0},r) \} < \infty,$$ alors l'\'equation~\eqref{nonlin} poss\`{e}de une unique solution $x \in C^{1}\big([0,\alpha), B(x_{0},r)\big),$ avec 
\begin{equation} \label{nouno}
\alpha =\min(\frac{1}{\kappa},\frac{r}{c}),
\end{equation}
o\`u $\kappa$ est la constante de Lipschitz de $f$ sur $B(x_{0},r).$
De plus, la suite d\'efinie par 
$$y_{0}=x_{0},  \;\; y_{i}(t)=x_{0} + \int_{t_{0}}^{t} f\big(s,y_{i-1}(s)\big) \, ds, \;\; i \geq 1, \;\; | t-t_{0} | \leq  \alpha,$$
converge uniform\'ement sur l'intervalle $| t-t_{0} | \leq \alpha$ vers la solution $x$ de~\eqref{nonlin}.
\end{them}
Une solution $x: [0,\alpha) \rightarrow \mathbb{R}^{n}$ de l'\'equation~\eqref{nonlin} est dite maximale si elle n'est pas prolongeable, c'est\texttt{-}\`a\texttt{-}dire que si  $y: [0,\tau)  \rightarrow \mathbb{R}^{n}$ est une autre solution, alors $\tau\leq \alpha.$ Bien entendu, d'apr\`{e}s l'unicit\'e de la solution $x=y$ sur $[0,\tau).$ 
Par cons\'equent, toute \'equation de la forme~\eqref{nonlin} poss\`{e}de toujours une solution maximale. De plus si $f$ est de classe $C^{k},$ alors la solution est de classe $C^{k+1}.$
\begin{cor}\label{prooopp}
Soient $M$ une vari\'et\'e \`a coins compacte, $\mathcal{G}$ une famille indicielle positive associ\'ee et $\pi : E \rightarrow M$ un fibr\'e vectoriel. Soit $\pi^{*}\mathcal{G}$ la famille indicielle de $E$ qui \`a une hypersurface bordante de la forme $\pi^{-1}(H)$ associe l'ensemble indiciel $ \mathcal{G}(H)$. Pour $u_{0} \in \mathcal{A}_{\phg}^{\mathcal{G}}(M ; E)$ et $f \in \mathcal{A}_{\phg}^{\pi^{*}\mathcal{G}} (E; \pi^{*}E),$ 
on consid\`{e}re la famille d'EDOs  
\begin{equation}\label{eeee}
\frac{\partial u(x,t)}{\partial t} = \pi_ {*}f\big(u(x,t)\big), \; u_{\scriptscriptstyle{\vert t=0 }}=u_{0},
\end{equation}
o\`u $\pi_{*} : \pi^{*} E \rightarrow E$ est l'application naturelle induite par $\pi$ envoyant la fibre $(\pi^{*}E)_{e}$ de $\pi^{*}E$ canoniquement sur $E_{\pi(e)}.$
Alors il existe $\tau >0$ tel que la famille d'\'equations~\eqref{eeee} poss\`{e}de une unique solution maximale\\ $u \in  C_{b}^{\infty}\big(\overset{\circ}{M}, E \otimes C^{\infty}([0,\tau))\big) = \bigcap_{k\in\mathbb{N}_{0} } C_{b}^{\infty}\big(\overset{\circ}{M}, E \otimes C^{k+1}([0,\tau))\big).$
\end{cor}
\begin{proof}
Soit $x\in M$ fix\'e. En vertu du Th\'eor\`{e}me~\ref{thee}, pour chaque \'equation de~\eqref{eeee}, il existe une constante $\tau_{x} >0$ telle qu'elle poss\`{e}de une unique solution $u(x, \cdot) \in  E_{x} \otimes C^{k+1}([0,\tau_{x}))$. Comme $M$ est compacte, il est alors facile de trouver, \`a partir de~\eqref{nouno}, un temps minimal $\tau$ d'existence qui fonctionne pour la famille d'\'equations de~\eqref{eeee}. Il reste \`a montrer que la solution $u$ est dans $C_{b}^{\infty}\big(\overset{\circ}{M}, E \otimes C^{k+1}([0,\tau))\big).$
En consid\'erant $u - u_{0}$ au lieu de $u$, on peut toujours se ramener au cas o\`u $u_{0}= 0.$ On prend un ouvert trivialisant quelconque $U \subset \overset{\circ}{M}$ du fibr\'e $E\rightarrow M$ de sorte que $E_{\scriptscriptstyle{\vert U}}\simeq U\times V$ et $\pi^{*}E_{\scriptscriptstyle{\vert E_{\scriptscriptstyle{\vert U}}}}\simeq U\times V\times V$. Pour $k \geq 0$ arbitraire, on consid\`{e}re la fonction $$F: U \times V \otimes C^{k+1}([0,\tau)) \rightarrow  V \otimes C^{k}([0,\tau))$$ d\'efinie par $\displaystyle F(x,v):= \frac{\partial v}{\partial t}-\pr_{2}\big(\pi_{*}f(x,v)\big),$ o\`u $\pr_{2} : U\times V \rightarrow V$ est la projection sur le second facteur. Il est clair que $F$ est de classe $C^{\infty}$ sur $U$ et que $F(x,\cdot)$ est de classe $C^{\infty}$. Comme $d_{v}F_{(x,v)} =\frac{\partial}{\partial t}-d_{v}f(x,v)$ induit une application lin\'eaire bijective de $V \otimes C^{k+1}([0,\tau))$ dans $V \otimes C^{k}([0,\tau))$, on peut appliquer le th\'eor\`{e}me des fonctions implicites~\cite{krantz2012implicit} pour obtenir que $u \in C^{\infty}\big(U;V \otimes C^{k+1}_{\mathcal{V}_{Z}}([0,\tau))\big).$ On prend finalement le cas $U \subset M$ tel que $U\cap \partial M \neq 0.$ Comme $f\in \mathcal{A}_{\phg}^{\pi^{*}\mathcal{G}} (E; \pi^{*}E),$ on peut alors se ramener au cas pr\'ec\'edent. En utilisant l'\'eclatement logarithmique total et en prolongeant ainsi la famille d'\'equations \`a $\Omega_{\log}$ comme dans la Proposition~\ref{prooop}, on d\'eduit que $u \in  C^{\infty}\big({\Omega}_{\log}; V \otimes C^{k+1}([0,\tau))\big) \subset C_{\mathcal{V}_{pc}}^{\infty}\big({U}_{\log}; V \otimes C^{k+1}([0,\tau))\big) = C_{b}^{\infty}\big(\overset{\circ}{U}; V \otimes C^{k+1}([0,\tau))\big).$
\end{proof}

Soit $\overset{\circ}{M}$ une vari\'et\'e riemannienne avec une structure de Lie \`a l'infini $(M,\mathcal{V})$ telle que le rayon d'injectivit\'e des m\'etriques compatibles est strictement positif. Soient $E \rightarrow M$ un fibr\'e vectoriel et $\{L_{t} : t\in [0,T]\} $ une famille d'op\'erateurs uniform\'ement $\mathcal{V}$\texttt{-}elliptiques de $\Diff^{2}_{\mathcal{V}}(\overset{\circ}{M}; E)$. 
Nous allons \'etudier des \'equations paraboliques quasi\texttt{-}lin\'eaires de la forme 
\begin{equation} \label{quasi}
(\frac{\partial }{\partial t}-L_{t}) u = Q(u)u + f, \; u_{\scriptscriptstyle{\vert t=0 }}=0,
\end{equation}
o\`u $f \in C_{\mathcal{V}}^{\infty}([0,T]\times \overset{\circ}{M}; E)$ est ind\'ependante de $u$ et $$Q: C_{\mathcal{V}}^{k+2,\alpha}([0,T]\times \overset{\circ}{M}; E) \times C_{\mathcal{V}}^{k+2,\alpha}([0,T]\times \overset{\circ}{M}; E)  \rightarrow C_{\mathcal{V}}^{k,\alpha}([0,T]\times \overset{\circ}{M}; E)$$ est une application telle que  
$\forall w, w', w'' \in   C_{\mathcal{V}}^{k+2,\alpha}([0,T]\times \overset{\circ}{M}; E),$
\begin{equation}\label{zqwe}
Q(w)(aw'+bw'')=aQ(w)w'+bQ(w)w'', \;\;a, b \in \mathbb{R},
\end{equation}
\begin{equation}\label{quad}
\left\| Q(w)w\right\|_{k,\alpha} \leq c \left\| w \right\|^{2}_{k+2,\alpha} \; \textrm{et}
\end{equation}
\begin{equation}\label{quaad}
 \left\| Q(w)w-Q(w')w'\right\|_{k,\alpha} \leq C \max\{ \left\| w \right\|_{k+2,\alpha},  \left\| w' \right\|_{k+2,\alpha}\} \left\| w-w' \right\|_{k+2, \alpha}, 
\end{equation}
pour $c$ et $C$ deux constantes r\'eelles. 
\begin{prop} \label{quuasi}
Il existe $\tau\in (0,T]$ tel que l'\'equation~\eqref{quasi} admet une unique solution $$u\in C_{\mathcal{V}}^{k+2,\alpha}([0,\tau)\times \overset{\circ}{M}; E).$$
\end{prop}
\begin{proof}
On suit l'approche de Bahuaud~\cite{bahuaud2011ricci}. En utilisant la Proposition~\ref{propiii} et le Corollaire~\ref{corr}, il existe une application lin\'eaire $$\Psi : C_{\mathcal{V}}^{k,\alpha}([0,T]\times \overset{\circ}{M}; E) \rightarrow  C_{\mathcal{V}}^{k+2,\alpha}([0,T]\times \overset{\circ}{M}; E)$$ telle que $\displaystyle (\frac{\partial }{\partial t}-L_{t} )\Psi(h)=h,$ $\Psi(h)_{\scriptscriptstyle{\vert t=0 }}=0$ et satisfait \`a l'estimation de Schauder 
\begin{equation} \label{scfauder}
\left\| \Psi(h) \right\|_{k+2,\alpha}\leq \kappa  \left\| h \right\|_{k,\alpha},
\end{equation}
pour une certaine constante $\kappa >0.$
Ainsi, l'\'equation~\eqref{quasi} est alors \'equivalente \`a $$u=\Psi (Q(u)u + f).$$
On d\'efinit l'application $\mathcal{Y}$ par $\forall \nu \in C_{\mathcal{V}}^{k+2,\alpha}([0,\tau)\times \overset{\circ}{M}; E),$
$$\mathcal{Y}(\nu)=\Psi (Q(\nu)\nu + f).$$
Pour certains $K$ et $\tau$ deux nombres r\'eels strictement positifs \`a d\'eterminer, on consid\`{e}re
$$\mathcal{Z}_{K, \tau} := \{ \nu \in C_{\mathcal{V}}^{k+2,\alpha}([0,\tau)\times \overset{\circ}{M}; E) \mid \nu (0)=0, \;\left\| \nu \right\|_{k+2, \alpha}\leq K\}.$$
Si $\nu \in \mathcal{Z}_{K, \tau}$ alors $u=\mathcal{Y}(\nu)$ est la solution de $$\frac{\partial u}{\partial t}-L_{t}u= Q(\nu)\nu + f , \; u(0,\cdot) \equiv 0.$$
On suit le m\^eme raisonnement des Lemmes 4.5 et 4.6 de~\cite{bahuaud2011ricci} pour prouver que $\mathcal{Y} : \mathcal{Z}_{K, \tau} \rightarrow \mathcal{Z}_{K, \tau} $ est une application contractante. En effet, soit $\nu \in \mathcal{Z}_{K, \tau} $ et on d\'efinit $$u_{1}:=\Psi (Q(\nu)\nu), $$ $$u_{2}:=\Psi(f).$$
En particulier, $u_{1}$ est la solution de l'\'equation parabolique
$$\frac{\partial u_{1}}{\partial t}-L_{t}u_{1}= Q(\nu)\nu, u_{1}(0, \cdot) \equiv 0.$$
\`A l'aide de l'estimation~\eqref{quad} de l'application $Q$ et l'estimation de Schauder~\eqref{scfauder}, on a que 
\begin{center}
$
\begin{array}{rcl}
  \left\| u_{1} \right\|_{k+2, \alpha}& \leq& \kappa \left\| Q(\nu)\nu  \right\|_{k, \alpha}\\
                                                                                                                                                 & \leq& c \kappa   \left\| \nu \right\|^{2}_{k+2, \alpha}\\
                                                                                                                                                 & \leq& c \kappa  K  \left\| \nu \right\|_{k+2, \alpha}.
 \end{array}
 $
\end{center}
On prend $K$ suffisamment petit pour que $c \kappa K \leq \frac{1}{2}$, ainsi 
\begin{equation} \label{zinne}
\left\| u_{1} \right\|_{k+2, \alpha}\leq \frac{1}{2} K.
\end{equation}
De m\^eme, $u_{2}$ est la solution de l'\'equation parabolique
$$\frac{\partial u_{2}}{\partial t}-L_{t}u_{2}=  f, u_{2}(0) \equiv 0.$$
En utilisant $u_{2}(0) \equiv 0$ et l'estimation de Schauder~\eqref{scfauder}, on a que  
\begin{center}
$
\begin{array}{rcl}
  \left\| u_{2} (t)\right\|_{k+2, \alpha}& \leq& \displaystyle{\int_{0}^{t}}\left\| \frac{\partial u_{2}}{\partial s} (s) \right\|_{k+2, \alpha} ds\\
                                                                                                                                                 & \leq&  \tau \left\|\displaystyle \frac{\partial u_{2}}{\partial s}  \right\|_{k+2, \alpha}\\
                                                                                                                                                 & \leq&  \tau \left\| u_{2}   \right\|_{k+4, \alpha}\\
                                                                                                                                                 & \leq&  \tau \kappa \left\| f  \right\|_{k+2, \alpha}.
 \end{array}
 $
\end{center}
On prend $\tau$ assez petit pour que $\tau \kappa \left\| f  \right\|_{k+2, \alpha} \leq \frac{1}{2}K,$ de sorte que
\begin{equation} \label{zinebb}
\left\| u_{2} \right\|_{k+2, \alpha} \leq \frac{1}{2} K.
\end{equation}
Par~\eqref{zinne} et~\eqref{zinebb}, on d\'eduit que l'application $\mathcal{Y} : \mathcal{Z}_{K, \tau}  \rightarrow \mathcal{Z}_{K, \tau} $ est bien d\'efinie. Aussi, pour $K$ suffisamment petit tel que $K\leq \frac{1}{2\kappa C}$, elle est contractante gr\^ace \`a l'estimation~\eqref{quaad}, 
\begin{center}
$
\begin{array}{rcl}
\left\| \mathcal{Y} (\nu_{1}) - \mathcal{Y}(\nu_{2})  \right\|_{k+2, \alpha} & \leq &  \left\| \Psi (Q(\nu_{1})\nu_{1} - Q(\nu_{2})\nu_{2}  )\right\|_{k+2, \alpha}\\
             & \leq &\kappa \left\| Q(\nu_{1})\nu_{1} - Q(\nu_{2})\nu_{2}  \right\|_{k, \alpha}\\
             & \leq &  \kappa C  \max\{ \left\| \nu_{1}\right\|_{k+2,\alpha};  \left\| \nu_{2} \right\|_{k+2,\alpha}\}  \left\| \nu_{1} - \nu_{2}  \right\|_{k+2, \alpha}\\
             & \leq & \kappa C K \left\| \nu_{1} - \nu_{2}  \right\|_{k+2, \alpha}.   
             
\end{array}
$
\end{center}
Par le Th\'eor\`{e}me du point fixe de Banach, on voit donc que pour $$\tau  \leq \displaystyle  \frac{K}{2\kappa (\left\| f \right\|_{k+2, \alpha}+1)}\; \textrm{et} \; K \leq \frac{1}{2 \kappa(c+C) },$$ l'\'equation~\eqref{quasi} poss\`{e}de une unique solution $u \in  C_{\mathcal{V}}^{k+2,\alpha}([0,\tau)\times \overset{\circ}{M}; E).$ 
\end{proof}
\begin{remq}
Puisque l'ellipticit\'e est une condition ouverte, en prenant $\tau$ assez petit, on peut aussi supposer que $L_{t} + Q(u)$ est elliptique $\forall t \in [0, \tau).$
\end{remq}
\begin{prop} \label{wakwak}
Si la solution $u$ de~\eqref{quasi} peut se r\'e\'ecrire comme \'etant la solution d'une \'equation parabolique quasi\texttt{-}lin\'eaire 
\begin{equation} \label{kaw}
\frac{\partial u}{\partial t} - a(u,\nabla u) \cdot \nabla^{2}u + b(u, \nabla u)=f,
\end{equation}
o\`u $a(u,\nabla u) \in C_{\mathcal{V}}^{k+1,\alpha}([0, \tau) \times\overset{\circ}{M};T^{(2,0)}\overset{\circ}{M} \otimes E)$ et $b(u, \nabla u) \in C_{\mathcal{V}}^{k+1,\alpha}([0, \tau)\times\overset{\circ}{M}; E)$ d\`{e}s que\\ $u\in C_{\mathcal{V}}^{k+2,\alpha}([0,\tau)\times \overset{\circ}{M}; E),$ alors la solution $u$ est en fait dans $C_{\mathcal{V}}^{\infty}([0,\tau)\times \overset{\circ}{M}; E).$
\end{prop}
\begin{proof}
Comme $a(u,\nabla u)$ et $b(u,\nabla u)$ sont de classe $C_{\mathcal{V}}^{k+1,\alpha}$ d\`{e}s que $u\in C_{\mathcal{V}}^{k+2,\alpha}([0,\tau)\times \overset{\circ}{M}; E),$ on obtient que $u\in C_{\mathcal{V}}^{k+3,\alpha}([0,\tau)\times \overset{\circ}{M}; E)$ par la Proposition~\ref{propiii}. Ainsi, on d\'eduit la r\'egularit\'e voulue par un argument de bootstrap.
\end{proof}
Soit $(M, S, \phi_{M}, \mathcal{V})$ un fibr\'e de structures de Lie fibr\'ees \`a l'infini associ\'e \`a un fibr\'e $\phi_{M} : M \rightarrow S,$ tel que le rayon d'injectivit\'e des m\'etriques compatibles sur les fibres soit strictement positif. Soient $\Psi_{U} : \phi_{M}^{-1}(U) \rightarrow U \times Z$ une trivialisation locale du fibr\'e de structures de Lie fibr\'ees \`a l'infini, pour $U$ un ouvert de $S$, et $(Z,\mathcal{V}_{Z})$ la structure de Lie \`a l'infini de la fibre $Z$. Soient $f\in \mathcal{A}^{\mathcal{F}_{\scriptscriptstyle{\vert U }}}_{\phg}\big(U,C^{k,\alpha}_{\mathcal{V}_{Z}}([0,T]\times \overset{\circ}{Z}; E)\big)$ et $L_{t}(s)$ une famille polyhomog\`{e}ne sur $[0,T] \times U$ d'op\'erateurs uniform\'ement $\mathcal{V}_{Z}$\texttt{-}elliptiques dans $\Diff^{2}_{\mathcal{V}_{Z}}( Z; E)$ par rapport \`a $\mathcal{G}_{\scriptscriptstyle{\vert U}},$ o\`u $\mathcal{F}$ et $\mathcal{G}$ deux familles indicielles positives de $S$ (avec ensemble indiciel $\mathbb{N}_{0}$ en $t=0$ et $t=T).$  On consid\`{e}re $Q$ une famille d'applications sur $\big(C_{\mathcal{V}}^{\infty}([0, T] \times\overset{\circ}{Z}; E)\big)^{2}$ d\'efinie par
\begin{equation} \label{poloo}
Q(u)v=a(\cdot,u,\nabla u)\cdot \nabla^{2}v + b(\cdot,u, \nabla u)\cdot \nabla v + c(\cdot,u)\cdot v,
\end{equation}
o\`u la d\'ependance des familles $a,$ $b$ et $c$ par rapport $u$ et $\nabla u$ est lisse de sorte qu'en appliquant la d\'erivation en cha\^ine, on a que, pour tout $s\in U,$ $a(s,u,\nabla u)\in C_{\mathcal{V}}^{k,\alpha}([0, T] \times\overset{\circ}{Z};T^{(2,0)}\overset{\circ}{Z} \otimes E),$\\ $b(s,u, \nabla u) \in C_{\mathcal{V}}^{k, \alpha}([0, T] \times\overset{\circ}{Z};T^{*}\overset{\circ}{Z} \otimes E)$ et $c(s,u)\in C_{\mathcal{V}}^{k,\alpha}([0, T] \times\overset{\circ}{Z}; E)$ d\`{e}s que $u\in C_{\mathcal{V}}^{k+1,\alpha}([0,T]\times \overset{\circ}{Z}; E).$
\begin{hyp} \label{hypp}
On suppose que les familles $a$, $b$ et $c$ de~\eqref{poloo} sont polyhomog\`{e}nes en $s\in U$ par rapport \`a $\mathcal{G}_{\scriptscriptstyle{\vert U}},$ et que $a(\cdot,0,0)_{\scriptscriptstyle{\vert t=0 }}=0$, $b(\cdot,0,0)_{\scriptscriptstyle{\vert t=0 }}=0$ et $c(\cdot,0)_{\scriptscriptstyle{\vert t=0 }}=0.$
On suppose aussi que la famille $Q$ satisfait~\eqref{quad} et~\eqref{quaad} uniform\'ement en $s \in U.$
\end{hyp}
\begin{prop} \label{sisisi} 
Soit $Q$ une famille d'applications donn\'ee par~\eqref{poloo} satisfaisant \`a l'Hypoth\`{e}se~\ref{hypp}. Il existe alors $\tau \in (0,T]$ tel que la famille d'\'equations $\mathcal{V}_{Z}$\texttt{-}paraboliques quasi\texttt{-}lin\'eaires param\'etr\'ee par $s\in U$,
\begin{equation} \label{quassi} 
(\frac{\partial }{\partial t}-L_{t}) u =Q(u)u + f, \; u_{\scriptscriptstyle{\vert t=0 }}= 0, 
\end{equation}
poss\`{e}de une unique solution $u \in C_{b}^{\infty}\big(\overset{\circ}{U},C^{k+2,\alpha}_{\mathcal{V}_{Z}}([0,\tau) \times \overset{\circ}{Z}; E)\big).$
\end{prop}
\begin{proof}
Pour chaque $s\in U$ fix\'e, on sait par la Proposition~\ref{quuasi} qu'il existe $\tau_{s}>0$ tel que~\eqref{wakwak} poss\`{e}de une unique solution $u(s) \in C_{\mathcal{V}_{Z}}^{\infty}([0,\tau_{s}) \times \overset{\circ}{Z}; E).$ 
On suppose pour l'instant que $\displaystyle \tau=\inf_{ s \in U}(\tau_{s}) >0.$
On consid\`{e}re $\mathcal{Y} : U \times \mathcal{Z}_{K, \tau}  \rightarrow \mathcal{Z}_{K, \tau} $ l'application d\'efinie par 
$$\mathcal{Y}(\cdot, \nu)= \Psi (Q(\cdot,\nu)(\nu) + f )\;\forall \nu \in C_{\mathcal{V}_{Z}}^{k, \alpha}([0,\tau)\times \overset{\circ}{Z}; E).$$
On sait d\'ej\`a, pour $s\in U$ fix\'e, que $\mathcal{Y}(s)$ est contractante. De plus, par la Proposition~\ref{prooop}, $\mathcal{Y}(\cdot, \nu)$ est lisse sur $\overset{\circ}{U}$ car elle est \'equivalente \`a la solution d'une famille d'\'equations paraboliques lin\'eaires param\'etr\'ee. Par cons\'equent, lorsque $U \subset \overset{\circ}{S},$ on utilise le th\'eor\`{e}me de point fixe de Banach \`a un param\`{e}tre~\cite{krantz2012implicit} pour montrer que $u \in C^{\infty}\big(U; C_{\mathcal{V}_{Z}}^{k}([0,\tau)\times \overset{\circ}{Z}; E)\big).$
Si $U \subset S$ tel que $U\cap \partial S \neq 0,$ par la polyhomog\'en\'eit\'e de $L_{t}$ et $Q,$ on peut se ramener au cas pr\'ec\'edent en utilisant l'\'eclatement logarithmique total et en prolongeant ainsi la famille d'\'equations \`a $\Omega_{\log}$ comme dans la Proposition~\ref{prooop}. On d\'eduit que
$u \in C_{b}^{\infty}\big(\overset{\circ}{U},C^{k+2,\alpha}_{\mathcal{V}_{Z}}([0,\tau) \times \overset{\circ}{Z}; E)\big).$
On revient finalement \`a montrer que $\displaystyle \tau= \inf_{ s \in U}(\tau_{s})>0$. En effet, comme $\bar{U}\subset M$ est compact, on sait par hypoth\`{e}se que la famille $Q$ satisfait~\eqref{quad} et~\eqref{quaad} uniformes en $s\in U.$ De m\^eme, les constantes des estimations de Schauder sont uniform\'ement en $s\in U,$  par le Remarque~\ref{llilila}, ce qui nous permet de prendre $\tau$ non nul.
\end{proof}

Maintenant, nous allons prouver le th\'eor\`{e}me principal de cette section.
\begin{them} \label{tlem}
Soit $(M, S, \phi_{M}, \mathcal{V})$ un fibr\'e de structures de Lie fibr\'ees \`a l'infini associ\'e \`a un fibr\'e\\ $\phi_{M} : M \rightarrow S,$ tel que le rayon d'injectivit\'e des m\'etriques compatibles sur les fibres soit strictement positif. Soient $\mathcal{F}$ et $\mathcal{G}$ deux familles indicielles positives de $M$ et $L_{t}(s)$ une famille polyhomog\`{e}ne sur $[0,T] \times S$ d'op\'erateurs de $\Diff^{2}_{\mathcal {V},\: \mathcal{G}}(M; E)$ uniform\'ement $\mathcal{V}_{M_{s}}$\texttt{-}elliptiques sur $M_{s}:=\phi^{-1}_{M}(s), \forall s\in S.$ Alors pour $f\in  \mathcal{A}^{\mathcal{F}}_{\phg}([0,T] \times M; E)$ et $Q$ la famille d'applications donn\'ee par~\eqref{poloo} satisfaisant \`a l'Hypoth\`{e}se~\ref{hypp} avec les familles $a$, $b$ et $c$ polyhomog\`{e}nes en $s\in S$ par rapport \`a $\mathcal{G}_{\scriptscriptstyle{\vert S}},$ il existe $\varepsilon \in (0,T]$ tel que la famille d'\'equations
\begin{equation} \label{quasssi} 
(\frac{\partial }{\partial t}-L_{t}) u =Q(u)u + f, \; u_{\scriptscriptstyle{\vert t=0 }}= 0, 
\end{equation}
poss\`{e}de une unique solution $u \in \mathcal{A}^{\mathcal{K}_{\infty}}_{\phg}([0,\varepsilon) \times M; E),$ o\`u $\mathcal{K}$ est la famille indicielle positive donn\'ee par $$ \mathcal{K}=\mathcal{G}_{\infty}+\mathcal{F}_{\infty}.$$ 
\end{them} 
\begin{proof}
Comme dans la preuve du Th\'eor\`{e}me~\ref{them}, on proc\`{e}de par induction sur la profondeur $l$ de $M.$ Pour une vari\'et\'e sans bord, c'est\texttt{-}\`a\texttt{-}dire une vari\'et\'e de profondeur $0$, le r\'esultat d\'ecoule directement de la Proposition~\ref{sisisi}. On suppose donc que le r\'esultat est vrai jusqu'\`a la profondeur $l-1$ et on doit montrer qu'il est aussi valide pour $l$. Il faut donc montrer qu'il existe $\varepsilon \in (0,T]$ tel que $u \in \mathcal{A}^{{\mathcal{K}}_{\scriptscriptstyle{\vert U\times Z }}}_{\phg}([0, \varepsilon) \times U\times Z; E)$ sur chaque ouvert de trivialisation locale $U\times Z$ du fibr\'e de structure de Lie.
Sur un tel ouvert, on sait d\'ej\`a par la Proposition~\ref{sisisi} que la famille d'\'equations~\eqref{quasssi} poss\`{e}de une famille unique de solutions $u \in C_{b}^{\infty}\big(\overset{\circ}{U},C^{\infty}_{\mathcal{V}_{Z}}([0,\tau) \times \overset{\circ}{Z}; E)\big),$ pour un certain $\tau>0.$  En prenant $\tau$ assez petit, on peut aussi supposer que $L_{t}+ Q(u)$ est elliptique $\forall t\in [0, \tau).$
Soient $H_{i}$ une hypersurface bordante de $U\times Z$ et
 $ \{z_{j}\}_ {j \in \mathbb{N}_{0}} $ une suite strictement croissante de nombres r\'eels positifs ou nuls avec 
\begin{equation}
z_{0}=0 \; \textrm{et} \; (z_{j},0) \in {\mathcal{K}}_{\scriptscriptstyle{\vert U\times Z }}(H_{i})\; \; \forall j \in \mathbb{N},
\end{equation}
\begin{equation}
(z,k) \in {\mathcal{K}}_{\scriptscriptstyle{\vert U\times Z }}(H_{i}) \Longrightarrow z=z_{j} \textrm{\:pour\:un\:certain}\; j \in  \mathbb{N}_{0}.
\end{equation}
On va montrer par r\'ecurrence sur $j$ qu'il existe $\varepsilon_{U} \in(0, \tau]$ de sorte que $$\forall j \in \mathbb{N}_{0}, \;\exists u_{j} \in  \mathcal{A}^{{\mathcal{K}}_{\scriptscriptstyle{\vert U \times Z }}}_{\phg}([0,\varepsilon_{U}) \times U \times Z ; E) \textrm{\:tel\:que}\; u_{j}(0, \cdot) \equiv 0 \;  \textrm{et}$$
\begin{equation} \label{xcxc}
u-u_{j} \in \rho_{i}^{\vartheta} C_{b}^{\infty}\big( \overset{\circ}{U}; C^{\infty}_{\mathcal{V}_{Z}}([0,\varepsilon_{U})\times \overset{\circ}{Z}; E)\big) \;\; \forall \vartheta < z_{j+1},
\end{equation}
o\`u $\rho_{i}$ est une fonction de d\'efinition de $H_{i}$. On traite d'abord l'ordre 0. On a besoin de trouver $u_{0}$ tel que 
\begin{equation} \label{habibi}
u-u_{0} \in \rho_{i}^{\vartheta} C_{b}^{\infty}\big( \overset{\circ}{U}; C^{\infty}_{\mathcal{V}_{Z}}([0,\varepsilon_{i}) \times \overset{\circ}{Z}; E)\big) \;\; \forall \vartheta < z_{1}, 
\end{equation}
pour un certain $\varepsilon_{i} >0.$\\ On suppose pour l'instant qu'il existe $\varepsilon_{U}>0$ tel que $$u \in \mathcal{A}^{{\mathcal{K}}_{\scriptscriptstyle{\vert U\times Z }}}_{\phg}([0,\varepsilon_{U}) \times U \times Z; E).$$
On sait d\'ej\`a par la Proposition~\ref{munim} que $(H_{i}, S_{i}, \phi_{i}, \mathcal {V}_{H_{i}})$ est un fibr\'e de structures de Lie fibr\'ees \`a l'infini par rapport \`a un certain fibr\'e $\phi_{i} : H_{i} \rightarrow S_{i}$ de fibre typique d\'enot\'ee $Z_{i}.$ On peut donc restreindre la famille d'\'equations~\eqref{quasssi} au coefficient d'ordre $0$ sur $H_{i}$.
En particulier, si $\dim Z_{i}=0,$ alors la restriction de~\eqref{quasi} au coefficient d'ordre $0$ sur $H_{i}$ devient une \'equation diff\'erentielle ordinaire non\texttt{-}lin\'eaire d'ordre 1 sur $H_{i}$ de la forme
\begin{equation}\label{remeeem}
\frac{\partial \tilde{u}_{0}}{\partial t}- l_{t} \tilde{u}_{0}=Q_{i,0}( \tilde{u}_{0}) + f_{i,0}, \;  \tilde{u}_{0}(0,\cdot) \equiv 0,
\end{equation}
o\`u $l_{t} \in \mathcal{A}_{\phg}^{\mathcal{G}_{\scriptscriptstyle{\vert H_{i} }}}([0,\varepsilon_{U})\times H_{i}; E),$ $Q_{i,0}$ et $f_{i,0}$ sont les coefficients d'ordre $0$ de $Q$ et $f$ respectivement.
Sinon, en utilisant la Proposition~\ref{fibre} fibre par fibre, notre restriction nous donne une famille d'\'equations $\mathcal{V}_{Z_{i}}$\texttt{-}paraboliques quasi\texttt{-}lin\'eaires 
\begin{equation}\label{remrem}
\frac{\partial  \tilde{u}_{0}(s)}{\partial t}-L_{t,i,0}(s) \tilde{u}_{0}(s)= Q_{i,0}\big(s, \tilde{u}_{0}(s)\big) \tilde{u}_{0}(s) + f_{i,0}(s) , \;  \tilde{u}_{0}(s)(0,\cdot) \equiv 0.
\end{equation}
Sans compromettre la preuve, on peut consid\'erer l'\'equation~\eqref{remeeem} comme un \texttt{<<} cas particulier \texttt{>>} de~\eqref{remrem}. Comme dans la preuve du Th\'eor\`eme~\ref{them}, on peut montrer que les m\'etriques compatibles dans les fibres de $\phi_{i} : H_{i} \rightarrow S_{i}$ ont un rayon d'injectivit\'e strictement positif.
Vu que $H_{i}$ est une vari\'et\'e de profondeur $l-1,$ on applique notre hypoth\`{e}se de r\'ecurrence sur la profondeur de la vari\'et\'e pour voir que la famille d'\'equations~\eqref{remrem} poss\`{e}de une unique solution polyhomog\`{e}ne $ \tilde{u}_{0}\in  \mathcal{A}^{{\mathcal{K}}_{\scriptscriptstyle{\vert H_{i}}}}_{\phg}([0,\varepsilon_{i}) \times H_{i}; E),$ pour un certain $\varepsilon_{i} \in (0,\tau].$ On pose $u_{0}:= \Xi_{H_{i}}(\tilde{u}_{0}),$ o\`u $\Xi_{H_{i}}$ est une application d'extension comme dans la D\'efinition~\ref{extension}.
D'apr\`{e}s la formule de Taylor en $u-u_{0}=0$, on voit qu'il existe un op\'erateur $B_{t} $ d\'ependant de $u_{0}$ et $u-u_{0}$ tel que 
\begin{equation} \label{taylor}
Q(u)-Q(u_{0})=(u-u_{0})\cdot B_{t}(u_{0},u-u_{0}).
 \end{equation}
En posant $w=u- u_{0},$  on aura que 
\begin{center}
$
\begin{array}{rcl}
 Q(u)u- Q(u_{0})u_{0}&=& Q(u)(u-u_{0})+\big(Q(u)-Q(u_{0})\big)u_{0} \\
         & = & Q(u)(u-u_{0})+(u-u_{0})\cdot B_{t}(u_{0},u-u_{0})u_{0} \\
         & = &  Q(u)w+w\cdot B_{t}(u_{0},w)u_{0}.
 \end{array}
 $                                   
\end{center} 

On note par $A_{t}$ la famille d'op\'erateurs elliptiques lin\'eaires donn\'ee par $$A_{t}(u_{0},\nu) = L_{t}(\nu) + Q(u)\nu + \nu\cdot B_{t}(u_{0},w)u_{0}$$ de sorte que 
 $w$ satisfait \`a l'\'equation d'\'evolution 
$$\frac{\partial w}{\partial t}-A_{t}(u_{0},w)=h,\;  w(0,\cdot) \equiv 0, \; h=f-\frac{\partial u_{0}}{\partial t}+\big(L_{t}+Q(u_{0})\big)u_{0}.$$
Comme $u_{0}$ est polyhomog\`{e}ne, 
 $h \in  \mathcal{A}^{{\mathcal{K}}_{\scriptscriptstyle{\vert  U \times Z }}}_{\phg}([0, \varepsilon_{i}) \times U \times Z;  E).$
 Comme la restriction de $u_{0}$ est $\tilde{u}_{0}$, on d\'eduit de~\eqref{remrem} que 
$$ h \in \rho_{i}^{\vartheta} C_{b}^{\infty}\big( \overset{\circ}{U}; C^{\infty}_{\mathcal{V}_{Z}}([0,\varepsilon_{i})\times \overset{\circ}{Z}; E)\big),\; \; \forall \vartheta < z_{1},$$
car la restriction de $h$ sur $\overset{\circ}{H}_{i}$ est nulle. Par le Corollaire~\ref{corrrr}, on a alors que $$w \in\rho_{i}^{\vartheta} C_{b}^{\infty}\big( \overset{\circ}{U}; C^{\infty}_{\mathcal{V}_{Z}}([0,\varepsilon_{i})\times \overset{\circ}{Z}; E)\big), \;\; \forall \vartheta < z_{1}.$$
On suppose maintenant que le r\'esultat~\eqref{xcxc} est vrai jusqu'\`a l'ordre $j-1.$
On pose $v=u-u_{j-1}$ v\'erifiant 
\begin{equation} \label{mmmoneqo}
\frac{\partial v}{\partial t}-A_{t}(u_{j-1},v)= f^{j}, \; f^{j}=f-\frac{\partial u_{j-1}}{\partial t} +  L_{t}(u_{j-1}) + Q(u_{j-1}) u_{j-1}.
\end{equation}
Gr\^ace \`a la formule de Taylor en $v=0$, il existe un op\'erateur $\tilde{B}_{t}$ d\'ependant de $u_{j-1}$ et $v$ tel que $$B_{t}(u_{j-1},v)u_{j-1}= B_{t}(u_{j-1},0) u_{j-1}+ (\tilde{B}_{t}(u_{j-1},v)u_{j-1})  \cdot v,$$ o\`u l'op\'erateur lisse $B_{t}(u_{j-1},0)$ est le terme d'ordre $0$ dans le d\'eveloppement de Taylor de $B_{t}$ ind\'ependant de $v$.
Ainsi, la famille d'\'equations~\eqref{mmmoneqo} peut se ramener sous la forme d'une famille d'\'equations paraboliques 
\begin{equation} \label{mmmoneeqo}
\frac{\partial v}{\partial t}- \big(L_{t} + Q(u)\big)v-v \cdot B_{t}(u_{j-1},0)u_{j-1}=  v \cdot(\tilde{B}_{t}(u_{j-1},v) u_{j-1})\cdot v + f^{j}.
\end{equation}
On note $ v \cdot(\tilde{B}_{t}(u_{j-1},v) u_{j-1})\cdot v + f^{j}$ par $\psi^{j}.$
Par la d\'efinition de $f^{j},$ on d\'eduit que
\begin{equation} \label{haana}
\psi^{j} \in\mathcal{A}^{{\mathcal{K}}_{\scriptscriptstyle{\vert  U \times Z }}}_{\phg}([0,\varepsilon_{i}) \times U\times Z; E)  + \rho_{i}^{2\vartheta} C_{b}^{\infty}\big( \overset{\circ}{U}; C^{\infty}_{\mathcal{V}_{Z}}([0,\varepsilon_{i})\times \overset{\circ}{Z}; E)\big), \;\forall \vartheta < z_{j}.
\end{equation}
Afin de trouver un candidat pour le coefficient du terme d'ordre 
$\rho_{i}^{z_{j}} (\log \rho_{i})^{k}$ dans le d\'eveloppement polyhomog\`{e}ne de $v$, on peut pr\'etendre un instant que $v \in \mathcal{A}^{{\mathcal{K}}_{\scriptscriptstyle{\vert  U \times Z }}}_{\phg}([0,\varepsilon_{i}) \times U\times Z; E) $. 
En utilisant~\eqref{taylor}, la famille d'\'equations~\eqref{mmmoneeqo} peut se r\'e\'ecrire 
\begin{equation} \label{mimmoneeqo}
\frac{\partial v}{\partial t}- \big(L_{t} + Q(u_{0})\big)v - v \cdot B_{t}(u_{j-1},0)u_{j-1}  - (u-u_{0})\cdot B_{t}(u_{0},u-u_{0}) v = \psi^{j} .
\end{equation}
Comme $$ v\in \mathcal{A}^{{\mathcal{K}}_{\scriptscriptstyle{\vert  U \times Z }}}_{\phg}([0,\varepsilon_{i}) \times U\times Z; E)  \cap \rho_{i}^{\vartheta} C_{b}^{\infty}\big( \overset{\circ}{U}; C^{\infty}_{\mathcal{V}_{Z}}([0,\varepsilon_{i})\times \overset{\circ}{Z}; E)\big), \;\forall \vartheta < z_{j}$$ 
et gr\^ace \`a la Proposition~\ref{fibre} et~\eqref{habibi}, on voit que la restriction de l'\'equation~\eqref{mimmoneeqo} au coefficient d'ordre $\rho_{i}^{z_{j}} (\log \rho_{i})^{k}$ sur $\overset{\circ}{H}_{i}$, o\`u $k\in \mathbb{N}_{0}$ est le plus grand entier tel que $(z_{j}, k) \in \mathcal{K}(H_{i})$, nous donne une famille d'\'equations $\mathcal{V}_{Z_{i}}$\texttt{-}paraboliques lin\'eaires localement de la forme
\begin{equation}\label{aheana}
\frac{\partial v _{j}(s)}{\partial t}- P_{t}(s) v_{j}(s) = \psi^{j}_{z_{j},k}(s), \; {v_{j}(s)}_{\scriptscriptstyle{\vert t=0}}  \equiv 0,
\end{equation}
o\`u $v_{j}(s)$ et $\psi^{j}_{z_{j},k}(s) \in \mathcal{A}^{{\mathcal{K}}_{\scriptscriptstyle{\vert  Z_{i} }}}_{\phg}([0,\varepsilon_{i}) \times Z_{i}; E)$ sont les coefficients du terme d'ordre $\rho_{i}^{z_{j}} (\log \rho_{i})^{k}$ de $v$ et $\psi^{j}$ respectivement restreints sur $\overset{\circ}{H}_{i}$ et la famille d'op\'erateurs $P_{t}(s)\in \Diff^{2}_{\mathcal {V}_{ Z_{i}},\: {\mathcal{K}}_{\scriptscriptstyle{\vert  Z_{i} }}}(Z_{i}; E)$ est la restriction \`a l'ordre $0$ sur $H_{i}$ de l'op\'erateur $$v \mapsto \big(L_{t} + Q(u_{0})\big)v - v \cdot B_{t}(u_{j-1},0)u_{j-1}.$$
En vertu du Th\'eor\`{e}me~\ref{them}, la famille d'\'equations~\eqref{aheana} poss\`{e}de une unique solution polyhomog\`{e}ne $v_{j} \in \mathcal{A}^{{\mathcal{K}}_{\scriptscriptstyle{\vert  H_{i} }}}_{\phg}([0,\varepsilon_{i}) \times H_{i}; E)$. 
 On pose ensuite $w = v - w_{j}$ avec $w_{j} = \rho_{i}^{z_{j}} (\log \rho_{i})^{k} \Xi_{H_{i}} (v_{j})$, de sorte que $w$ satisfait \`a l'\'equation d'\'evolution 
\begin{multline*}
\frac{\partial w}{\partial t}-\big(L_{t} + Q(u_{0})\big)w -w \cdot B_{t}(u_{j-1},0)u_{j-1} = (u-u_{0})\cdot B_{t}(u_{0},u-u_{0}) w + \varphi^{j}, w(0,\cdot) \equiv 0,\\  \;\
 \varphi^{j}=\psi^{j}-\frac{\partial  w_{j} }{\partial t}+\big(L_{t} + Q(u_{0})\big)w_{j} +(u-u_{0})\cdot B_{t}(u_{0},u-u_{0}) w_{j}+w_{j} \cdot B_{t}(u_{j-1},0)u_{j-1}.
\end{multline*}
Par d\'efinition de la suite $\{z_{j}\}$, on remarque que $z_{j}\geq z_{1}\geq z_{j+1}-z_{j}, \;\forall j \in \mathbb{N}.$ Gr\^ace \`a~\eqref{haana} et notre choix de $w_{j},$ on aura que   
\begin{equation} \label{hzaana}
 \varphi^{j} \in\mathcal{A}^{{\mathcal{K}}_{\scriptscriptstyle{\vert  U \times Z }}}_{\phg}([0,\varepsilon_{i}) \times U\times Z; E)  + \rho_{i}^{\vartheta} C_{b}^{\infty}\big( \overset{\circ}{U}; C^{\infty}_{\mathcal{V}_{Z}}([0,\varepsilon_{i})\times \overset{\circ}{Z}; E)\big), \;\forall \vartheta < z_{j+1},
\end{equation}
et que 
$\varphi^{j} \in  \rho_{i}^{z_{j}} (\log \rho_{i})^{k-1} C_{b}^{\infty}\big( \overset{\circ}{U}; C^{\infty}_{\mathcal{V}_{Z}}([0,\varepsilon_{i}) \times \overset{\circ}{Z}; E)\big).$
Donc, par le Corollaire~\ref{corrrr}, $$w \in \rho_{i}^{z_{j}} (\log \rho_{i})^{k-1} C_{b}^{\infty}\big( \overset{\circ}{U}; C^{\infty}_{\mathcal{V}_{Z}}([0,\varepsilon_{i})\times \overset{\circ}{Z}; E)\big).$$
En r\'ep\'etant cet argument $k$ fois, on obtient alors $$r_{k}, r_{k-1}, ...,r_{0} \in  \mathcal{A}^{{\mathcal{K}}_{\scriptscriptstyle{\vert  H_{i} }}}_{\phg}([0,\varepsilon_{i}) \times H_{i}; E)\;\; \textrm{avec}\;r_{k}= v_{j}$$  
de sorte que $$\hat{v}=v- \sum \limits_{p=0}^{k} \Xi_{H_{i}}(r_{p})\rho_{i}^{z_{j}} (\log \rho_{i})^{p}, \;\;\; r_{p}(0,\cdot) \equiv 0,$$ a pour \'equation d'\'evolution sous la forme $$\frac{\partial \hat{v}}{\partial t}-\big(L_{t} + Q(u_{0})\big)\hat{v} - \hat{v} \cdot B_{t}(u_{j-1},0)u_{j-1} =  (u-u_{0})\cdot B_{t}(u_{0},u-u_{0}) \hat{v} + \hat{h}^{j}, \;  \hat{v} (0,\cdot) \equiv 0, \; \textrm{avec}\; $$
\begin{multline*}
\hat{h}^{j} \in \rho_{i}^{z_{j}} C_{b}^{\infty}\big( \overset{\circ}{U}; C^{\infty}_{\mathcal{V}_{Z}}([0,\varepsilon_{i}) \times \overset{\circ}{Z}; E)\big)\cap \\\bigg( \mathcal{A}^{{\mathcal{K}}_{\scriptscriptstyle{\vert  U \times Z }} }_{\phg}([0, \varepsilon_{i} ) \times U \times Z ; E)+ \rho_{i}^{\vartheta} C_{b}^{\infty}\big( \overset{\circ}{U}; C^{\infty}_{\mathcal{V}_{Z}}([0,\varepsilon_{i})\times \overset{\circ}{Z}; E)\big)\bigg), \;\forall \vartheta < z_{j+1},
\end{multline*}
et $ \rho_{i}^{-z_{j}} \hat{h}^{j}_{\scriptscriptstyle{\vert \overset{\circ}{H}_{i}}}=0.$ 
On a donc $\hat{h}^{j} \in \rho_{i}^{\vartheta} C_{b}^{\infty}\big( \overset{\circ}{U}; C^{\infty}_{\mathcal{V}_{Z}}([0,\varepsilon_{i})\times \overset{\circ}{Z}; E)\big), \;\forall \vartheta < z_{j+1}.$ Par~\eqref{taylor} et le Corollaire~\ref{corrrr}, on en d\'eduit que $$\hat{v} \in \rho_{i}^{\vartheta} C_{b}^{\infty}\big( \overset{\circ}{U}; C^{\infty}_{\mathcal{V}_{Z}}([0,\varepsilon_{i})\times \overset{\circ}{Z}; E)\big), \;\forall \vartheta < z_{j+1}.$$
Le pas d'induction est donc compl\'et\'e en prenant
$\displaystyle u_{j}=u_{j-1} + \sum \limits_{p=0}^{k} \Xi_{H_{i}}(r_{p})\rho_{i}^{z_{j}} (\log \rho_{i})^{p}.$
Finalement, on prend $\displaystyle \varepsilon_{U}= \min_{i \in\{1,...,l\}}\varepsilon_{i}$. Comme $S$ est compact, on peut extraire un recouvrement fini d'ouverts de trivialisation $\{ U_{j}, j\in J\}$ de $S$ de sorte qu'on peut prendre $\displaystyle \varepsilon= \min_{j \in J}\varepsilon_{U_{j}}.$
\end{proof}
\begin{cor}\label{finin}
Soit $(M,\mathcal{V}_{SF})$ une structure de Lie fibr\'ee \`a l'infini telle que le rayon d'injectivit\'e des m\'etriques compatibles est strictement positif. Soit $\mathcal{G}$ une famille indicielle positive de $M$. Si dans l'\'equation~\eqref{quassi}, $\{L_{t} : t\in [0,T]\} $ est une famille d'op\'erateurs uniform\'ement $\mathcal{V}_{SF}$-elliptiques de $\Diff^{2}_{\mathcal {V}_{SF},\: \mathcal{G}}(M; E)$ et $f \in \mathcal{A}^{\mathcal{F}}_{\phg}([0,T] \times M; E),$ alors 
il existe $\varepsilon \in (0,T]$ tel que la solution de cette \'equation est dans $\mathcal{A}^{\mathcal{K}}_{\phg}([0,\varepsilon) \times M; E)$ avec $ \mathcal{K}=\mathcal{G}_{\infty}+\mathcal{F}_{\infty}.$
\end{cor} 
\begin{proof}
C'est une cons\'equence directe du Th\'eor\`{e}me \ref{tlem} en prenant la base $S$ un point de fibr\'e $\phi_{M}.$
\end{proof}
\section{Polyhomog\'en\'eit\'e des solutions du flot de Ricci\texttt{-}DeTurck}
Nous \'etablissons que la polyhomog\'en\'eit\'e des m\'etriques compatibles avec une structure de Lie fibr\'ee \`a l'infini est pr\'eserv\'ee le long du flot de Ricci\texttt{-}DeTurck au moins pour un court laps de temps. En particulier, lorsque la m\'etrique initiale est asymptotiquement Einstein, la polyhomog\'en\'eit\'e sera pr\'eserv\'ee par le flot de Ricci\texttt{-}DeTurck tant que le flot existe.
Nous introduisons bri\`{e}vement certaines d\'efinitions et r\'esultats fondamentaux li\'ees au concept du flot de Ricci qui serviront par la suite. Pour plus de d\'etails, nous renvoyons le lecteur \`a~\cite{lee1997riemannian},~\cite{lee2013riemannian} et~\cite{chow2004ricci}. 
\begin{defi}\label{fatma}
Soient $(M, g)$ une vari\'et\'e riemannienne et $\nabla$ la connexion de Levi\texttt{-}Civita de la m\'etrique. L'application $R : \mathfrak{X}(M) \times \mathfrak{X}(M) \times \mathfrak{X}(M)  
\rightarrow \mathfrak{X}(M)$ d\'efinie par 
\begin{equation} \label{fatm}
R(X,Y)Z=\nabla_{X} \nabla_{Y}Z-\nabla_{Y} \nabla_{X} Z - \nabla_{\nabla_{X}Y- \nabla_{Y}X}Z
\end{equation}
est appel\'ee l'endomorphisme de Riemann ou le tenseur de courbure.
\end{defi}
En termes de coordonn\'ees locales, on r\'e\'ecrit
$$R(\partial_{i},\partial _{j})\partial _{k}= R^{l}_{ijk} \partial_{l} \;\textrm{o\`u}$$ 
\begin{equation} \label{cour}
 R^{l}_{ijk}= \partial_{i}\Gamma_{ jk}^{l}- \partial_{j}\Gamma_{ ik}^{l}+ \Gamma_{ jk}^{s}\Gamma_{ is}^{l}-\Gamma_{ ik}^{s}\Gamma_{ js}^{l},
\end{equation} 
et $\Gamma_{ ij}^{l}$ est le symbole de Christoffel. 
L'endomorphisme de Riemann $R$ est vu comme \'etant un homomorphisme du fibr\'e vectoriel $\otimes ^{2}
TM $ vers le fibr\'e vectoriel $\End(TM) \simeq TM \otimes T^{*}M$, et ainsi est un champs de tenseurs de type $(3,1)$ d\'efini localement par
$$ R=R^{l}_{ijk} dx_{i} \otimes dx_{j} \otimes dx_{k} \otimes \partial^{l}. $$
Pour $m \in M$ donn\'e, il est souvent utile de consid\'erer une carte normale centr\'ee en $m$. Dans cette carte normale, on sait que le symbole de Christoffel s'annule en $m, {\Gamma_{ ij}^{k}}_{\scriptscriptstyle{\vert m}} = 0$ donc ${\nabla_{i} \partial_{j}}_{\scriptscriptstyle{\vert m}} =0$. Par suite,
 $${(\nabla_{s} F)_{\,j_{1}...j_{k}}^ {\,i_{1}...i_{l}}}_{\scriptscriptstyle{\vert m}} = {(\partial_{s} F)_{\,j_{1}...j_{k}}^ {\,i_{1}...i_{l}}}_{\scriptscriptstyle{\vert m}}$$
et $${\nabla_{i,j}^{2}}_{\scriptscriptstyle{\vert m}}:=(
 \nabla_{\partial_{i}} \nabla_{\partial_{j}}- \nabla_{\nabla_{\partial_{i}}\partial_{j}})_{\scriptscriptstyle{\vert m}} = {\nabla_{i}\nabla_{j}}_{\scriptscriptstyle{\vert m}}.$$
Ces faits simplifient consid\'erablement les calculs des tenseurs en coordonn\'ees normales.
\begin{prop} (Lemme 3.2~\cite{chow2004ricci})
Pour une famille lisse $g(t)$ des m\'etriques riemanniennes, la d\'eriv\'ee de symbole de Christoffel de la connexion de Levi\texttt{-}Civita par rapport de temps est donn\'ee par 
\begin{equation}\label{civiv}
\frac{\partial}{\partial t}\Gamma_{ ij}^{k}=\frac{1}{2} g^{kl} \big((\nabla_{i} h)_ {jl}+(\nabla_{j} h)_ {il}-(\nabla_{l} h)_ {ij}\big),
\end{equation}
o\`u $h$ est la d\'eriv\'ee de $g$ par rapport au temps.
\end{prop}
\begin{prop}(Lemme 3.3~\cite{chow2004ricci}) \label{sasa}
Pour une famille lisse $g(t)$ des m\'etriques riemanniennes, la d\'eriv\'ee de l'endomorphisme de Riemann par rapport au temps est donn\'ee par $$\frac{\partial}{\partial t}R^{l}_{ijk}= \frac{1}{2} g^{ls} \big((\nabla^{2}_{i,j} h)_ {ks}+(\nabla^{2}_{i,k} h)_{js}-(\nabla^{2}_{i,s} h)_ {jk} -(\nabla^{2}_{j,i} h)_ {ks}-(\nabla^{2}_{j,k} h)_ {is}+(\nabla^{2}_{j,s} h)_ {ik}\big),$$ o\`u $h$ est la d\'eriv\'ee de $g$ par rapport au temps.
 \end{prop}
\begin{defi}
Le tenseur de Riemann ou tenseur de courbure $Rm$ est le champ de tenseurs de type $(4,0)$ d\'efini par $$Rm(X,Y,Z,W)=g\big(R(X,Y)W,Z\big) \:\; \forall X,Y,Z,W\in\mathfrak{X}(M).$$ 
\end{defi}
\begin{defi}
La courbure de Ricci $\Ric$ est le champ de tenseurs de type $(2,0)$ d\'efini par  $\Ric=\tr Rm$.
En termes de coordonn\'ees locales, on a $\Ric_{ij}=R^{k}_{kij}$ (o\`u on somme sur l'indice r\'ep\'et\'e) et gr\^ace \`a la sym\'etrie du tenseur $Rm$, $\Ric$ est sym\'etrique.
\end{defi} 
\begin{prop} (Lemme 3.5~\cite{chow2004ricci})
La d\'eriv\'ee de la courbure de Ricci par rapport au temps est donn\'ee localement par 
\begin{equation} \label{ricci}
\frac{\partial}{\partial t}\Ric_{ij}= \frac{1}{2} g^{ks} \big((\nabla^{2}_{k,i} h)_ {js}+(\nabla^{2}_{k,j} h)_ {is}-(\nabla^{2}_{k,s} h)_ {ij}-(\nabla^{2}_{i,j} h)_ {ks}\big).
\end{equation}
\end{prop}
\begin{proof}
Il suffit de prendre $i=l$ dans la Proposition~\ref{sasa}. 
On a alors 
$$\frac{\partial}{\partial t}\Ric_{jk}=\frac{\partial}{\partial t}R^{i}_{ijk}= \frac{1}{2} g^{is} \big((\nabla^{2}_{i,j} h)_ {ks}+(\nabla^{2}_{i,k} h)_{js}-(\nabla^{2}_{i,s} h)_ {jk} $$$$-(\nabla^{2}_{j,k} h)_ {is}\big)+ \frac{1}{2} g^{is} \big((\nabla^{2}_{j,s} h)_ {ik}-(\nabla^{2}_{j,i} h)_ {ks}\big).$$
Le r\'esultat~\eqref{ricci} provient alors du fait que 
$$ g^{is} (\nabla^{2}_{j,s} h)_ {ik}- g^{is}(\nabla^{2}_{j,i} h)_ {ks}=\ g^{is} (\nabla^{2}_{j,s} h)_ {ik}- \ g^{si}(\nabla^{2}_{j,s} h)_ {ik}=0. $$
\end{proof}
Soit $g(t)$ une famille \`a un param\`{e}tre de m\'etriques riemanniennes sur une vari\'et\'e $\overset{\circ}{M}.$
\begin{defi}
Le flot de Ricci est l'\'equation non\texttt{-}lin\'eaire d'\'evolution d\'efinie par  $\frac{\partial}{\partial t} g(t) = -2 Ric\big(g(t)\big),$ $g(0)=g_{0}.$ 
\end{defi}
Supposons maintenant que $g_{0}$ est une m\'etrique compatible avec une structure de Lie fibr\'ee \`a l'infini $(M,\mathcal{V})$ telle que le rayon d'injectivit\'e soit strictement positif, o\`u $M$ est une vari\'et\'e \`a coins compacte telle que $\overset{\circ}{M}$ soit identifi\'e avec l'int\'erieur de $M.$
Nous avons vu dans la premi\`{e}re section que $g_{0}$ est une m\'etrique compl\`{e}te et que le tenseur de courbure et les d\'eriv\'ees contravariantes sont born\'ees. De ce fait, nous pouvons appliquer le Th\'eor\`{e}me de Shi~\cite{shi1989deforming} pour d\'emontrer l'existence local du flot de Ricci sur $[0, T)$, o\`u $T<\infty$, not\'ee $\bar{g} \in C_{\mathcal{V}}^{\infty}([0, T) \times\overset{\circ}{M}; \;T^{(2,0)}\overset{\circ}{M}),$ telle que les m\'etriques $\bar{g}(t)$ sont des $\mathcal{V}$\texttt{-}m\'etriques, c'est\texttt{-}\`a\texttt{-}dire que les $\bar{g}(t)$ sont bi\texttt{-}Lipschitz \'equivalentes, le long du flot, \`a $g_{0}$, ou plus pr\'ecis\'ement que pour $T$ assez petit,
$$\exists C>0,\: \forall t\in[0,T), \;\; e^{-C}g_{0}\leq \bar{g}(t)\leq e^{C}g_{0}.$$ Pour l'unicit\'e de la solution, nous nous r\'ef\'erons aux travaux de Bing\texttt{-}Long Chen et Xi\texttt{-}Ping Zhu~\cite{chen2006uniqueness}.

Nous allons maintenant d\'efinir le flot de Ricci\texttt{-}DeTurck~\cite{chow2004ricci}.
Un op\'erateur diff\'erentiel non\texttt{-}lin\'eaire $P$ est dit elliptique si sa lin\'earisation $DP$ est elliptique, c'est\texttt{-}\`a\texttt{-}dire que son symbole principal $\sigma_{2}(DP)(\xi)$ sous la direction de $\xi \in T^{*}\overset{\circ}{M}$ est un isomorphisme si $\xi \neq 0$. En revenant au r\'esultat~\eqref{ricci} et en regardant $\Ric(g)$ comme un op\'erateur diff\'erentiel agissant sur $g,$ nous voyons que sa lin\'earisation est
$$D\Ric(g)(h)_{ij} =\frac{1}{2} g^{sk} \big((\nabla^{2}_{k,i}h)_ {js}+(\nabla^{2}_{k,j}h)_ {is}-(\nabla^{2}_{k,s} h)_ {ij}-(\nabla^{2}_{i,j} h)_ {ks}\big).$$
Comme $\sigma_{2}\big(D\Ric(g)\big)(\xi)_{ij}$ n'est pas elliptique, voir page 72 dans~\cite{chow2004ricci}, cela implique que le flot de Ricci n'est pas parabolique. Mais par un truc remontant \`a DeTurck, voir la preuve du Th\'eor\`{e}me 3.13~\cite{chow2004ricci}, nous pouvons nous ramener \`a une \'equation parabolique quasi\texttt{-}lin\'eaire en modifiant la solution par une famille de diff\'eomorphismes. En effet, nous pouvons r\'e\'ecrire la lin\'earisation du flot de Ricci sous la forme (voir (3.28)~\cite{chow2004ricci}) 
$$D\Ric(g)(h)_{ij} = - \frac{1}{2} (\Delta h_{ij} - \nabla_{i}V_{j} - \nabla_{j}V_{i} + S_{ij}),$$
o\`u $V_{j}:=g^{sk}(\nabla_{k}h_{sj}- \frac{1}{2}\nabla_{j}h_{sk})$ et 
$S_{ij}:= g^{sk}(2 R^{r}_{kij} h_{rs} -  Ric_{is} h_{jk} - Ric_{js} h_{ik}).$\\
Soit $\tilde{g}$ une m\'etrique de r\'ef\'erence et $\tilde{\Gamma}^{j}_{sk}$ son symbole de Christoffel et consid\'erons le champ de vecteurs 
$$W^{j}:=W^{j}(g,\tilde{g})= g^{sk}\big(\Gamma^{j}_{sk} -\tilde{\Gamma }^{j}_{sk}\big),$$ o\`u $\Gamma^{j}_{sk}$ est le symbole de Christoffel de la m\'etrique $g$.
Pour $W_{i}=g_{ij}W^{j}$ le champ de covecteurs associ\'e, on pose
$$A(g)_{ij}=\nabla_{i}W_{j} +\nabla_{j}W_{i}=\nabla_{i}\big(\frac{1}{2} g_{jp} g^{sk}(\Gamma^{p}_{sk}-\tilde{\Gamma}^{p}_{sk})\big)- \nabla_{j}\big(\frac{1}{2}g_{ip} g^{sk}(\Gamma^{p}_{sk}-\tilde{\Gamma}^{p}_{sk})\big).$$
En utilisant~\eqref{civiv}, nous pouvons voir que 
$$D A(g)(h)_{ij} = \nabla_{i}V_{j} +\nabla_{j}V_{i} + T_{ij},$$
o\`u $T_{ij}$ est un terme lin\'eaire d'ordre 1 en $h$.
Ainsi,
\begin{align*}
D(-2\Ric+A)(g)(h)_{ij} &=-2D\Ric(g)(h)_{ij} + DA(g)(h)_{ij}\\
                                   &= \Delta h_{ij}  + S_{ij} + T_{ij} \\
                                   &= \Delta h_{ij}  + 2g^{ks} R^{r}_{kij} h_{rs} - g^{ks} Ric_{is} h_{jk} -g^{ks} Ric_{js} h_{ik} + T_{ij}\\
                                   &= \mathcal{L}_{g} (h)_{ij}+ T_{ij},
\end{align*}
o\`u $\mathcal{L}_{g}(h)_{ij}:=\Delta h_{ij}  + 2g^{ks} R^{r}_{kij} h_{rs} - g^{ks} Ric_{is} h_{jk} -g^{ks} Ric_{js} h_{ik}$ est le laplacien de Lichnerowicz correspondant \`a $g$ qui est un op\'erateur elliptique. Par cons\'equent, $-2\Ric+A$ est un op\'erateur elliptique.
\begin{defi}
L'\'equation parabolique donn\'ee localement par
\begin{equation} \label{deturckk}
\frac{\partial}{\partial t} g(t)_{ij} = -2 \Ric (g(t))_{ij}+\nabla_{i}W_{j} (t)+\nabla_{j}W_{i}(t), \;  \; g(0)=g_{0},
\end{equation}
o\`u $\displaystyle W_{i}(t):=g_{ij}W^{j}(t)$ avec $W^{j}(t):=W^{j}(g(t), g(0)),$ est appel\'ee le flot de Ricci\texttt{-}DeTurck, 
\end{defi}
\begin{prop}(Lemme 2.1~\cite{shi1989deforming}) \label{propi}
Le flot de Ricci\texttt{-}DeTurck est donn\'e localement par
\begin{multline*} 
\frac{\partial}{\partial t} g_{ij} =g^{ab} \tilde{\nabla}_{a}\tilde{\nabla}_{b}g_{ij} - g^{ab} \tilde{g}^{pq} (g_{ip} \tilde{R}m_{jaqb}+ g_{jp}       \tilde{R}m_{iaqb}) + \frac{g^{ab} g^{pq}}{2} (\tilde{\nabla}_{i}g_{pa}\tilde{\nabla}_{j}g_{qb} \\ + 2\tilde{\nabla}_{a}g_{jp}\tilde{\nabla}_{q}g_{ib} - 2 \tilde{\nabla}_{a}g_{jp}\tilde{\nabla}_{b}g_{iq}-2\tilde{\nabla}_{j}g_{pa}\tilde{\nabla}_{b}g_{iq}-2\tilde{\nabla}_{i}g_{pa}\tilde{\nabla}_{b}g_{jq}),
\end{multline*}
 $g(0)=g_{0} =\tilde {g}.$ On utilise $\sim$ pour pr\'esenter les connexions de Levi\texttt{-}Civita et les tenseurs de courbure associ\'es \`a la m\'etrique initiale $g_{0}$.
\end{prop}
Dans ce qui suit, posons $u=g-\tilde{g}.$
\begin{prop} \label{sif}
L'\'equation du flot de Ricci\texttt{-}DeTurck dans la Proposition~\ref{propi} peut \^etre r\'e\'ecrite sous la forme  
\begin{equation} \label{reme}
\frac{\partial u}{\partial t} -  \mathcal{L}_{\tilde{g}}u= \mathcal{Q}(u)u -  2 \Ric(\tilde{g}), \; u_{\scriptscriptstyle{\vert t=0 }}=0,
\end{equation}
avec
$$
Q(u)v=a(u)\cdot \nabla^{2}v + b(u, \nabla u)\cdot \nabla v + c(u)\cdot v,
$$
o\`u $ \mathcal{L}_{\tilde{g}}$ est le laplacien de Lichnerowicz correspondant \`a $\tilde{g}$ et $a$, $b$ et $c$ sont lisses dans leur d\'ependance en $u$ et $\nabla u,$ tels que $a(0)_{\scriptscriptstyle{\vert t=0 }}=0$, $b(0,0)_{\scriptscriptstyle{\vert t=0 }}=0,$ $c(0)_{\scriptscriptstyle{\vert t=0 }}=0.$
De plus, $Q$ satisfait~\eqref{quad} et~\eqref{quaad}. 
\end{prop}
\begin{proof}
Comme la m\'etrique $\tilde{g}+u$ est inversible, elle satisfait \`a l'identit\'e
\begin{equation}\label{cadre}
(\tilde{g}+u)^{ab}=\tilde{g}^{ab}-\tilde{g}^{al}\tilde{g}^{bm}u_{ml}+(\tilde{g}+u)^{bl}\tilde{g}^{am}\tilde{g}^{pq}u_{lp}u_{mq}.
\end{equation}
Puisque $ \tilde{\nabla}_{b}\tilde{g}_{ij}=0$, l'\'equation de la Proposition~\ref{propi} est \'equivalente \`a 
$$
\frac{\partial}{\partial t} u_{ij} = \big((\tilde{g}+u)^{ab}-\tilde{g}^{ab}\big) \tilde{\nabla}_{a}\tilde{\nabla}_{b}u_{ij}  + (Pu)_{ij}
                                       + \big((\tilde{g}+u)^{-1} \star (\tilde{g}+u)^{-1} \star \tilde{\nabla}u \star \tilde{\nabla}u\big)_{ij}, \;\; u_{\scriptscriptstyle{\vert t=0 }}= 0,
$$
o\`u $\star$ d\'enote les contractions lin\'eaires dont les formules pr\'ecises ne joueront aucun r\^ole dans la suite et 
$$
(Pu)_{ij} = \tilde{g}^{ab} \tilde{\nabla}_{a}\tilde{\nabla}_{b}u_{ij}  -(\tilde{g}+u)^{ab} (\tilde{g}+u)_{ip}\tilde{g}^{pq}  \tilde{R}m_{jaqb} - (\tilde{g}+u)^{ab} (\tilde{g}+u)_{jp} \tilde{g}^{pq} \tilde{R}m_{iaqb}. 
$$
\`A l'aide de~\eqref{cadre}, on a que
\begin{align*} 
(Pu)_{ij}&= \tilde{g}^{ab} \tilde{\nabla}_{a}\tilde{\nabla}_{b}u_{ij} -\tilde{g}^{pq}\Ric(\tilde{g})_{jq}u_{ip}-\tilde{g}^{pq}\Ric(\tilde{g})_{iq}u_{jp}+2\tilde{g}^{al}R^{m}_{aij}u_{ml}+(c(u)\cdot u)_{ij}-2\Ric(\tilde{g})_{ij}\\
                   &=\mathcal{L}_{\tilde{g}}(u)_{ij}+(c(u)\cdot u)_{ij}-2\Ric(\tilde{g})_{ij},
\end{align*}
o\`u  $c$ est comme dans l'\'enonc\'e de la proposition.
Le r\'esultat est alors obtenu en prenant
\begin{align*} 
\big(a(u)\cdot \nabla^{2} v\big)_{ij}&= \big((\tilde{g}+u)^{ab}-\tilde{g}^{ab}\big) \tilde{\nabla}_{a}\tilde{\nabla}_{b}v_{ij}  \\
  &=\big(\tilde{g}^{ab}-\tilde{g}^{al}\tilde{g}^{bm}u_{ml}+(\tilde{g}+u)^{bl}\tilde{g}^{am}\tilde{g}^{pq}u_{lp}u_{mq}-\tilde{g}^{ab}\big) \tilde{\nabla}_{a}\tilde{\nabla}_{b}v_{ij}\\
             &=\big(\tilde{g}^{-1} \star \tilde{g}^{-1} \star u \star \tilde{\nabla}^{2}v +(\tilde{g}+u)^{-1}\star \tilde{g}^{-1} \star \tilde{g}^{-1} \star u \star u \star \tilde{\nabla}^{2}v\big)_{ij}\\
\big(b(u, \nabla  u) \cdot \nabla v\big)_{ij}&=\big((\tilde{g}+u)^{-1} \star (\tilde{g}+u)^{-1} \star \tilde{\nabla}u \star \tilde{\nabla}v\big)_{ij} \\
\big(c(u)\cdot v\big)_{ij}&=(\tilde{g}^{-1} \star \tilde{g}^{-1} \star u \star v\star \tilde{R})_{ij}.        
\end{align*}
\end{proof}
Pour des m\'etriques associ\'ees \`a une structure de Lie fibr\'ee \`a l'infini sur une vari\'et\'e non compacte, on montre que la polyhomog\'en\'eit\'e le long du flot de Ricci\texttt{-}DeTurck est pr\'eserv\'ee, \`a tout le moins pour un cours laps de temps.
\begin{them}\label{theore}
Soit $(\overset{\circ}{M}, g_{0})$ une vari\'et\'e riemannienne de dimension $n$ avec une structure de Lie fibr\'ee \`a l'infini $(M,\mathcal{V}_{SF})$ telle que le rayon d'injectivit\'e  est strictement positif. Si $g_{0}$ est polyhomog\`{e}ne par rapport une famille indicielle $\mathcal{G}$ positive de $M$ et $g(t)$ pour $t \in [0,T)$ est la solution du flot de Ricci\texttt{-}DeTurck, alors il existe $ \varepsilon \in (0,T)$ tel que 
$$g \in  \mathcal{A}^{\mathcal{G}_{\infty}}_{\phg}([0, \varepsilon) \times M; {}^{\mathcal {V}_{SF}}TM^{(2,0)}).$$
\end{them}
\begin{proof}
Par la Proposition~\ref{sif}, on sait que le flot de Ricci\texttt{-}DeTurck~\eqref{deturckk} est \'equivalent \`a l'\'equation~\eqref{reme}, qui a la m\^eme forme que l'\'equation~\eqref{quasssi} pour un certain $ \varepsilon \in (0,T)$. Le r\'esultat est donc une cons\'equence directe du Corollaire~\ref{finin}.
\end{proof}
Remarquons que par la discussion \`a la fin de la premi\`{e}re section, plus pr\'ecis\'ement le Corollaire 3.20~\cite{ammann2004geometry}, le rayon d'injectivit\'e des solutions du flot de Ricci $g(t)$ est strictement positif puisqu'elles sont bi\texttt{-}Lipschitz \'equivalentes \`a la m\'etrique initiale $g_{0}.$

Nous montrons maintenant que la polyhomog\'en\'eit\'e \`a l'infini des m\'etriques compatibles avec une structure de Lie fibr\'ee \`a l'infini est pr\'eserv\'ee globalement par le flot de Ricci\texttt{-}DeTurck si la m\'etrique initiale est asymptotiquement Einstein.
\begin{defi} \label{sak}
La m\'etrique $g_{0}$ est dite asymptotiquement Einstein s'il existe $\lambda>0$ et $\delta>0$ tels que 
$$ \Ric(g_{0})+\lambda g_{0} \in \rho^{\delta} C_{\mathcal{V}}^{\infty}(\overset{\circ}{M};T^{(2,0)}\overset{\circ}{M}), \;\;\;\rho= \prod_{ H \in \mathcal{M}_{1}(M)} \rho_{H}.$$
\end{defi}
Cela sugg\`{e}re de normaliser le flot de Ricci. 
\begin{defi}
Le flot de Ricci normalis\'e est donn\'e par  
\begin{equation} \label{normal}
\frac{\partial}{\partial t} g(t) = -2\Ric\big(g(t)\big)-2\lambda g(t),\; \; g(0)=g_{0}.
\end{equation}
\end{defi}
\begin{defi} 
Le flot de Ricci\texttt{-}DeTurck normalis\'e est d\'efini localement par
\begin{equation} \label{deturck}
\frac{\partial}{\partial t} g(t)_{ij} = -2\Ric\big(g(t)\big)_{ij}-2\lambda g(t)_{ij}+\nabla_{i}W_{j} (t)+\nabla_{j}W_{i}(t), \;  \; g(0)=g_{0}.
\end{equation}
\end{defi}
\begin{prop}\label{propii}
Soit $(\overset{\circ}{M}, g_{0})$ une vari\'et\'e riemannienne avec une structure de Lie \`a l'infini $(M,\mathcal{V})$ telle que le rayon d'injectivit\'e de $g_{0}$ est strictement positif.
Si la m\'etrique $g_{0}$ est asymptotiquement Einstein et $g$ est la solution du flot Ricci\texttt{-}DeTurck normalis\'e, alors 
\begin{equation} \label{finich}
g-g_{0} \in \rho^{\delta} C_{\mathcal{V}}^{\infty}([0, T) \times\overset{\circ}{M};T^{(2,0)}\overset{\circ}{M}),
\end{equation}
o\`u $\delta$ est la constante d'Einstein de $g_{0}$ apparue dans la D\'efinition~\ref{sak}.
\end{prop}
\begin{proof}
On consid\`{e}re $u=g - g_{0}$ au lieu de $g.$ \`A partir de la Proposition~\ref{propi} et par un calcul standard, l'\'equation d'\'evolution de $u$ peut se ramener \`a  l'\'equation lin\'eaire parabolique suivante 
\begin{equation} 
\frac{\partial}{\partial t} u = L_{t}(u) + f, \; u_{\scriptscriptstyle{\vert t=0 }}=0, \;f=-2\Ric(\tilde{g})-2\lambda \tilde{g}, \label{hamnaa}
\end{equation}
o\`u $L_{t} \in \Diff^{2}_{\mathcal{V}}\big(\overset{\circ}{M};\; \End( T^{(2,0)}\overset{\circ}{M})\big)$ est un op\'erateur uniform\'ement $\mathcal{V}$\texttt{-}elliptique donn\'e localement par
$$\big(L_{t}(v)\big)_{ij} = g^{ab} \tilde{\nabla}_{a}\tilde{\nabla}_{b}v_{ij} + \big(Q_{1}(g_{0},u,\tilde{\nabla}u) \cdot \tilde{\nabla} v\big)_{ij} +\big(Q_{2}(g_{0},u,\tilde{\nabla} u) \cdot v\big)_{ij},$$
avec $Q_{1} \in C_{\mathcal{V}}^{\infty}([0, T) \times\overset{\circ}{M};T^{(2,3)}\overset{\circ}{M})$ et $Q_{2} \in C_{\mathcal{V}}^{\infty}([0, T) \times\overset{\circ}{M};T^{(2,2)}\overset{\circ}{M}).$ 
D'autre part, on a que $$f\in \rho^{\delta} C_{\mathcal{V}}^{\infty}(\overset{\circ}{M};T^{(2,0)}\overset{\circ}{M}),$$ puisque que $g_{0}$ est une m\'etrique asymptotiquement Einstein. Le r\'esultat d\'ecoule alors du Corollaire~\ref{corr}.
\end{proof}
\begin{them} \label{theorii}
Soit $(\overset{\circ}{M}, g_{0})$ une vari\'et\'e riemannienne de dimension $n$ avec une structure de Lie fibr\'ee \`a l'infini $(M,\mathcal{V}_{SF})$ telle que le rayon d'injectivit\'e est strictement positif. Si la m\'etrique $g_{0}$ est asymptotiquement Einstein polyhomog\`{e}ne par rapport \`a une famille indicielle positive $\mathcal{G}$ de $M$ et si $g(t)$ pour $t \in [0,T)$ est la solution du flot de Ricci\texttt{-}DeTurck, alors 
$$g \in  \mathcal{A}^{\mathcal{G}_{\infty}}_{\phg}([0,T) \times M; {}^{\mathcal {V}_{SF}}TM^{(2,0)}).$$
\end{them}
\begin{proof}
D'abord, on consid\`{e}re $u=g - g_{0}$ au lieu de $g$ pour se mettre dans le cadre du Th\'eor\`{e}me~\ref{tlem}. En vertu de ce dernier, on remarque que l'existence de laps de temps $\varepsilon$ o\`u la m\'etrique est polyhomog\`{e}ne d\'epend uniquement du temps d'existence des \'equations au bord \`a l'ordre $0.$ Puisque $g_{0}$ est asymptotiquement Einstein, ces \'equations ont des solutions triviales tant que le flot existe. En effet, on sait par le premier pas de la preuve par induction du Th\'eor\`{e}me~\ref{tlem} que $u_{0}$ est le terme d'ordre $0$ de la polyhomog\'en\'eit\'e de la solution $u$ pour un certain $\varepsilon>0.$  Puisque la m\'etrique est asymptotiquement Einstein, la Proposition~\ref{propii} nous permet de prendre $u_{0}=0,$ ainsi par~\eqref{finich}, $z_{1}$ sera le plus petit r\'eel positif tel que
$$u \in \rho_{i}^{\vartheta} C_{\mathcal {V}_{SF}}^{\infty}([0, T) \times\overset{\circ}{M};T^{(2,0)}\overset{\circ}{M}), \forall \vartheta < z_{1}.$$
On peut donc prendre $\varepsilon= T$ dans le Th\'eor\`{e}me~\ref{tlem}.
\end{proof} 
Les deux Th\'eor\`{e}me~\ref{theore} et Th\'eor\`{e}me~\ref{theorii} concernent la polyhomog\'en\'eit\'e de solution au flot de Ricci\texttt{-}DeTurck lorsque la m\'etrique initiale $g_{0}$ a \'et\'e suppos\'ee polyhomog\`{e}ne. 
En particulier, lorsqu'elle est dans $C^{\infty}(M; {}^{\mathcal {V}_{SF}}TM^{(2,0)})$, on peut facilement d\'eduire dans le corollaire suivant qu'il en sera de m\^eme pour les solutions du flot de Ricci normalis\'e et du flot de Ricci\texttt{-}DeTurck. 
\begin{cor}
Soit $(\overset{\circ}{M}, g_{0})$ une vari\'et\'e riemannienne de dimension $n$ avec une structure de Lie fibr\'ee \`a l'infini $(M,\mathcal{V}_{SF}),$ telle que le rayon d'injectivit\'e est positif. 
Si $g_{0} \in C^{\infty}(M; {}^{\mathcal {V}_{SF}}TM^{(2,0)})$ et si $\bar{g}(t)$ pour $t \in [0,T)$ est la solution du flot de Ricci alors $\exists \varepsilon \in (0,T)$ tel que
$$\bar{g} \in  C^{\infty}([0, \varepsilon) \times M; {}^{\mathcal {V}_{SF}}TM^{(2,0)}).$$
En particulier, si $g_{0}$ est asymptotiquement Einstein, alors la solution du flot de Ricci normalis\'e~\eqref{normal} $\bar{g}$ est dans   $C^{\infty}([0,T) \times M; {}^{\mathcal {V}_{SF}}TM^{(2,0)})$. 
\end{cor}
\begin{proof}
D'abord, on remplace $\mathcal{G}_{\infty}(H)$ par $\mathbb{N}_{0} \times \{0\} \;\textrm{pour}\; H \in  \mathcal{M}_{1}(M)$ dans le Th\'eor\`{e}me~\ref{theore} et le Th\'eor\`{e}me~\ref{theorii}. On sait alors que la solution au flot de Ricci\texttt{-}DeTurck est dans $C^{\infty}(M; {}^{\mathcal {V}_{SF}}TM^{(2,0)})$ $\forall t <\varepsilon$ pour un certain $\varepsilon>0$. 
Par cons\'equent, $W(t)= W^{j}(t)\frac{\partial}{\partial x^{j}}$ est un $\mathcal {V}$\texttt{-}champ de vecteurs lisses jusqu'au bord, donc $W(t)  \in \mathcal {V}_{SF} \subset \mathcal {V}_{b}.$ Ainsi la famille de diff\'eomorphismes $t \rightarrow \alpha_{t}$ avec $\alpha (0) = Id$ engendr\'ee par $W$ est en fait des diff\'eomorphismes de la vari\'et\'e \`a coins $M.$ Et par suite, on peut voir que la solution $\bar{g}=\alpha^{*}g(t)$ est une solution du flot de Ricci normalis\'e, voir Th\'eor\`{e}me 3.13~\cite{chow2004ricci}. Par cons\'equent, on obtient que $\bar{g} \in  C^{\infty}([0, \varepsilon) \times M; {}^{\mathcal {V}_{SF}}TM^{(2,0)}),$ pour un certain  $ \varepsilon \in (0,T).$ Si $g_{0}$ est asymptotiquement Einstein, alors on peut prendre $\varepsilon= T$.
\end{proof}

\nocite{*}
\bibliographystyle{plain}
\hspace{0.5cm}

\small{D\'epartement de Math\'ematiques, Universit\'e  du Qu\'ebec \`a Montr\'eal, 201 Avenue du Pr\'esident-Kennedy, Montr\'eal, QC H2X 3Y7.} 

 \small{\textit{Courriel}\! \!\!\!: ammar.mahdi@courrier.uqam.ca}

\end{document}